\documentclass{amsart}
\usepackage{amscd,amssymb, amsfonts}
\usepackage{amscd}

\DeclareFontFamily{OMS}{rsfs}{\skewchar\font'60}
\DeclareFontShape{OMS}{rsfs}{m}{n}{<-5>rsfs5 <5-7>rsfs7 <7->rsfs10 }{}
\DeclareSymbolFont{rsfs}{OMS}{rsfs}{m}{n}
\DeclareSymbolFontAlphabet{\scr}{rsfs}

\newcommand{\sP}{\scr P}
\newcommand{\sS}{\scr S}

\let\catsymbfont\mathcal 
\newcommand{\aA}{{\catsymbfont{A}}}

\newcommand{\aC}{{\catsymbfont{C}}}

\newcommand{\aF}{{\catsymbfont{F}}}
\newcommand{\aM}{{\catsymbfont{M}}}

\newcommand{\aS}{{\catsymbfont{S}}}
\newcommand{\aT}{{\catsymbfont{T}}}
\newcommand{\aU}{{\catsymbfont{U}}}

\let\iso\cong
\let\sma\wedge
\renewcommand{\to}{\mathchoice{\longrightarrow}{\rightarrow}{\rightarrow}{\rightarrow}}

\newcommand{\Th}{\operatorname{TH}}

\newcommand{\bA}{{\mathbb{A}}}
\newcommand{\bB}{{\mathbb{B}}}
\newcommand{\bC}{{\mathbb{C}}}

\newcommand{\bL}{{\mathbb{L}}}
\newcommand{\bP}{{\mathbb{P}}}
\newcommand{\bR}{{\mathbb{R}}}
\newcommand{\bS}{{\mathbb{S}}}
\newcommand{\bT}{{\mathbb{T}}}
\newcommand{\bU}{{\mathbb{U}}}

\newcommand{\thp}{\ltimes}   
\newcommand{\htp}{\simeq}
\newcommand{\Sma}{\bigwedge}

\newcommand{\Z}{{\mathbb{Z}}}

\newcommand{\Sp}{\textit{Sp}}
\newcommand{\Spin}{\textit{Spin}}
\newcommand{\co}{\colon\thinspace}
\newcommand{\cy}{\text{cy}}
\newcommand{\xr}{\xrightarrow}
\newcommand{\xl}{\xleftarrow}
\newcommand{\sL}{\mathcal{L}}
\newcommand{\I}{\mathcal I}
\newcommand{\U}{\aU}

\newcommand{\lprod}{\boxtimes_{\sL}}
\newcommand{\lmap}{F_{\boxtimes_{\sL}}}
\newcommand{\sfree}[1]{* \boxtimes (\sL(1) \times #1)}

\def\quickop#1{\expandafter\DeclareMathOperator\csname
#1\endcsname{#1}}
\quickop{id}\quickop{Id}\quickop{op}
\quickop{Ar}\quickop{T}\quickop{Map}\quickop{Hom}
\newcommand{\hocolim}{\operatornamewithlimits{hocolim}}
\newcommand{\colim}{\operatornamewithlimits{colim}}

\renewcommand{\T}{T}
\newcommand{\Br}{\mathit{Br}}
\newcommand{\GL}{\mathit{GL}}
\newcommand{\Top}{\mathit{Top}}

\newtheorem{maintheorem}{Theorem}
\newtheorem{theorem}{Theorem}[section]
\newtheorem{corollary}[theorem]{Corollary}
\newtheorem{lemma}[theorem]{Lemma}
\newtheorem{proposition}[theorem]{Proposition}
\theoremstyle{definition}
\newtheorem{definition}[theorem]{Definition}
\theoremstyle{remark}

\newtheorem*{remark}{Remark}

\numberwithin{equation}{section}

\newdimen\wordsquish\wordsquish=1ex
\newdimen\redsquish\redsquish=0ex

\input xy
\xyoption{all}

\begin{document}

\title[Topological Hochschild homology of Thom spectra]
{Topological Hochschild homology of Thom spectra and the free loop space}
\author{A. J. Blumberg}
\address{Department of Mathematics\\Stanford University\\Stanford, CA 94305}
\email{blumberg@math.stanford.edu}
\author{R. L. Cohen}
 \address{Department of Mathematics \\Stanford University\\Stanford, CA 94305}
\email{ralph@math.stanford.edu}
\author{C. Schlichtkrull}
\address{Department of Mathematics, University of Bergen, Johannes
  Brunsgate 12, 5008 Bergen, Norway}
\email{krull@math.uib.no}
\thanks{The first author was partially supported by an NSF postdoctoral fellowship}\date{\today}
\thanks{The second author was partially supported by a grant from the
  NSF}

\begin{abstract}
We describe the topological Hochschild homology of ring spectra that arise as 
Thom spectra for loop maps $f\co X \rightarrow BF$, where $BF$ denotes the 
classifying space for stable spherical fibrations. To do this, we consider symmetric 
monoidal models of the category of spaces over $BF$ and corresponding strong
symmetric monoidal Thom spectrum functors. Our main result identifies the topological Hochschild homology as the Thom spectrum of a certain
stable bundle over the free loop space $L(BX)$. This leads to explicit calculations of the topological Hochschild homology for a large class of ring spectra, including all of the classical cobordism spectra $MO$, $MSO$, $MU$, etc., and the Eilenberg-Mac Lane spectra 
$H\mathbb Z/p$ and $H\mathbb Z$.
\end{abstract}

\maketitle

\section{Introduction}
Many interesting ring spectra arise naturally as Thom spectra.  It is
well-known that one may associate a Thom spectrum $\T(f)$ to any map
$f\co X\to BF$, where $BF$ denotes the classifying space for stable
spherical fibrations. This construction is homotopy invariant in the sense that the stable homotopy type of $\T(f)$ only depends on the homotopy class of $f$, see
\cite{lewis-may-steinberger}, \cite{mahowald}. Furthermore, if $f$ is a loop map, then it follows from a result of Lewis \cite{lewis-may-steinberger} that $\T(f)$ is an $A_{\infty}$ ring spectrum. In this case the topological Hochschild homology spectrum $\Th(T(f))$ is defined.  For
example, all of the Thom spectra $MG$ representing the classical
cobordism theories (where $G$ denotes one of the stabilized Lie groups
$O$, $SO$, $\Spin$, $U$, or $\Sp$) arise from canonical infinite loop maps $BG\to BF$.
In this paper we provide an explicit description of the
topological Hochschild homology of such a ring spectrum $T(f)$ in terms of the
Thom spectrum of a certain stable bundle over the free loop space.  In
order to state our main result we begin by recalling some elementary
results about the free loop space $L(B)$ of a connected space $B$.
Fixing a base point in $B$, we have the usual fibration sequence
\[
\Omega(B)\to L(B)\to B 
\]  
obtained by evaluating a loop at the base point of $S^1$.  This
sequence is split by the inclusion $B \to L(B)$ as the constant
loops.  When $B$ has the structure of a homotopy associative and
commutative H-space, $L(B)$ also has such a structure and the
composition  
\begin{equation}\label{freeloopdecomposition}
\Omega(B)\times B\to L(B)\times L(B) \to L(B)
\end{equation}
is an equivalence of H-spaces.  If we assume that $B$ has the homotopy
type of a CW complex, then the same holds for $L(B)$ and inverting the
above equivalence specifies a well-defined homotopy class
$L(B)\stackrel{\sim}{\to}\Omega(B)\times B$.  
Applying this to the delooping $B^2F$ of the infinite loop space $BF$, 
we obtain a splitting
\[
L(B^2F)\simeq\Omega(B^2F)\times B^2F\simeq
BF\times B^2F. 
\]
Let $\eta:S^3\to S^2$ denote the unstable Hopf map and also in mild
abuse of notation the map obtained by precomposing as follows,
\[
\eta:B^2F\simeq \Map_*(S^2,B^4F)\stackrel{\eta^*}{\to}
\Map_*(S^3,B^4F)\simeq BF.
\]

The following result is the main theorem of the paper.

\begin{maintheorem}\label{maintheorem}
Let $f\co X\to BF$ be the loop map associated to a map of connected based 
spaces $Bf\co BX\to B^2F$. Then there is a natural stable equivalence
\[
\Th(\T(f))\simeq \T(L^{\eta}(Bf)),
\]
where $L^{\eta}(Bf)$ denotes the composite
\[
L^{\eta}(Bf)\co L(BX)\stackrel{L(Bf)}{\to}L(B^2F)\simeq BF\times
B^2F\stackrel{\text{id}\times\eta}{\to} BF\times 
BF\to BF.
\] 
Here the last arrow represents multiplication in the H-space $BF$.
\end{maintheorem}

When $f$ is the constant map, $\T(f)$ is equivalent to the spherical
group ring $\Sigma^{\infty}\Omega(BX)_+$, where $+$ indicates a
disjoint base point.  In this case we recover the stable equivalence
of B\"okstedt and Waldhausen, 
\[
\Th(\Sigma^{\infty}\Omega(BX)_+)\simeq \Sigma^{\infty}L(BX)_+.
\]

The real force of Theorem~\ref{maintheorem} comes from the fact that
the Thom spectrum $T(L^{\eta}Bf)$ can be analyzed effectively in many
cases.  We will say that $f$ is an $n$-fold loop map if there exists
an $(n-1)$-connected based space $B^nX$ and a homotopy commutative
diagram of the form 
\[
\begin{CD}
\Omega^n (B^nX)@>\Omega^n(B^nf)>> \Omega^n(B^{n+1}F)\\
@AA\sim A @AA \sim A\\
X@>f>> BF,
\end{CD}
\]
where the vertical maps are equivalences as indicated. When $X$ is a 2-fold loop space, the product decomposition in (\ref{freeloopdecomposition}) can be applied to $L(BX)$. We shall then prove the following.

\begin{maintheorem}\label{2foldtheorem}
If $f$ is a 2-fold loop map, then there is a stable equivalence
$$
\Th(\T(f))\simeq \T(f)\wedge \T(\eta\circ Bf),
$$
where $T(\eta\circ Bf)$ denotes the Thom spectrum of
$BX\stackrel{Bf}{\to} B^2F\stackrel{\eta}{\to}BF$.
\end{maintheorem}

For 3-fold loop maps we can describe $\Th(T(f))$ without reference to $\eta$.
 
\begin{maintheorem}\label{3foldtheorem}
If $f$ is a 3-fold loop map, then there is a stable equivalence
$$
\Th(\T(f))\simeq \T(f)\wedge BX_+.
$$
\end{maintheorem} 

If $f$ is an infinite loop map one can realize the stable equivalence in Theorem \ref{3foldtheorem} as an equivalence of $E_{\infty}$ ring spectra. This is carried out in 
\cite{blumberg-THH-thom-einf} and \cite{schlichtkrull-higher}, working respectively with the $S$-module approach \cite{ekmm} and the symmetric spectrum approach \cite{hovey-shipley-smith} \cite{mmss} to structured ring spectra. The operadically sophisticated reader will note that the term ``n-fold loop maps'' is a device-independent way of describing maps that are structured by $E_n$ operads (that is, operads equivalent to the little $n$-cubes operad).  We have chosen this elementary description since this is the kind of input data one encounters most often in the applications and since many of the examples in the literature (such as the ones in \cite{mahowald}) are formulated in this language. 
For the technical part of our work it will be important to pass back and forth between loop maps and maps structured by operads and we explain how to do this in Appendix \ref{loopappendix}. It is known by \cite{Brun-Fiedorowicz-Vogt} and \cite{Mandell} that if $T$ is an $E_n$ ring spectrum, then $\Th(T)$ is an $E_{n-1}$ ring spectrum. We expect the following strengthening of Theorem \ref{3foldtheorem} to hold: if $X$ is a grouplike $E_n$ space and $f\co X\to BF$ an $E_n$ map, then the equivalence in Theorem \ref{3foldtheorem} is an equivalence of $E_{n-1}$ ring spectra. 

\subsection*{The classical cobordism spectra}
Let $G$ be one of the stabilized Lie groups $O$, $SO$, $\Spin$, $U$, or $\Sp$.  Then the Thom spectrum $MG$ arises from an infinite loop map $BG\to BF$ and so Theorem
\ref{3foldtheorem} applies to give a stable equivalence: 
\[
\Th(MG)\simeq MG\wedge BBG_+.
\]
Spelling this out using the Bott periodicity theorem, we get the
following corollary. 
\begin{corollary}
There are stable equivalences of spectra
\begin{align*}
&\Th(MO)\simeq MO\wedge (U/O)\langle 1\rangle_+\\
&\Th(MSO)\simeq MSO\wedge (U/O)\langle2\rangle_+\\
&\Th(MSpin)\simeq M\Spin\wedge(U/O)\langle3\rangle_+\\
&\Th(MU)\simeq MU\wedge SU_+\\
&\Th(M\Sp)\simeq M\Sp\wedge(U/\Sp)\langle1\rangle_+,
\end{align*}
where here $(U/O)\langle n\rangle$ and $(U/\Sp)\langle n\rangle$
denote the $n$-connected covers of $U/O$ and $U/\Sp$ respectively.
\end{corollary}

These results also admit a cobordism interpretation.

\begin{corollary}\label{bordisminterpretation}
Let $G$ be one of the stabilized Lie groups considered above, and let
$\Omega^G_*$ denote the corresponding $G$-bordism theory. Then there
is an isomorphism 
\[
\pi_*\Th(MG)\simeq\Omega_*^G(BBG).
\]
\end{corollary}

There are many other examples of cobordism spectra for which Theorem
\ref{3foldtheorem} applies, see \cite{stong}. In the case of the
identity map $BF\to BF$ we get the spectrum $MF$, and we again have a stable
equivalence 
\[
\Th(MF)\simeq MF\wedge BBF_+.
\]

\subsection*{The Eilenberg-Mac Lane spectra $H\mathbb Z/p$ and $H\mathbb Z$}
Another application of our results is to the calculation of the topological Hochschild homology of the Eilenberg-Mac Lane spectra $H\mathbb Z/p$ and 
$H\mathbb Z$. These calculations are originally due to B\"okstedt \cite{bokstedt2} using a very different approach. Our starting point here is the fact that these spectra can be realized as Thom spectra. For $H\mathbb Z/2$ this is a theorem of Mahowald \cite{mahowald}; if 
$f\co\Omega^2(S^3)\to BF$ is an extension of the generator of $\pi_1BF$ to a 2-fold loop map, then $T(f)$ is equivalent to $H\mathbb Z/2$, see also \cite{cohen-may-taylor}. In general, given a connected space $X$ and a map $f\co X\to BF$, the associated Thom spectrum has $\pi_0T(f)$ equal to $\mathbb Z$ or $\mathbb Z/2$, depending on whether $T(f)$ is oriented or not. Hence 
$H\mathbb Z/p$ cannot be realized as a Thom spectrum for a map to $BF$ when $p$ is odd. However, by an observation due to Hopkins \cite{thomified}, if one instead  consider the classifying space 
$BF_{(p)}$ for $p$-local spherical fibrations, then $H\mathbb Z/p$ may be realized as the $p$-local Thom spectrum associated to a certain 2-fold loop map $f\co\Omega^2(S^3)\to BF_{(p)}$. We recall the definition of $BF_{(p)}$ and the $p$-local approach to Thom spectra in Section \ref{p-localsection}. Our methods work equally well in the $p$-local setting and we shall prove that the $p$-local version of Theorem \ref{2foldtheorem} applies to give the following result.    

\begin{theorem}\label{Z/p-theorem}
There is a stable equivalence 
$$
\Th(H\mathbb Z/p)\simeq H\mathbb Z/p\wedge \Omega(S^3)_+
$$
for each prime $p$.
\end{theorem}
On the level of homotopy groups this implies that
$$
\pi_*\Th(H\mathbb Z/p)= H_*(\Omega(S^3),\mathbb Z/p)= \mathbb Z/p[x], 
$$
where $x$ has degree 2, see e.g., \cite{whitehead}, Theorem 1.18. In the case of the Eilenberg-Mac Lane spectrum $H\mathbb Z_{(p)}$ for the $p$-local integers we instead considers the 2-fold loop space $\Omega^2(S^3\langle 3\rangle)$ where $S^3\langle 3\rangle$ is the 3-connected cover of $S^3$. Arguing as in \cite{cohen-may-taylor} it follows that the $p$-local Thom spectrum of the 2-fold loop map
$$
\Omega^2(S^3\langle 3\rangle)\to\Omega^2(S^3)\to BF_{(p)}
$$
is equivalent to $H\mathbb Z_{(p)}$. Using this we show that the $p$-local version of Theorem \ref{2foldtheorem} applies to give a stable equivalence
$$
\Th(H\mathbb Z_{(p)})\simeq H\mathbb Z_{(p)}\wedge \Omega(S^3\langle 3\rangle)_+
$$
for each prime $p$. Since topological Hochschild homology commutes with localization this has the following consequence for the integral Eilenberg-Mac Lane spectrum. 
\begin{theorem}\label{Z-theorem}
There is a stable equivalence
$$
\Th(H\mathbb Z)\simeq H\mathbb Z\wedge \Omega(S^3\langle 3\rangle)_+
$$
where $S^3\langle 3\rangle$ denotes the 3-connected cover of $S^3$.
\end{theorem}
This gives that
$$
\pi_i\Th(H\mathbb Z)=H_i(\Omega(S^3\langle 3\rangle),\mathbb Z)=
\begin{cases}
\mathbb Z,&i=0,\\
\mathbb Z/(\frac{i+1}{2}),&\text{for $i>0$ odd},\\
0,&\text{for $i>0$ even}.
\end{cases}
$$
The last isomorphism is easily obtained by applying the Serre spectral sequence to the fibration sequence
$$
S^1\to \Omega(S^3\langle 3\rangle)\to \Omega(S^3).
$$
An alternative approach to the calculation of $\Th(H\mathbb Z)$ has been developed by the first author in \cite{blumberg-THH-thom-einf}. One may also ask if a similar approach can be used for the Eilenberg-Mac Lane spectra $H\mathbb Z/p^{n}$ in general. This turns out not to be the case since these spectra cannot be realized as $A_{\infty}$ Thom spectra for $n>1$, see 
\cite{blumberg-THH-thom-einf} for details. 

\subsection*{Thom spectra arising from systems of groups}
A more geometric starting point for the construction of Thom spectra is to consider systems of groups $G_n$ equipped with compatible homomorphisms to the orthogonal groups $O(n)$. We write $MG$ for the associated Thom spectrum whose $n$th space is the Thom space of the vector bundle represented by $BG_n\to BO(n)$. If the groups $G_n$ come equipped with suitable associative pairings $G_m\times G_n\to G_{m+n}$, then $MG$ inherits a multiplicative structure. 
For example, we have the commutative symmetric ring spectra $M\Sigma$ and 
$M\GL(\mathbb Z)$ associated to the symmetric groups $\Sigma_n$ and the general linear groups $\GL_n(\mathbb Z)$. In Section \ref{groupthom} we show how to deduce the following results from Theorem \ref{3foldtheorem}. 

\begin{theorem}\label{Msigma}
There is a stable equivalence
$$
\Th(M\Sigma)\simeq M\Sigma\wedge \tilde Q(S^1)_+
$$  
where $\tilde Q(S^1)$ denotes the homotopy fiber of the canonical map $Q(S^1)\to S^1$.
\end{theorem}
Here $Q(S^1)$ denotes the infinite loop associated to the suspension spectrum 
$\Sigma^{\infty}(S^1)$. 
\begin{theorem}\label{MGL}
There is a stable equivalence
$$
\Th(M\GL(\mathbb Z))\simeq M\GL(\mathbb Z)\wedge B(B\GL(\mathbb Z)^+)_+ 
$$
where $B(B\GL(\mathbb Z)^+)$ denotes the first space in the 0-connected cover of the algebraic K-theory spectrum for $\mathbb Z$. 
\end{theorem}

There is a plethora of examples of this kind, involving for instance braid groups, automorphism groups of free groups, and general linear groups. 

We finally mention an application of our results in connection with the analysis of quasi-symmetric functions. Let $\mathbb CP^{\infty}\to BU$ be the canonical map obtained by identifying $\mathbb CP^{\infty}$ with $BU(1)$ and let 
$\xi\co \Omega\Sigma(\mathbb C P^{\infty})\to BU$ be the extension to a loop map. The point of view in \cite{baker-richter} is that $\Omega\Sigma(\mathbb C P^{\infty})$ is a topological model for the ring of quasi-symmetric functions in the sense that the latter may be identified with the integral cohomology ring $H^*(\Omega\Sigma(\mathbb C P^{\infty}))$. On cohomology $\xi$ then corresponds to the inclusion of the symmetric functions in the ring of quasi-symmetric functions. 
Based on Theorem \ref{maintheorem}, the authors deduce in \cite{baker-richter} that 
$\Th(T(\xi))$ may be identified with the Thom spectrum of the map
$$
L(\Sigma\mathbb C P^{\infty})\to L(BBU)\simeq BU\times BBU \xr{\mathrm{proj}}BU.
$$   
This uses that $\eta\co BBU\to BU$ is null homotopic. In particular, the spectrum homology of $\Th(T(\xi))$ is isomorphic to the homology of $L(\Sigma \mathbb C P^{\infty})$

\subsection*{The strategy for analyzing $\Th(T(f))$}
We here begin to explain the ideas and constructions going into the
proof of the main results. Let $(\mathcal A,\boxtimes, 1_\mathcal A)$
be a symmetric monoidal category.  Recall that if $A$ is a monoid in
$\mathcal A$, then the cyclic bar construction is the cyclic object 
$B^{\cy}_{\bullet}(A)\co [k]\mapsto A^{(k+1)\boxtimes}$, 
with face operators
\[
d_i(a_0\boxtimes\dots\boxtimes a_k)=
\begin{cases}
a_0\boxtimes\dots\boxtimes a_ia_{i+1}\boxtimes \dots\boxtimes
a_k,&i=0,\dots k-1,\\ 
a_ka_0\boxtimes\dots\boxtimes a_{k-1},&i=k,
\end{cases}
\]
degeneracy operators
\[
s_i(a_0\boxtimes\dots\boxtimes a_k)=a_0\boxtimes\dots a_{i}\boxtimes
1_{\mathcal C}\boxtimes a_{i+1}\dots\boxtimes a_k,\quad i=0,\dots,k, 
\]
and cyclic operators
\[
t_i(a_0\boxtimes\dots\boxtimes a_k)=a_k\boxtimes
a_0\boxtimes\dots\boxtimes a_{k-1}.
\]
Here the notation is supposed to be self-explanatory.  We denote the
geometric realization of this object (when this notion makes sense) as
$B^{\cy}(A)$.
When $\mathcal A$ is one of the modern symmetric monoidal categories of spectra and $T$ is a ring spectrum (i.e.\ a monoid in $\mathcal A$), then $B^{\cy}(T)$ is a model of the topological Hochschild homology provided that $T$ satisfies suitable cofibrancy conditions, see 
\cite{ekmm} and \cite{shipleyTHH}. 

Suppose now that $f\co A\to BF$ is a loop map that has been rectified to a map of topological monoids and imagine temporarily that $BF$ could be realized as a commutative topological
monoid.  Then associated to $f$ we would have the simplicial map
\[
B^{\cy}_{\bullet}(f)\co B^{\cy}_{\bullet}(A)\to
B^{\cy}_{\bullet}(BF)\to BF,
\]
where $BF$ is viewed as a constant simplicial space and the last map is given by levelwise multiplication. The intuitive picture underlying our results is the notion that the Thom spectrum functor should take the cyclic bar construction in spaces to the cyclic bar construction in spectra in the sense that $T(B^{\cy}(f))$ should be stably equivalent to 
$B^{\cy}(T(f))$. Ignoring issues of cofibrancy, this exactly gives a description of $\Th(T(f))$ in terms of a Thom spectrum.
This picture makes contact with the free loop space in the following fashion: When $A$ is a topological monoid $B^{\cy}(A)$ inherits an action of the circle group $\bT$ from the
cyclic structure.  This gives rise to the composite map
\[
\bT\times B^{\cy}(A)\to B^{\cy}(A)\to B(A),
\]
where $B(A)$ is the realization of the usual bar construction and the
second map is the realization of the simplicial map  
\[
(a_0,\dots,a_k)\mapsto (a_1,\dots,a_k).
\]
The adjoint map $B^{\cy}(A)\to L(B(A))$ fits in a commutative diagram
of spaces 
\begin{equation}\label{BcyLdiagram}
\xymatrix{
A \ar[r] \ar[d] & B^{\cy}(A) \ar[r] \ar[d] & B(A) \ar@{=}[d] \\
\Omega(B(A)) \ar[r] & L(B(A)) \ar[r] & B(A),\\
}
\end{equation}
where the vertical map on the left is the usual group-completion. 
Standard results on realizations of simplicial quasifibrations imply
that if $A$ is grouplike ($\pi_0A$ is a group) and well-based 
(the inclusion of the unit is a Hurewicz cofibration), then the upper sequence is a fibration sequence up to homotopy and the vertical maps are weak homotopy equivalences.  This
suggests that we should be able to connect the description in terms of
the cyclic bar construction to the free loop space. Furthermore, if $A$ is an infinite loop space, then there is a canonical splitting $B^{\cy}(A)\to A$ of the upper sequence and it is proved in 
\cite{schlichtkrull-units} that the composition 
$$
B(A)\to L(B(A))\xl{\sim}B^{\cy}(A)\to A
$$
represents multiplication by the Hopf map $\eta$. Applied to $BF$ this is the essential reason why $\eta$ appears in the statement of our main theorem.  

However, there are formidable technical impediments to making this
intuitive picture precise.  For one thing, $BF$ cannot be realized as a
commutative topological monoid since it is not a generalized Eilenberg-Mac Lane space. Moreover, the classical comparison of the Thom spectrum of a cartesian product to
the smash product of the Thom spectra is insufficiently rigid; one
obtains a simplicial map relating $B_{\bullet}^{\cy}(T(f))$ and $T(B_{\bullet}^{\cy}(f))$ 
only after passage to the stable homotopy category. This is not
sufficient for computing the topological Hochschild homology.

A major part of this paper is concerned with developing suitable technical
foundations to carry out the program above.  Our approach is as
follows: We introduce a symmetric monoidal category $(\mathcal{A},
\boxtimes, 1_{\mathcal{A}})$ which is a refined model of the category of spaces 
in the sense that $E_{\infty}$ objects can be realized as strictly commutative 
$\boxtimes$-monoids. In particular,  $BF$ will admit a model as a
commutative $\boxtimes$-monoid $BF_{\mathcal{A}}$.  In this setting we show that
the Thom spectrum functor can be refined to a strong symmetric
monoidal functor $\T_{\mathcal A}\co\mathcal A/BF_{\mathcal A}\to \mathcal S$,    
where $\mathcal S$ is a suitable symmetric monoidal category of spectra. 
Here $\mathcal A/BF_{\mathcal A}$ denotes the category of objects in
$\mathcal A$ over $BF_{\mathcal A}$ with the symmetric monoidal structure inherited
from $\mathcal{A}$: given two objects $f\co X\to BF_{\mathcal A}$ and 
$g\co Y\to BF_{\mathcal A}$, the monoidal product is defined by 
\[
f\boxtimes g\co X\boxtimes Y\to BF_{\mathcal A}\boxtimes BF_{\mathcal
  A}\to BF_{\mathcal A}.
\]
This is symmetric monoidal precisely because $BF_{\mathcal A}$ is
commutative. That $\T_{\mathcal A}$ is strong monoidal means that there is
a natural isomorphism 
\[
\T_{\mathcal A}(f)\sma \T_{\mathcal A}(g)\xr{\cong}\T_{\mathcal A}(f\boxtimes g)
\]
and this implies that we can directly implement the intuitive strategy discussed
above.  Of course, there is significant technical work necessary to
retain homotopical control over the quantities involved in the formula
above, but the basic approach does become as simple as indicated.

We construct two different realizations of the category
$\mathcal{A}$.  In a precise sense, our constructions herein are
space-level analogues of the constructions of the modern symmetric
monoidal categories of spectra.  Just as there is an operadic
approach to a symmetric monoidal category of spectra given by \cite{ekmm} and
a ``diagrammatic'' approach given by (for example) symmetric spectra 
\cite{hovey-shipley-smith} \cite{mmss}, we have operadic and diagrammatic 
approaches to producing $\mathcal{A}$. Since there are several good choices 
for the category $\mathcal A$, we
shall in fact give an axiomatic description of the properties needed
to prove Theorem~\ref{maintheorem}.  The point is that even though
these settings are in some sense equivalent, the natural input for 
the respective Thom spectrum functors is very different.  Working in 
an axiomatic setting gives a flexible framework for adapting the constructions 
to fit the input provided in particular cases. Both of our realizations are based in 
part on the earlier constructions of May-Quinn-Ray \cite{may-quinn-ray} and 
Lewis-May \cite{lewis-may-steinberger}.

\subsection*{$\mathcal L(1)$-spaces and $S$-modules}
Our first construction is intimately related to the $S$-modules of
\cite{ekmm}.  Let $\sL(n)$ denote the space of linear isometries
$\sL(U^n, U)$, for a fixed countably infinite-dimensional real inner
product space $U$.  The object $\sL(1)$ is a monoid under composition,
and we consider the category of $\sL(1)$-spaces.  Following the
approach of \cite{kriz-may}, we construct an ``operadic smash
product'' on this category of spaces defined as the coequalizer
\[X \boxtimes Y = \sL(2) \times_{\sL(1) \times \sL(1)} (X \times Y).\]
This product has the property that an $A_\infty$ space is a monoid and
an $E_\infty$ space is a commutative monoid. Therefore $BF$ is a 
commutative monoid with respect to the $\boxtimes$ product, and so we can 
adapt the Lewis-May Thom spectrum functor to construct a Thom
spectrum functor from a certain subcategory of $\sL(1)$-spaces over $BF$ to
$S$-modules which is strong symmetric monoidal.  The observation that
one could carry out the program of \cite{ekmm} in the setting of
spaces is due to Mike Mandell, and was worked out in the thesis of the
first author \cite{blumberg-thesis}.

\subsection*{{$\mathcal I$}-spaces and symmetric spectra}
\label{preI-spaces}
Our second construction is intimately related to the symmetric spectra
of \cite{hovey-shipley-smith} and \cite{mmss}.  Let $\mathcal I$ be the category with objects 
the finite sets $\mathbf n=\{1,\dots,n\}$ and morphisms the injective maps.  
The empty set $\mathbf 0$ is an initial object.  The usual concatenation 
$\mathbf m\sqcup \mathbf n$ of finite ordered sets makes this a symmetric
monoidal category with symmetric structure maps $\mathbf
m\sqcup\mathbf n\to \mathbf n\sqcup\mathbf m$ given by the obvious
shuffles. By definition, an \emph{${\mathcal I}$-space} is a
functor $X\co\mathcal I\to\mathcal U$ and we write
$\mathcal I\mathcal U$ for the category of such functors.  Just as for
the diagrammatic approach to the smash product of spectra, this category 
inherits a symmetric monoidal structure from $\I$ via the left Kan extension
\[
(X \boxtimes Y)(n)=\colim_{\mathbf n_1 \sqcup \mathbf n_2 \rightarrow \mathbf n}
X(n_1) \times Y(n_2)
\] 
along the concatenation functor $\sqcup\co \I^2\to \I$. The unit for the monoidal product is the constant $\I$-space $\I(\mathbf 0,-)$ which we denote by $*$. As we recall in Section \ref{Lewissection}, the correspondence $BF\co \mathbf n\mapsto BF(n)$ defines a commutative monoid in 
$\I\mathcal U$. We can therefore adapt the usual levelwise Thom space functor to construct a strong symmetric monoidal Thom spectrum functor from the category $\I\mathcal U/BF$ to the category of symmetric spectra. This point of view on Thom spectra has been worked out in detail by the third author \cite{schlichtkrull-thom}. A similar construction applies to give a strong symmetric monoidal Thom spectrum functor with values in orthogonal spectra. 

\subsection*{Acknowledgments}
The authors are grateful to a number of mathematicians for helpful conversations related to this work, including Gunnar Carlsson, Mike Hopkins, Wu-chung Hsiang, Ib Madsen, Mike Mandell, and Peter May.

\subsection*{Organization of the paper}
In  Section \ref{thomspectrumsection} we begin by reviewing the construction of the
Lewis-May Thom spectrum functor. We then set up an axiomatic framework which specifies 
the properties of a rigid Thom spectrum functor needed to prove Theorem~\ref{maintheorem}.  Following this, in Section~\ref{proofsection}, we show how our main
theorems can be deduced from these axioms. The rest of the paper is
devoted to implementations of the axiomatic framework. We collect the relevant
background material for the $S$-module approach in
Section~\ref{Lprelims} and verify the axioms in this setting in
Section~\ref{L-spacesection}.  In Section~\ref{secBF} we discuss
modifications needed to work in the context of universal
quasifibrations.  We then switch gears and consider the setting of
$\I$-spaces and symmetric spectra. In Section~\ref{symmetricpreliminaries} we collect and
formulate some background material on symmetric spectra and we verify
the axioms in this setting in Section~\ref{I-spacesection}. Finally, we discuss the technical details of the passage from loop maps to maps structured by operads in Appendix 
\ref{loopappendix}.

\section{Thom spectrum functors}\label{thomspectrumsection}
In this section we formalize the properties of a rigid Thom spectrum functor needed in order to prove our main theorems. We begin by reviewing the details of the underlying Thom spectrum functor, roughly following Lewis \cite[IX]{lewis-may-steinberger}. Our rigid Thom spectrum functors are built on this foundation and this construction provides the ``reference'' homotopy type we expect from a Thom spectrum functor.
\subsection{The Lewis-May Thom spectrum functor}\label{Lewissection}
We work in the categories $\mathcal U$ and $\mathcal T$ of based and unbased compactly generated weak Hausdorff spaces. By a spectrum $E$ we understand a sequence of based spaces $E_n$ for $n\geq 0$, equipped with a sequence of structure maps $S^1\wedge E_n\to E_{n+1}$ (this is what is called a prespectrum in \cite{lewis-may-steinberger}). We write 
$\Sp$ for the category of spectra in which a morphism is a sequence of maps that strictly commute with the structure maps.  Let $F(n)$ be the topological monoid of base point preserving homotopy equivalences of $S^n$. Following Lewis we use the notation $BF(n)$ for the usual bar construction $B(*,F(n),*)$ and $EF(n)$ for the one-sided bar construction $B(*,F(n),S^n)$.  The map $EF(n)\to BF(n)$ induced by the projection $S^n\to *$ is a quasifibration with fiber $S^n$ over each point in the base, and the inclusion of the basepoint in $S^n$ defines a section which is a Hurewicz cofibration. We refer the reader to  \cite{lewis-may-steinberger} and \cite{may-class} for more details of these constructions. The Thom space of a map $f\co X\to BF(n)$ is defined to be the quotient space 
$$
T(f)=f^*EF(n)/X,
$$   
where $f^*EF(n)$ is the pullback of $EF(n)$ and $X$ is viewed as a subspace via the induced section. Let $BF$ be the colimit of the spaces 
$BF(n)$ under the natural inclusions. Given a map $f\co X\to BF$, we define a filtration of $X$ by letting $X(n)$ be the closed subspace $f^{-1}BF(n)$. The Thom spectrum $T(f)$ then has as its $n$th space the Thom space $T(f_n)$, where $f_n$ denotes the restriction of $f$ to a map $X(n)\to BF(n)$. The structure maps are induced by the pullback diagrams
$$
\begin{CD}
S^1\bar\wedge f^*_nEF(n)@>>> f^*_{n+1}EF(n+1)\\\
@VVV @VVV \\
X(n)@>>> X(n+1),
\end{CD}
$$
where $\bar\wedge$ is the fibre-wise smash product. 
This construction defines our Thom spectrum functor
$$
T\co \mathcal U/BF\to \Sp.
$$
We next recall how to extend this construction to a functor with values in the category $\sS$ of coordinate-free spectra from \cite{lewis-may-steinberger}. Thus, we consider spectra indexed on the finite dimensional subspaces $V$ of an inner product space $U$ of countable infinite dimension. Let $F(V)$ be the topological monoid of base point preserving homotopy equivalences of the one-point compactification $S^V$. As in the case of $S^n$ there is an associated quasifibration $EF(V)\to BF(V)$ with fiber $S^V$. We again write $BF$ for the colimit of the spaces $BF(V)$ and observe that this is homeomorphic to the space $BF$ consider above. Given a map $f\co X\to BF$, let $X(V)$ be the closed subset $f^{-1}BF(V)$ and let $T(f)(V)$ be the Thom space of the induced map $f_V\co X(V)\to BF(V)$. This defines an object $T(f)$ in the category $\sP$ of coordinate-free prespectra from \cite{lewis-may-steinberger} and composition with the spectrification functor $L\co\sP\to \sS$ gives the Thom spectrum functor
$$
T_{\sS}\co \mathcal U/BF\to\sS
$$
from \cite{lewis-may-steinberger}. 
It is an important point of the present paper that these Thom spectrum functors can be rigidified to strong monoidal functors by suitable modifications of the domains and targets. This is based on the fact \cite{Boardman-Vogt}, \cite{may-quinn-ray}, that the correspondence 
$V\mapsto BF(V)$ defines a (lax) symmetric monoidal functor (see e.g.\ \cite{maclane}) from the category 
$\mathcal V$ of finite dimensional inner product spaces and linear isometries to the category 
$\mathcal U$ of spaces. The monoidal structure maps
$$
BF(V)\times BF(W)\to BF(V\oplus W)
$$ 
are induced by the monoid homomorphisms that map a pair of equivalences to their smash product. Here we implicitly use that the bar construction preserves products. There are now two ways in which this multiplicative structure leads to a ``representation'' of the space $BF$ as a strictly commutative monoid. On the one hand it follows from \cite[1.1.6]{may-quinn-ray} that 
$BF$ has an action of the linear isometries operad and we shall see in Section \ref{Lprelims} that this implies that it defines a commutative monoid in the weakly symmetric monoidal category of 
$\mathcal L(1)$-spaces. On the other hand, the category of functors from $\mathcal V$ to $\mathcal U$ is itself a symmetric monoidal category and the (lax) symmetric monoidal structure of the functor $V\mapsto BF(V)$ exactly corresponds to a commutative monoid structure in this functor category. Composing with the functor $\I\to\mathcal V$ that maps a finite set to the vector space it generates we get a ``representation'' of $BF$ as a commutative monoid in the category of $\I$-spaces. 

More generally, one may consider $\mathcal V$-diagrams of topological monoids 
$V\mapsto G(V)$ equipped with a unital, associative, and commutative natural pairing 
$$
G(V)\times G(W)\to G(V\oplus W).
$$ 
After applying the bar construction this gives a (lax) symmetric monoidal functor $V\mapsto BG(V)$.  For instance, we have the (lax) symmetric monoidal functors $V\mapsto BO(V)$ defined by the orthogonal groups and $V\mapsto B\Top(V)$ defined by the groups $\Top(V)$ of base point preserving homeomorphisms of $S^V$. We write $BG$ for the colimit of the spaces $BG(V)$ and as in the case of $BF$ this can be ``represented'' as a strictly commutative monoid both in an ``operadic'' and a ``diagrammatic'' fashion. Given a natural transformation of $\mathcal V$-diagrams of monoids $G(V)\to F(V)$, we define the Thom spectrum functor on $\U/BG$ to be the composition  
$$
T\co \U/BG\xr{}\U/BF\xr{T}\Sp,\text{\  respectively \ \ }
T_{\sS}\co \U/BG\xr{} \U/BF\xr{T_{\sS}} \sS.
$$ 

Finally, we must comment on homotopy invariance.  Due to the fact that quasifibrations and cofibrations are not in general preserved under pullback, the Thom spectrum functor is not a homotopy functor on the whole category $\mathcal U/BF$. A good remedy for this is to functorially replace an object $f$ by a Hurewicz fibration $\Gamma(f)\co \Gamma_f(X)\to BF$ in the usual way.
It then follows from \cite[IX, 4.9]{lewis-may-steinberger} that the composite functor 
\[
T\Gamma\co \mathcal U/BF\xr{\Gamma} \mathcal U/BF\xr{T}\Sp
\]    
is a homotopy functor in the sense that it takes weak homotopy equivalences over $BF$ to stable equivalences. Given a map $f\co X\to BF$, there is a natural homotopy equivalence $X\to \Gamma_f(X)$, which we may view as a natural transformation from the identity functor on $\mathcal U/BF$ to $\Gamma$. We think of $\T(f)$ as representing the ``correct'' homotopy type if the induced map $\T(f)\to \T(\Gamma(f))$ is a stable equivalence, and in this case we say that 
$f$ is \emph{$T$-good}. It follows from the above discussion that the restriction of $T$ to the full subcategory of $T$-good objects is a homotopy functor. Thus, for the statement in Theorem 
\ref{maintheorem} to be homotopically meaningful we have tacitly chosen $T$-good representatives before applying the Thom spectrum functor. 
We say that a $\mathcal V$-diagram of monoids $V\mapsto G(V)$ is \emph{group valued} if each of the monoids $G(V)$ is a group. In this case the homotopical analysis of the Thom spectrum functor simplifies since $\U/BG$ maps into the subcategory of $T$-good objects in $\U/BF$. The Thom spectrum functor on $\mathcal U/BG$ is therefore a homotopy functor if $G$ is group valued. 

In the following we describe in an axiomatic manner the properties required of a rigid Thom spectrum functor in order to prove Theorem \ref{maintheorem}. 
For simplicity we only formulate these axioms for maps over $BG$ when $G$ is group valued.  One may also formulate such axioms in the general case by repeated use of the functor 
$\Gamma$, but we feel that this added technicality would obscure the presentation. In the implementation of the axioms in Section \ref{L-spacesection} and Section 
\ref{I-spacesection} we discuss how to modify the constructions so as to obtain Theorem \ref{maintheorem} in general.

\subsection{Rigid Thom spectrum functors}\label{axiomssection}

Let $\mathcal S$ denote a symmetric monoidal category of ``spectra''. The reader should have in mind for instance the categories of $S$-modules \cite{ekmm} or topological symmetric spectra \cite{mmss}.  Formally, we require that $\mathcal S$
be a symmetric monoidal topological category and that there is a
continuous functor $U\co \mathcal S\to\Sp$, which we think of as a
forgetful functor.  A morphism in $\mathcal S$ is said to be a
\emph{weak equivalence} if the image under $U$ is a stable equivalence
of spectra.  We further require that $\mathcal S$ be cocomplete and
tensored over unbased spaces.  As we recall in Section
\ref{realizationsection}, this implies that there is an internal
notion of geometric realization for simplicial objects in $\mathcal
S$.  We also assume that the category of monoids in $\mathcal S$ comes
equipped with a full subcategory whose objects we call \emph{flat
  monoids}.  Given a flat monoid in $\mathcal S$, we define its
topological Hochschild homology to be the geometric realization of the
cyclic bar construction. (In the implementations, the flat objects are
sufficiently ``cofibrant'' for this to represent the correct homotopy
type. We recall that for an ordinary discrete ring, flatness is
sufficient for its Hochschild homology to be represented by the cyclic
bar construction).

We write  $\mathcal A$ for our refined category of spaces. Formally,
we require that $\mathcal A$ be a closed symmetric monoidal
topological category with monoidal product $\boxtimes$ and unit
$1_{\mathcal A}$, and we  assume that there is a continuous functor
$U\co\mathcal A\to \mathcal U$ which we again think of as a forgetful
functor. A morphism in $\mathcal A$ is said to be a \emph{weak
  equivalence} if the image under $U$ is a weak homotopy equivalence
of spaces. We also require that $\mathcal A$ be cocomplete and
tensored over unbased spaces and that $U$ preserves colimits and
tensors.  

The following list of axioms {\bf A1}--{\bf A6} specifies the properties we require of a rigid Thom spectrum functor. Here $BG$ denotes the colimit of a (lax) symmetric monoidal $\mathcal V$-diagram $V\mapsto BG(V)$ as discussed in Section \ref{Lewissection} and we assume that $G$ is group valued.
\begin{description}
\item[A1]
There exists a commutative monoid $BG_{\mathcal A}$ in $\mathcal A$
and a weak homotopy equivalence  
$
\zeta\co BG_{U}\xr{\sim} BG,
$
where $BG_{U}$ denotes the image of $BG_{\mathcal A}$ under the
functor $U$. We further assume that $BG_{\mathcal A}$ is augmented in
the sense that there is a map of monoids $BG_{\mathcal A}\to
1_{\mathcal A}$.
\end{description}

\begin{description}
\item[A2]
There exists a strong symmetric monoidal ``Thom spectrum'' functor
\[
\T_{\mathcal A}\co \mathcal A/BG_{\mathcal A}\to \mathcal S
\]
that preserves weak equivalences, and commutes with colimits and
tensors with unbased spaces.   We require that $\T_{\mathcal A}$ be a
lift of the Lewis-May Thom spectrum functor on $\mathcal U/BG$ in the sense that the two compositions in the diagram   
\[
\xymatrix{
\relax\mathcal A/BG_{\mathcal A} \ar[d]^U \ar[r]^{\T_{\mathcal A}} &
\ar[d]^U \mathcal S\\
\mathcal U/BG_U \ar[r]^{\T} & \Sp \\
}
\]
are related by a chain of natural stable equivalences. Here $\T$
denotes the composite functor 
\[
\mathcal U/BG_{U}\xr{\zeta_*}\mathcal U/BG\xr{\T} \Sp.
\]
\end{description}

These two axioms already guarantee that we can carry out the argument 
sketched in the introduction.  Let $\alpha \co A \rightarrow
BG_{\mathcal A}$ be a monoid morphism, and let 
$B^{\cy}(\alpha)$ be the realization of the simplicial map  
\[
B^{\cy}_{\bullet}(\alpha)\co B^{\cy}_{\bullet}(A)\to
B^{\cy}_{\bullet}(BG_{\mathcal A}) \to BG_{\mathcal A},
\]
where we view $BG_{\mathcal A}$ as a constant simplicial object.

\begin{theorem}\label{Thomrigid}
Let $\alpha \co A \to BG_{\mathcal A}$ be a monoid morphism  in
$\mathcal A$.  Then $\T_{\mathcal A}(\alpha)$ is a  monoid in $\mathcal S$ and there
is an isomorphism 
$$
B^{\cy}(\T_{\mathcal A}(\alpha)) \cong \T_{\aA} (B^{\cy}(\alpha)).
$$
Furthermore, there is a stable equivalence 
$$
U \T_{\mathcal A}(B^{\cy} (\alpha)) \xr{\sim} \T(UB^{\cy}(\alpha)).
$$
\end{theorem}

The simplicity of this result, once we have set up the framework of
the two axioms, is very satisfying. However, since we are really
interested in topological Hochschild homology, we must be able to
represent our Thom spectra as flat monoids in $\mathcal S$ and for
this reason we introduce the functor $C$ below. This should be thought
of as a kind of cofibrant replacement functor and for the application
of Theorem \ref{Thomrigid} it is essential that this ``replacement''
takes place in the category $\mathcal A$ before applying $T$. Adapting
the usual convention for topological monoids to our setting, we say
that a monoid in $\mathcal A$ is \emph{well-based} if the unit
$1_{\mathcal A}\to A$ has the homotopy extension property, see Section
\ref{realizationsection} for details. We write $\mathcal A[\bT]$
for the category of monoids in $\mathcal A$.

\begin{description}
\item[A3]
There exists a functor 
$$
C\co \mathcal A[\bT]\to\mathcal A[\bT],\quad A\mapsto A^c,
$$
and a natural weak equivalence $A^c\to A$ in $\mathcal A[\bT]$. We require that the monoid $A^c$ be well-based and that the
composite functor 
$$
\mathcal A[\bT]/BG_{\mathcal A}\xr{C}\mathcal A[\bT]/BG_{\mathcal A}\xr{\T_{\mathcal A}}\mathcal S,\quad \alpha\mapsto \T_{\mathcal A}(A^c\to A\xr{\alpha} 
BG_{\mathcal A}) 
$$
takes values in the full subcategory of flat monoids in $\mathcal S$. 
\end{description}

As explained earlier, we think of the symmetric monoidal category
$\mathcal A$ as a refined model of the category of spaces in which we
can represent $E_{\infty}$ monoids by strictly commutative
monoids. Whereas the functor $U$ should be thought of as a forgetful
functor,  the functor $Q$ introduced below encodes the relationship
between the monoidal product $\boxtimes$ and the cartesian  product of
spaces.

\begin{description}
\item[A4]
There exists a strong symmetric monoidal functor $Q\co \mathcal A\to
\mathcal U$  that preserves colimits and tensors with unbased
spaces. We further assume that there is a natural transformation
$U\to Q$ that induces a weak homotopy equivalence 
$$
U(A^c_1\boxtimes \dots\boxtimes A^c_k)\to Q(A^c_1\boxtimes\dots\boxtimes A^c_k)
$$
for all $k\geq 0$ and all $k$-tuples of monoids $A_1,\dots,A_k$.
\end{description}
For $k=0$, the last requirement amounts to the condition that $U(1_{\mathcal A})\to *$ be a weak homotopy equivalence.
Until now we have not made any assumptions on the homotopical behavior of 
$BG_{\mathcal A}$  with respect to the monoidal structure. The next axiom ensures that we may replace $BG_{\mathcal A}$ by a commutative monoid which is in a certain sense well-behaved.
\begin{description}
\item[A5]
There exists a well-based commutative monoid $BG_{\mathcal A}'$ in $\mathcal A$ and a weak equivalence of monoids $BG_{\mathcal A}'\to BG_{\mathcal A}$. We assume that the canonical map (induced by the augmentation)
$$
U\big(\underbrace{BG'_{\mathcal A}\boxtimes\dots\boxtimes BG'_{\mathcal A}}_{k}\big)\to
\underbrace{UBG'_{\mathcal A}\times\dots\times UBG'_{\mathcal A}}_{k}
$$
is an equivalence for all $k$.
\end{description}
Since the functor $Q$ is monoidal it takes monoids in $\mathcal A$ to ordinary topological monoids. We say that a monoid $A$ in $\mathcal A$ is \emph{grouplike} if the topological monoid  $QA^c$ is grouplike in the ordinary sense. Now let 
$\alpha\co A\to BG_{\mathcal A}$ be a monoid morphism and let us write $X=UA$ and $f=U\alpha$. 
We define $BX$ and $B^2G_U$ to be the realizations of the simplicial spaces 
$UB_{\bullet}(A^c)$ and $UB_{\bullet}(BG^c_{\mathcal A})$, and we define $Bf$ to be the realization of the simplicial map induced map $\alpha$, that is,  
\begin{equation}\label{Bf}
Bf \co BX=|UB_{\bullet}(A^c)|\to |UB_{\bullet}(BG^c_{\mathcal A})|=B^2G_U. 
\end{equation}
We shall see in Section \ref{axiomconsequences} that {\bf A3} and {\bf A4} imply that this is a delooping of $f$ if $A$ is grouplike. 

\begin{theorem}\label{BcyLtheorem}
Suppose that {\bf A1}--{\bf A5} hold and that $A$ is grouplike. Then there is a stable equivalence
$$
T(UB^{\cy}(\alpha^c))\simeq T(L^{\eta}(Bf)),
$$ 
where $L^{\eta}(Bf)$ is the map 
$$
L(BX)\stackrel{L(Bf)}{\to}L(B^2G_U)\simeq BG_U\times
B^2G_U\stackrel{\text{id}\times\eta}{\to} BG_U\times 
BG_U\to BG_U,
$$
defined as in Theorem \ref{maintheorem}.
\end{theorem}

Combining this result with Theorem \ref{Thomrigid}, we get a stable equivalence
$$
UB^{\cy}(\T_{\mathcal A}(\alpha^c))\simeq \T(L^{\eta}(Bf))
$$
and since $\T_{\mathcal A}(\alpha^c)$ is a flat replacement of 
$\T_{\mathcal A}(\alpha)$ by assumption, this gives an abstract
version of Theorem \ref{maintheorem}. In order to obtain the latter,
we must be able to lift space level data to $\mathcal A$. This is the
purpose of our final axiom.  Here $\mathcal C_{\mathcal A}$ denotes an
$A_{\infty}$ operad and $\mathcal U[\mathcal C_{\mathcal A}]$ is the
category of spaces with $\mathcal C_{\mathcal A}$-action.

\begin{description}
\item[A6]
There exists an $A_{\infty}$ operad $\mathcal C_{\mathcal A}$ that acts on $BG_U$ and a functor
$$
R\co\mathcal U[\mathcal C_{\mathcal A}]/BG_U\to\mathcal A[\bT]/BG_{\mathcal A}, 
\quad (X\xr{f} BG_U)\mapsto (R_f(X)\xr{R(f)}BG_{\mathcal A}),
$$
such that $R(\mathit{id})\co R_{\mathit{id}}(BG_U)\to BG_{\mathcal A}$ is a weak equivalence and the composite functor 
$$
\mathcal U[\mathcal C_{\mathcal A}]/BG_U\xr{R}\mathcal A[\bT]/BG_{\mathcal A}\to\mathcal A[\bT]\xr{C}\mathcal A[\bT]\xr{Q}\mathcal U[\bT]\to\mathcal U [\mathcal C_{\mathcal A}]
$$
is related to the forgetful functor by a chain of natural weak homotopy equivalences in 
$\mathcal U[\mathcal C_{\mathcal A}]$. 
\end{description}

The second arrow represents the forgetful functor and the last arrow
represents the functor induced by the augmentation from the
$A_{\infty}$ operad $\mathcal C_{\mathcal A}$ to the associa\-tivity operad, 
see \cite{may-geom}. It follows from \textbf{A3}, \textbf{A4}
and \textbf{A6} that there is a chain of weak homotopy equivalences
relating $X$ to $UR_f(X)$. Thus, in this sense $R$ is a partial right
homotopy inverse of $U$.  We shall later see that if $X$ is grouplike,
then the conditions in \textbf{A6} ensure that the delooping of $f$
implied by the operad action is homotopic to the map defined in
(\ref{Bf}).

\section{Proofs of the main results from the axioms}\label{proofsection}
We first recall some background material on tensored categories and
geometric realization.

\subsection{Simplicial objects and geometric realization}\label{realizationsection}
Let $\mathcal A$ be a cocomplete topological category. Thus, we assume
that $\mathcal A$ is enriched over $\mathcal U$ in the sense that the
morphism sets $\mathcal A(A,B)$ are topologized and the composition
maps continuous. The category $\mathcal A$ is tensored over unbased
spaces if there exists a continuous functor  
$
\otimes\co \mathcal A\times \mathcal U\to \mathcal A,
$
together with a natural homeomorphism   
$$
\Map(X,\mathcal A(A,B))\cong \mathcal A(A\otimes X,B),
$$
where $A$ and $B$ are objects in $\mathcal A$ and $X$ is a space. For the category $\mathcal U$ itself, the tensor is given by the cartesian product, and in $\Sp$ the tensor of a spectrum $A$ with an (unbased) space $X$ is the levelwise smash product $A\wedge X_+$. 
Assuming that $\mathcal A$ is tensored, there is an internal notion of geometric realization of simplicial objects.
Let $[p]\mapsto \Delta^p$ be the usual cosimplicial space  used to define the geometric realization. Given a simplicial object $A_{\bullet}$ in $\mathcal A$, we define the realization $|A_{\bullet}|$ to be the coequalizer of the diagram
$$
\coprod_{[p]\to[q]}A_q\otimes \Delta^p\rightrightarrows \coprod_{[r]}A_r\otimes \Delta^r,
$$  
where the first coproduct is over the morphisms in the simplicial category and the two arrows are defined as for the realization of a simplicial space. Notice that if we view an object $A$ as a constant simplicial object, then its realization is isomorphic to $A$. In the case where $\mathcal A$ is the category $\mathcal U$,  the above construction gives the usual geometric realization of a simplicial space and if $A_{\bullet}$ is a simplicial spectrum, then $|A_{\bullet}|$ is the usual levelwise realization.
The following lemma is an immediate consequence of the definitions. 
\begin{lemma}\label{realizationlemma}
Let $\mathcal A$ and $\mathcal B$ be cocomplete topological categories that are tensored over unbased spaces and let $\Psi\co \mathcal A\to\mathcal B$ be a continuous functor that preserves colimits and tensors. Then $\Psi $ also preserves realization of simplicial objects in the sense that there is a natural isomorphism $|\Psi A_{\bullet}|\cong\Psi |A_{\bullet}|$. \qed    
\end{lemma}

Following \cite{mmss}, we say that a morphism $U\to V$ in $\mathcal A$ is an 
\emph{$h$-cofibration} if the induced morphism from the mapping cylinder 
$$
V\cup_UU\otimes I\to V\otimes I
$$ 
admits a retraction in $\mathcal A$. This generalizes the usual notion
of a Hurewicz cofibration in $\mathcal U$, that is, of a map having the
homotopy extension property. Using the terminology from the
space level setting as in \cite{segal}, Appendix A, we say that a
simplicial object in $\mathcal A$ is \emph{good} if the degeneracy
operators are $h$-cofibrations.  We observe that a functor that
preserves colimits and tensors as in Lemma~\ref{realizationlemma} also
preserves $h$-cofibrations. It therefore also preserves the goodness
condition for simplicial objects.  If $\mathcal A$ has a monoidal
structure, then we say that a monoid $A$ is \emph{well-based} if the
unit $1_{\mathcal A}\to A$ is an $h$-cofibration.
  
\begin{lemma}\label{goodlemma}
Let $\mathcal A$ be a closed symmetric monoidal topological category
that is cocomplete and tensored over unbased spaces. If $A$ is a
well-based monoid in $\mathcal A$, then the simplicial objects  
$B_{\bullet}(A)$ and $B^{\cy}_{\bullet}(A)$ are good.
\end{lemma}

\begin{proof}
We claim that if $U\to V$ is an $h$-cofibration in $\mathcal A$, then
the induced morphism $U\boxtimes W\to V\boxtimes W$ is again an
$h$-cofibration for any object $W$. In order to show this we use that
$\mathcal A$ is closed to establish a canonical isomorphism 
$$
(V\cup_UU\otimes I)\boxtimes W\cong 
V\boxtimes W\cup_{V\boxtimes W}U\boxtimes W\otimes I.
$$ 
Similarly, we may identify $(V\otimes I)\boxtimes W$ with $(V\boxtimes
W)\otimes I$ and the claim follows. The assumption that $A$ be well-based therefore
implies the statement of the lemma.
\end{proof}
 
\subsection{Consequences of the axioms}\label{axiomconsequences}
Now let  $\mathcal A$ and $\mathcal S$ be as in Section
\ref{axiomssection}, and assume that the axioms
\textbf{A1}--\textbf{A6} hold. We shall prove the consequences of
the axioms stated in Section \ref{axiomssection}. 

\medskip
\noindent\textit{Proof of Theorem \ref{Thomrigid}.}
If $\alpha\co A\to BF_{\mathcal A}$ is a monoid morphism in $\mathcal A$, then the assumption that $T_{\mathcal A}$ be strong symmetric monoidal implies that we have an isomorphism of cyclic objects $B^{\cy}_{\bullet}(T_{\mathcal A}(\alpha))\cong T_{\mathcal A} 
(B^{\cy}_{\bullet}(\alpha))$. The first statement therefore follows from
Lemma \ref{realizationlemma} since we have assumed that $T_{\mathcal
  A}$ preserves colimits and tensors. The second statement follows
from the assumption that the diagram in  \textbf{A2} commutes up to
stable equivalence.  \qed 

\bigskip
The following lemma is an immediate consequence of Lemma \ref{goodlemma} and the assumption that $U$ and $Q$ preserve colimits and tensors. 
\begin{lemma}\label{goodspacelemma}
If $A$ is a well-based monoid in $\mathcal A$, then the simplicial objects $B_{\bullet}(A)$ and $B^{\cy}_{\bullet}(A)$ are good and so are the simplicial spaces obtained by applying $U$ and 
$Q$. \qed 
\end{lemma}

If $Z_{\bullet}$ is a simplicial object in $\mathcal A$ with internal
realization $Z$, then it follows from Lemma \ref{realizationlemma}
that $UZ$ is homeomorphic to the realization of the simplicial space
$UZ_{\bullet}$ obtained by applying $U$ degree-wise. If $Z_{\bullet}$
is a cyclic object, then $UZ_{\bullet}$ is a cyclic space and $UZ$
inherits an action of the circle group $\bT$. Recall that a monoid $A$
in $\mathcal A$ is said to be grouplike if the topological monoid
$QA^c$ is grouplike in the ordinary sense.

\begin{proposition}\label{deloopingproposition} 
If $A$ is grouplike, then $UB(A^c)$ is a delooping of $UA$.
\end{proposition}

\begin{proof}
The natural transformation in \textbf{A4} gives rise to a map of
simplicial spaces $UB_{\bullet}(A^c)\to Q B_{\bullet}(A^c)$ which is a
levelwise weak homotopy equivalence by assumption. Since by Lemma
\ref{goodspacelemma} these are good simplicial spaces, it follows that
the topological realization is also a weak homotopy
equivalence. Furthermore, since $Q$ is strong symmetric monoidal,
$QB(A^c)$ is isomorphic to the classifying space of the grouplike
topological monoid $QA^c$, hence is a delooping of the latter. Thus,
we have a chain of weak homotopy equivalences
$$
\Omega(UB(A^c))\simeq \Omega(QB(A^c))\simeq\Omega(B(QA^c)) 
\simeq QA^c\simeq UA^c\simeq UA,
$$
where the two last equivalences are implied by \textbf{A3} and \textbf{A4} .
\end{proof}

Suppose now that $A$ and therefore also $A^c$ is augmented over the unit $1_{\mathcal A}$. We then have the following analogue of  (\ref{BcyLdiagram}),
$$
\bT\times UB^{\cy}(A^c)\to UB^{\cy}(A^c)\to UB(A^c),
$$
where the last arrow is defined using the augmentation.

\begin{proposition}\label{BcyLproposition}
If $A$ is grouplike, then the adjoint map 
$$
UB^{\cy}(A^c)\to L(UB(A^c))
$$
is a weak homotopy equivalence.
\end{proposition}
\begin{proof}
It follows from the proof of Proposition \ref{deloopingproposition} that there is a weak homotopy equivalence $UB(A^c)\to QB(A^c)$ and by a similar argument we get a weak homotopy equivalence $UB^{\cy}(A^c)\to QB^{\cy}(A^c)$. These maps are related by a commutative diagram
$$
\begin{CD}
UB^{\cy}(A^c)@>>> L(UB(A^c))\\
@VV\sim V @VV\sim V\\
Q(B^{\cy}(A^c))@>>> L(QB(A^c)),
\end{CD}
$$
where, replacing $U$ by $Q$, the bottom map is defined as the map in the proposition. Using that $Q$ is strong symmetric monoidal, we can write the latter map in the form 
$$
B^{\cy}(QA^c)\to L(B(QA^c)),
$$
and as discussed in the introduction, this map is a weak homotopy
equivalence since $QA^c$ is grouplike. This implies the result.   
\end{proof}

Now let $\alpha\co A\to BG_{\mathcal A}$ be a monoid morphism in $\mathcal A$ of the form considered in Theorem \ref{BcyLtheorem}. We wish to analyze the map obtained by applying $U$ to the composite morphism 
$$
B^{\cy}(\alpha^c)\co B^{\cy}(A^c)\to B^{\cy}(BG_{\mathcal A}^c)\to B^{\cy}(BG_{\mathcal A})\to BG_{\mathcal A}. 
$$ 
Notice first that we have a commutative diagram
$$
\begin{CD}
UB^{\cy}(A^c)@>>>UB^{\cy}(BG_{\mathcal A}^c)\\
@VV\simeq V @VV\simeq V\\
L(UB(A^c))@>>>L(UB(BG_{\mathcal A}^c)),
\end{CD}
$$
where the vertical maps are weak homotopy equivalences by Proposition
\ref{BcyLproposition}.  Writing $B^2G_U$ for the delooping
$UB(BG_{\mathcal A}^c)$ as usual, we must identify the homotopy class
represented by the diagram 
\begin{equation}\label{LUBequation}
L(B^2G_U)\xl{\sim}UB^{\cy}(BG_{\mathcal A}^c)\to UB^{\cy}(BG_{\mathcal A})\to BG_U.
\end{equation}
We shall do this by applying the results of \cite{schlichtkrull-units}
and for this we need to recall some general facts about
$\Gamma$-spaces.  Consider in general a commutative well-based monoid
$Z$ in $\mathcal A$ that is augmented over the unit $1_{\mathcal
  A}$. Such a monoid gives rise to a $\Gamma$-object in $\mathcal A$,
that is, to a functor $Z\co \Gamma^o\to\mathcal A$, where $\Gamma^o$
denotes the category of finite based sets. It suffices to define $Z$
on the skeleton subcategory specified by the objects $\mathbf
n_+=\{*,1,\dots,n\}$, where we let  
$
Z(\mathbf n_+)=Z^{\boxtimes n}
$    
with an implicit choice of placement of the parenthesis in the iterated monoi\-dal product.
By definition, $Z(\mathbf 0_+)=1_{\mathcal A}$. 
The $\Gamma$-structure is defined using the symmetric monoidal
structure of $\mathcal A$, together with the multiplication and
augmentation of $Z$. (This construction models the tensor of $Z$ with finite based sets in the category of augmented commutative monoids in $\mathcal A$; similar constructions are considered in \cite{Basterra-Mandell} and \cite{schlichtkrull-units}.) 
From this point of view the diagram of simplicial objects 
$$
Z\to B^{\cy}_{\bullet}(Z)\to B_{\bullet}(Z)
$$
may be identified with that obtained by evaluating $Z$ degree-wise on
the cofibration sequence of simplicial sets $S^0\to S^1_{\bullet+}\to
S^1_{\bullet}$, see \cite{schlichtkrull-units}, Section 5.2. Composing
with the functor $U$ we get the $\Gamma$-space $UZ$ and the assumption
that the monoid $Z$ be well-based assures that this construction is
homotopically well-behaved. Notice also that $UZ$ is degree-wise
equivalent to the reduced $\Gamma$-space $\tilde UZ$ defined by the
quotient spaces 
$$
\tilde UZ(\mathbf n_+)=UZ(\mathbf n_+)/U(1_{\mathcal A}).
$$ 
Following Bousfield and Friedlander \cite{bousfield-friedlander}, we
say that $UZ$ is a \emph{special $\Gamma$-space} if the canonical maps
$$
U(\underbrace{Z\boxtimes \dots\boxtimes
  Z}_n)\to\underbrace{UZ\times\dots\times UZ}_n
$$ 
are weak homotopy equivalences for all $n$. In this case the
underlying space $UZ$ inherits a weak H-space structure, and we say
that the $\Gamma$-space is \emph{very special} if the monoid of
components is a group. This is equivalent to the condition that $Z$ be grouplike as a 
monoid in $\mathcal A$.  Consider now the composition  
$$
\bT\times UZ(S^1_+)\to UZ(S^1_+)\to UZ(S^1)
$$
defined in analogy with (\ref{BcyLdiagram}).
The following result is an immediate consequence of
\cite{schlichtkrull-units}, Proposition 7.3.
\begin{proposition}[\cite{schlichtkrull-units}]\label{Schprop}
If $Z$ is a well-based commutative monoid in $\mathcal A$ such that
the $\Gamma$-space $UZ$ is very special, then the adjoint map
$$
UZ(S^1_+)\xr{\sim} L(UZ(S^1))
$$
is a weak homotopy equivalence and the diagram
$$
UZ(S^1)\to L(UZ(S^1))\xl{\sim} UZ(S^1_+)\to UZ(S^0)
$$
represents multiplication by $\eta$. \qed
\end{proposition}
The last map in the above diagram is induced by the retraction
$S^1_+\to S^0$.

\medskip
\noindent\textit{Proof of Theorem \ref{BcyLtheorem}.}
It remains to identify the homotopy class represented by
(\ref{LUBequation}) and as explained in the introduction we have a
splitting
$$
L(B^2G_U)\simeq BG_U\times B^2G_U.
$$ 
We must prove that the homotopy class specified by the diagram
\begin{equation}\label{UBGequation}
B^2G_U\to L(B^2G_U)\xl{\sim} UB^{\cy}(BG_{\mathcal A}^c)\to BG_U
\end{equation}
is multiplication by $\eta$ in the sense of the theorem. Let 
$BG_{\mathcal A}'\to BG_{\mathcal A}$ be as in \textbf{A6} and let the
spaces $BG_U'$ and $B^2G_U'$ be defined as $UBG'_{\mathcal A}$ and
$UB(BG'^c_{\mathcal A})$. We then obtain a diagram
\begin{equation}\label{UBG'equation}
B^2G'_U \to L(B^2G_U')\xl{\sim} UB^{\cy}(BG_{\mathcal A}'^c)\to BG'_U
\end{equation}
which is term-wise weakly homotopy equivalent to (\ref{UBGequation}). 
Writing $Z$ for the commutative monoid $BG_{\mathcal A}'$, the
assumptions in \textbf{A5} ensure that $Z$ gives rise to a very
special $\Gamma$-space $UZ$. We claim that the diagram
(\ref{UBG'equation}) is term-wise weakly equivalent to the diagram in
Proposition \ref{Schprop}. Indeed, it follows from \textbf{A5} that
the natural weak equivalence in \textbf{A3} gives rise to degree-wise
weak homotopy equivalences 
$$
UB_{\bullet}(BG_{\mathcal A}')\to UZ(S^1_{\bullet}), \quad\text{and}\quad 
UB^{\cy}_{\bullet}(BG_{\mathcal A}')\to UZ(S^1_{\bullet+}).
$$
Since these are good simplicial spaces by Lemma \ref{goodspacelemma},
the induced maps of realizations are also weak homotopy
equivalences. The statement of the theorem now follows immediately
from Proposition \ref{Schprop}.\qed

\subsection{Proofs of the main theorems}\label{mainproofsection}
We first recall some general facts about deloopings of $A_{\infty}$
maps from \cite{may-geom}. Thus, let $\mathcal C$ be an $A_{\infty}$
operad with augmentation $\mathcal C\to\mathcal M$ where $\mathcal M$
denotes the associative operad. Let $C$ and $M$ be the associated
monads and consider for a $\mathcal C$-space $X$ the diagram of weak
homotopy equivalences of
$\mathcal C$-spaces
$$
X\xl{\sim}B(C,C,X)\xr{\sim}B(M,C,X)
$$
defined as in \cite[13.5]{may-geom}. The $\mathcal C$-space $B(M,C,X)$ is in fact a topological monoid and we define $B'X$ to be its classifying space, defined by the usual bar construction. If $X$ is grouplike, then $B'X$ is a delooping in the sense that $\Omega(B'X)$ is related to $X$ by a chain of weak homotopy equivalences. This construction is clearly functorial in $X$: given a map of $\mathcal C$-spaces $X\to Y$, we have an induced map $B'X\to B'Y$. 
Now let
$f\co X\to BG_U$ be a $\mathcal C_{\mathcal A}$-map and let $\alpha\co A\to BG_{\mathcal A}$ be the object in $\mathcal A[\bT]/BG_{\mathcal A}$ obtained by applying the functor $R$. We let
$Bf\co BX\to B^2G_U$ be defined as in (\ref{Bf}), where we recall that $BX$ and $B^2G_U$ denote the spaces $UB(A^c)$ and $UB(BG^c_{\mathcal A})$. The first step in the proof of Theorem \ref{maintheorem} is to compare the maps $Bf$ and $B'f$.  

\begin{lemma}\label{Bfcompatible}
There is a commutative diagram 
\[
\begin{CD}
B'X@>\sim >> BX\\
@VV B'f V @VV Bf V\\
B'BG_U@>\sim >> B^2G_U 
\end{CD}
\]
in which the vertical arrows represent chains of compatible weak homotopy equivalences.
\end{lemma}
\begin{proof}
By definition, $A$ is the monoid $R_f(X)$ and it follows from \textbf{A6} that there is a chain of weak equivalences of $\mathcal C_{\mathcal A}$-spaces relating $X$ to $QR_f(X)^c$. Applying the bar construction from \cite{may-geom} we obtain a chain of weak equivalences of topological monoids
\[
B(M,C,X)\simeq B(M,C,QR_f(X)^c)\simeq QR_f(X)^c.
\]
The last equivalence comes from the fact that $QR_f(X)^c$ is itself a
topological monoid, see \cite[13.5]{may-geom}.  This chain in turn
gives a chain of equivalences of the classifying spaces and composing
with the equivalence induced by the natural transformation $U\to Q$ we
get the upper row in the diagram. In particular, applied to the
identity on $BG_U$, this construction gives a chain of weak
equivalences relating $B'(BG_U)$ to $UB(R_{\mathit id}(BG_U)^c)$. 
Furthermore, the weak
equivalence $R_{\mathit id}(BG_U)\to BG_{\mathcal A}$ from \textbf{A6}
gives rise to a weak equivalence 
\[
UB(R_{\mathit id}(BG_U)^c)\to UB(BG_{\mathcal A}^c)=B^2G_U
\] 
and composing with this we get the bottom row in the diagram. It is clear
from the construction that the horizontal rows are compatible as
claimed.    
\end{proof}

\medskip
\noindent\textit{Proof of Theorem \ref{maintheorem}.} 
We first reformulate the theorem in a more precise form. As explained in Appendix 
\ref{loopappendix}, we may assume that our loop map $f\co X\to BG_U$ is a $\mathcal C_{\mathcal A}$-map with a delooping of the form $B_1f\co B_1X\to B_1BG_U$, where $B_1$ denotes the May classifying space functor. We again write $\alpha\co A\to BG_{\mathcal A}$ for the monoid morphism in 
$\mathcal A$ obtained by applying the functor $R$. Then $UB^{\cy}(T_{\mathcal A}(\alpha^c))$ represents the topological Hochschild homology spectrum $\Th(T(f))$ and it follows from Theorem \ref{Thomrigid} and Theorem \ref{BcyLtheorem} that there is a stable equivalence
$$
UB^{\cy}(T_{\mathcal A}(\alpha^c))\simeq T(L^{\eta}(Bf))
$$
where $Bf$ is defined as in (\ref{Bf}). We claim that there is a homotopy commutative diagram of the form
$$
\begin{CD}
B_1X@>\sim>> B'X@>\sim >> BX\\
@VVB_1f V @VV B'f V @VV Bf V\\
B_1BG_U@>\sim>> B'BG_U@>\sim>>B^2G_U.
\end{CD}
$$
Indeed, it is proved in \cite{Thomason} that the functors $B_1$ and $B'$ are naturally equivalent which gives the homotopy commutativity of the left hand square while the homotopy commutativity of the right hand square follows from Lemma \ref{Bfcompatible}. The result now follows from the homotopy invariance of the Thom spectrum functor.\qed

\medskip
In preparation for the proof of Theorem \ref{2foldtheorem} we recall that the Thom spectrum functor $T$ is multiplicative in the following sense: given maps $f\co X\to BG_U$ and $g\co Y\to BG_U$, there is a stable equivalence
$$
T(f\times g)\simeq T(f)\wedge T(g),
$$
where $f\times g$ denotes the map
$$
f\times g\co X\times Y\xr{f\times g}BG_U\times BG_U\to BG_U
$$
defined using the $H$-space structure of $BG_U$. We refer to
\cite{lewis-may-steinberger} and \cite{schlichtkrull-thom} for
different accounts of this basic fact. Of course, one of the main
points of this paper is to ``rigidify'' this stable equivalence.

\bigskip
\noindent\textit{Proof of Theorem \ref{2foldtheorem}.} As explained in the introduction, the loop space structure on $BX$ gives rise to a weak homotopy equivalence
$
X\times BX\xr{\sim} L(BX).
$
This fits in a homotopy commutative diagram
$$
\begin{CD}
X\times BX@>f\times Bf>> BG_U\times B^2G_U@>\id\times \eta>> BG_U\times BG_U\\
@VV\sim V @VV \sim V @VVV\\
L(BX)@>L(Bf)>> L(B^2G_U)@>>> BG_U,
\end{CD}
$$
where the composition in the bottom row is the map $L^{\eta}(Bf)$. By homotopy invariance of the Thom spectrum functor we get from this the stable equivalences
$$
T(L^{\eta}(Bf))\simeq T(f\times(\eta\circ Bf))\simeq T(f)\wedge T(\eta\circ Bf)
$$
and the result follows from Theorem \ref{maintheorem}.\qed

\bigskip
\noindent\textit{Proof of Theorem \ref{3foldtheorem}.} 
Notice first that if $X$ is a 3-fold loop map, then the unstable Hopf map $\eta$ gives rise to a map
$$
\eta\co BX\simeq \Map_*(S^2,B^3X)\xr{\eta^*}\Map_*(S^3,B^3X)\simeq X.
$$
Let $\Phi$ be the self homotopy equivalence of $X\times BX$ defined by
$$
\Phi=
\left[\begin{array}{cc}
\id & \eta\\
0 & \id
\end{array}\right]\co
X\times BX\xr{\sim} X\times BX.
$$
Given a 3-fold loop map $f\co X\to BG_U$, we then have a homotopy commutative  
diagram
$$
\begin{CD}
X\times BX@> f\times Bf>> BG_U\times B^2G_U@>\id\times \eta >>BG_U\times BG_U@>>>BG_U\\
@V\sim V\Phi V @V\sim V \Phi V @. @|\\
X\times BX @>f\times Bf >> BG_U\times B^2G_U @>\id\times *>> BG_U\times BG_U@>>> BG_U,
\end{CD}
$$
where $*$ denotes the trivial map. It follows from the proof of Theorem \ref{2foldtheorem} that the composition of the maps in the upper row is weakly homotopy equivalent to $L^{\eta}(Bf)$. Thus, by homotopy invariance of the Thom spectrum functor we get the stable equivalence
$$
T(L^{\eta}(Bf))\simeq T(f\times*)\simeq T(f)\wedge BX_+
$$
and the result again follows from Theorem \ref{maintheorem}.\qed

\subsection{Eilenberg-Mac Lane spectra and $p$-local Thom spectra}\label{p-localsection}
In this section we first collect some background material on $p$-local Thom spectra for a prime $p$. We then complete the proofs of the theorems in the introduction calculating the topological Hochschild homology of certain Eilenberg-Mac Lane spectra. There is a $p$-local version of the Thom spectrum functor 
$$
T\co\mathcal U/BF_{(p)}\to \Sp_{(p)},
$$
defined on the category of spaces over the classifying space $BF_{(p)}$ for $p$-local stable spherical fibrations and with values in the category $\Sp_{(p)}$ of $p$-local spectra. In fact, each of our ``rigid'' Thom spectrum functors has a $p$-local analogue. This is particularly easy to explain in the setting of $\I$-spaces and symmetric spectra: Let $S^1_{(p)}$ be the 
$p$-localization of $S^1$ and let $S^n_{(p)}$ be the $n$-fold smash product. We write $F_{(p)}(n)$ for the topological monoid of homotopy equivalences of $S^n_{(p)}$ and $BF_{(p)}(n)$ for its classifying space. The definition of the symmetric Thom spectrum functor can then readily be modified to a $p$-local version with values in the category of module spectra over the $p$-local sphere spectrum. A similar construction works in the $\mathcal L(1)$-space setting using a continuous localization functor as in \cite{iwase}. 

It follows by inspection of our formal framework and the implementations in Section 
\ref{L-spacesection} and Section \ref{I-spacesection} that there are $p$-local versions of Theorems \ref{maintheorem}, \ref{2foldtheorem} and \ref{3foldtheorem}. The application of this to the calculation of $\Th(H\mathbb Z/p)$ is as follows. Let $BF_{(p)}$ be the colimit of the spaces $BF_{(p)}(n)$ and notice that $\pi_1BF_{(p)}$ is the group $\mathbb Z_{(p)}^*$ of units in $\mathbb Z_{(p)}$. Consider the map $S^1\to BF_{(p)}$ corresponding to the unit $1-p$ and let $f\co\Omega^2(S^3)\to BF_{(p)}$ be the extension to a 2-fold loop map. It is an observation of Hopkins that the associated $p$-local Thom spectrum $T(f)$ is a model of $H\mathbb Z/p$. The argument is similar to that for $H\mathbb Z/2$, using that the Thom space of the map $S^1\to BF_{(p)}(1)$ is a mod $p$ Moore space. 

\medskip
\noindent\textit{Proof of Theorem \ref{Z/p-theorem}.} 
Applying the $p$-local version of Theorem \ref{2foldtheorem} to the 2-fold loop map 
$f\co \Omega^2(S^3)\to BF_{(p)}$, we get that
$$
\Th(H\mathbb Z/p)\simeq \Th(T(f))\simeq H\mathbb Z/p\wedge T(\eta\circ Bf)
$$
and it follows from the $H\mathbb Z/p$ Thom isomorphism theorem that there is a stable equivalence
$$
H\mathbb Z/p\wedge T(\eta\circ Bf)\simeq H\mathbb Z/p\wedge \Omega(S^3)_+.
$$
This is (the $p$-local version of) the homotopy theoretical interpretation of the Thom 
isomorphism \cite{mahowald-ray}, \cite{may-fiberwise}. Indeed, given a map $g\co X\to BF_{(p)}$, the condition for $T(g)$ to be $H\mathbb Z/p$ orientable is that the composite map 
$$
X\xr{g}BF_{(p)}\to B\mathbb Z_{(p)}^*\to B(\mathbb Z/p)^*
$$ 
be null homotopic. 
In our case this holds since $\Omega(S^3)$ is simply connected.\qed

\medskip
In order to realize $H\mathbb Z_{(p)}$ as a Thom spectrum we compose with the canonical map $\Omega^2(S^3\langle 3\rangle)\to\Omega^2(S^3)$ as explained in the introduction. 

\begin{theorem}
There is a stable equivalence
$$
\Th(H\mathbb Z_{(p)})\simeq H\mathbb Z_{(p)}\wedge \Omega(S^3\langle 3\rangle)_+.
$$
\end{theorem}
\begin{proof}
Let $f\co \Omega^2(S^3\langle 3\rangle)\to BF_{(p)}$ be the composite map. By the $p$-local version of Theorem \ref{2foldtheorem} we have the stable equivalence
$$
\Th(H\mathbb Z_{(p)})\simeq \Th(T(f))\simeq H\mathbb Z_{(p)}\wedge T(\eta\circ Bf)
$$
and since $\Omega(S^3\langle 3\rangle)$ is simply connected,
$$
H\mathbb Z_{(p)}\wedge T(\eta\circ Bf)\simeq H\mathbb Z_{(p)}\wedge \Omega(S^3\langle 3\rangle)_+ 
$$
by the $H\mathbb Z_{(p)}$ Thom isomorphism.
\end{proof}
Using that topological Hochschild homology commutes with localization and that $\Th(H\mathbb Z)$ is a generalized Eilenberg-Mac Lane spectrum, this immediately implies the calculation in Theorem \ref{Z-theorem}

\subsection{The Thom spectra $M\Br$, $M\Sigma$ and $M\GL(\mathbb Z)$}\label{groupthom}
As explained in the introduction, a compatible system of group homomorphisms $G_n\to O(n)$ gives rise to a Thom spectrum $MG$. In particular, we have the systems of groups
$$
\Br_n\to \Sigma_n\to \GL_n(\mathbb Z)\to O(n)
$$
given by the braid groups $\Br_n$, the symmetric groups $\Sigma_n$, and the general linear groups $\GL_n(\mathbb Z)$. Each one of these families has a natural ``block sum'' pairing, inducing a multiplicative structure on the associated Thom spectrum, see \cite{bullett} and 
\cite{cohen}. Except for $M\Br$, these spectra are in fact commutative symmetric ring spectra. In order to fit these examples in our framework we observe that in each case the induced map 
$BG_{\infty}\to BO$ from the stabilized group factors over Quillen's plus-construction to give a loop map $BG_{\infty}^+\to BO$. Furthermore, it follows from the fact that $BG_{\infty}\to BG_{\infty}^+$ induces an isomorphism in homology that $MG$ is equivalent to the corresponding Thom spectrum $T(BG_{\infty}^+)$. For the braid groups, Cohen proves \cite{cohen} that 
$B\Br_{\infty}^+$ is equivalent to $\Omega^2(S^3)$ and combining this with Mahowald's theorem it follows that $M\Br$ is a model of $H\mathbb Z/2$. Thus, in this case we recover the calculation in Theorem \ref{Z/p-theorem} for $p=2$. Notice, that the identification of $M\Br$ with $H\mathbb Z/2$ implies that $M\Sigma$, $M\GL(\mathbb Z)$ and $MO$ are $H\mathbb Z/2$-module spectra, hence generalized Eilenberg-Mac Lane spectra.

\medskip
\noindent\textit{Proof of Theorem \ref{Msigma}.}
By the above remarks, $M\Sigma$ is equivalent to the Thom spectrum of the map 
$B\Sigma_{\infty}^+\to BO$. It follows from the Barratt-Priddy-Quillen Theorem that 
$B\Sigma_{\infty}^+$ is equivalent to the base point component of $Q(S^0)$ and that the map in question is the restriction to the base point component of the infinite loop map 
$Q(S^0)\to BO\times \mathbb Z$. Hence $\tilde Q(S^1)$ is a delooping of $B\Sigma_{\infty}^+$ and the result follows from Theorem \ref{3foldtheorem}.\qed

\medskip
\noindent\textit{Proof of Theorem \ref{MGL}.}
It again follows from the above discussion that $M\GL(\mathbb Z)$ is equivalent to the Thom spectrum of the infinite loop map $B\GL_{\infty}^+\to BO$. Here the domain is the base point component of Quillen's algebraic K-theory space of $\mathbb Z$ and the result follows from Theorem \ref{3foldtheorem}. \qed

\section{Operadic products in the category of spaces}\label{Lprelims}

In this section, we adapt the construction of the ``operadic'' smash
product of spectra from \cite{ekmm} to the context of topological
spaces.  The conceptual foundation of the approach to structured ring
spectra undertaken in \cite{ekmm} is the observation that by
exploiting special properties of the linear isometries operad,
it is possible to define a weakly symmetric monoidal product on a
certain category of spectra such that the (commutative) monoids are
precisely the ($E_\infty$) $A_\infty$ ring spectra.  Since many of the
good properties of this product are a consequence of the nature of the
operad, such a construction can be carried out in other categories ---
for instance, \cite{kriz-may} studied an algebraic version of this
definition.  Following an observation of Mandell, we introduce a
version of this construction in the category of spaces.  Much of this
work appeared in the first author's University of Chicago thesis
\cite{blumberg-thesis}.

\subsection{The weakly symmetric monoidal category of $\sL(1)$-spaces}

Fix a countably infinite-dimensional real inner product space $U$,
and let $\sL(n)$ denote the $n$th space of the linear isometries
operad associated to $U$; recall that this is the space of linear
isometries $\sL(U^n, U)$.  In particular, the space $\sL(1) = \sL(U,
U)$ is a topological monoid.  We begin by considering the category of
$\sL(1)$-spaces: unbased spaces equipped with a map $\sL(1) \times X
\rightarrow X$ which is associative and unital.  We can equivalently
regard this category as the category $\aU[\bL]$ of algebras over the
monad $\bL$ on the category $\aU$ of unbased spaces which takes $X$ to
$\sL(1) \times X$.  

The category $\aU[\bL]$ admits a product $X \lprod Y$ defined in
analogy with the product $\sma_{\sL}$ on the category of $\bL$-spectra
\cite[\S I.5.1]{ekmm}.  Specifically, there is an obvious action of
$\sL(1) \times \sL(1)$ on $\sL(2)$ via the inclusion of $\sL(1) \times
\sL(1)$ in $\sL(2)$.  In addition, there is a natural action of
$\sL(1) \times \sL(1)$ on $X \times Y$ given by the isomorphism 
\[(\sL(1) \times \sL(1)) \times (X \times Y) \cong (\sL(1) \times X)
\times (\sL(1) \times Y).\]
We define $\lprod$ as the balanced product of these two actions.

\begin{definition}
The product $X \lprod Y$ is the coequalizer of the diagram
\[
\xymatrix{
\sL(2) \times \sL(1) \times \sL(1) \times X \times Y \ar[r]<2pt> \ar[r]<-2pt> & \sL(2) \times X \times Y. \\
}
\]
The coequalizer is itself an $\sL(1)$-space via the action of $\sL(1)$
on $\sL(2)$.
\end{definition}

The arguments of \cite[\S I.5]{ekmm} now yield the following proposition.

\begin{proposition}\label{propfree}
\hspace{5 pt}
\begin{enumerate}
\item The operation $\lprod$ is associative.  For any $j$-tuple
  $M_{1},\dotsc , M_{k}$ of $\sL(1)$-spaces there is a canonical and
  natural isomorphism of $\sL(1)$-spaces 
\[
     M_{1}\lprod \dotsb \lprod M_{k} \iso \sL(k)\times_{\sL(1)^{k}}
     M_{1}\times \dotsb \times M_{k} 
\]
where the iterated product on the left is associated in any fashion.
\item The operation $\lprod$ is commutative.  There is a natural
  isomorphism of $\sL(1)$-spaces  
\[
     \tau: X\lprod Y \iso Y\lprod X
\]
with the property that $\tau^{2}=\id$.
\end{enumerate}
\end{proposition}

There is a corresponding mapping $\sL(1)$-space $\lmap (X,Y)$ which
satisfies the usual adjunction; in fact, the definition is forced by
the adjunctions.

\begin{definition}
The mapping space $\lmap (X,Y)$ is the equalizer of the diagram
\[
\xymatrix{
\Map_{\aU[\bL]}(\sL(2) \times X, Y) \ar[r]<2pt> \ar[r]<-2pt> &
\Map_{\aU[\bL]}(\sL(2) \times \sL(1) \times \sL(1) \times X, Y).\\
}
\]
Here one map is given by the action of $\sL(1) \times \sL(1)$ on
$\sL(2)$ and the other via the adjunction 
\[\Map_{\aU[\bL]}(\sL(2) \times \sL(1) \times \sL(1) \times X, Y) \cong
\Map_{\aU[\bL]}(\sL(2) \times \sL(1) \times X, \Map_{\aU[\bL]}(\sL(1),
Y))\]
along with the action $\sL(1) \times X \to X$ and coaction 
\[Y \to \Map_{\aU[\bL]}(\sL(1),Y).\]
\end{definition}

A diagram chase verifies the following proposition.

\begin{proposition}
Let $X$, $Y$, and $Z$ be $\sL(1)$-spaces.  Then there is an adjunction
homeomorphism
\[\Map_{\aU[\bL]}(X \lprod Y, Z) \cong \Map_{\aU[\bL]}(X,
F^\boxtimes(Y,Z)).
\]
\end{proposition}

The natural choice for the unit of the product $\lprod$ is
the point $*$, endowed with the trivial $\sL(1)$-action.  As in
\cite[\S1.8.3]{ekmm}, there is a unit map $* \lprod X \to
X$ which is compatible with the associativity and commutativity
properties of $\lprod$.

\begin{proposition}
Let $X$ and $Y$ be $\sL(1)$-spaces.  There is a natural unit map of
$\sL(1)$-spaces $\lambda \colon * \lprod X \to X$.  The symmetrically
defined map $X \lprod * \to X$ coincides with the composite $\lambda
\tau$.  Under the associativity isomorphism $\lambda \tau \lprod \id
\cong \id \lprod \lambda$, and, under the commutativity isomorphism,
these maps also agree with $* \lprod (X \lprod Y) \to X \lprod Y$. 
\end{proposition}

However, just as is the case in the category of $\bL$-spectra, $*$ is
not a strict unit for $\lprod$: the unit map $* \lprod X \rightarrow
X$ is not necessarily an isomorphism.  However, it is always a weak
equivalence.  We omit the proof of this fact, as it is very technical
and essentially similar to the proof of the analogous fact for
$\bL$-spectra \cite[1.8.5]{ekmm}.  We remark only that it is a
consequence of the remarkable point-set properties of the linear
isometries operad and the isomorphism $* \lprod * \rightarrow *$.

\begin{proposition}\label{P:unit}
For any $\sL(1)$-space $X$, the unit map $\lambda \colon * \lprod X
\rightarrow X$ is a weak equivalence of $\sL(1)$-spaces.
\end{proposition}

In summary, the category $\aU[\bL]$ is a closed weak symmetric
monoidal category, with product $\lprod$ and weak 
unit $*$.  Recall that this means that the $\aU[\bL]$, $\lprod$,
$F^\boxtimes(-,-)$, and $*$ satisfy all of the axioms of a closed
symmetric monoidal category except that the unit map is not required
to be an isomorphism \cite[\S II.7.1]{ekmm}.

\subsection{Monoids and commutative monoids for $\lprod$}

In this section, we study $\lprod$-monoids and commutative
$\lprod$-monoids in $\aU[\bL]$; these are defined as algebras over
certain monads, following \cite[\S2.7]{ekmm}.  We show that
$\lprod$-monoids are $A_\infty$ spaces and commutative
$\lprod$-monoids are $E_\infty$ spaces.  One can prove this directly,
as is done in the algebraic setting in \cite[\S V.3.1]{kriz-may}, but
we prefer to follow the categorical approach given for $\bL$-spectra
and $\sma_{\sL}$ in \cite[\S II.4]{ekmm}.

In any weakly symmetric monoidal category, monoids and commutative
monoids can be regarded as algebras over the monads $\bT$ and $\bP$
defined as follows.  Let $X^{\boxtimes j}$ denote the $j$-fold power with
respect to $\lprod$, where $X^0 = *$.  Then we define the monads on
the category of $\sL(1)$-spaces as 
\[
\xymatrix{\bT X = \bigvee_{j \geq 0} X^{\boxtimes j} && \bP X = \bigvee_{j \geq 0}
  X^{\boxtimes j}/\Sigma_j \\
},
\]
where the unit is given by the inclusion of $X$ into the wedge and the
product is induced by the obvious identifications (and the unit map,
if any indices are 0).  We regard $A_\infty$ and $E_\infty$ spaces as
algebras over the monads $\bB$ and $\bC$ on based spaces.  Recall that
these monads are defined as 
\[
\xymatrix{
\bB Y = \bigvee_{j \geq 0} \sL(j) \times X^j && \bC Y = \bigvee_{j
  \geq 0} \sL(j) \times_{\Sigma_j} X^j, \\
}
\]
subject to an equivalence relation which quotients out the basepoint
(where here $X^n$ indicates the iterated cartesian product).  If we
ignore the quotient, we obtain corresponding monads on unbased spaces
which we will denote $\bB^\prime$ and $\bC^\prime$.  The main tool for
comparing these various categories of algebras is the lemma
\cite[\S II.6.1]{ekmm}, which we write out below for clarity.

\begin{lemma}\label{lemmonad}
Let $\bS$ be a monad in a category $\aC$ and let $\bR$ be a monad in
the category $\aC[\bS]$ of $\bS$-algebras.  Then the category
$\aC[\bS][\bR]$ of $\bR$-algebras in $\aC[\bS]$ is isomorphic to the
category $\aC[\bR \bS]$ of algebras over the composite monad $\bR \bS$
in $\aC$.  Moreover, the unit of $\bR$ defines a map $\bS \to \bR \bS$
of monads in $\aC$.  An analogous assertion holds for comonads.
\end{lemma}

Recall that the category of based spaces can be viewed as the category
of algebras over the monad $\bU$ in unbased spaces which adjoins a
disjoint basepoint.  In mild abuse of notation, we will also refer to
the monad on $\sL(1)$-spaces which adjoins a disjoint basepoint with
trivial $\sL(1)$-action as $\bU$.

\begin{proposition}\label{propoperadcomp}
There is an isomorphism of categories between $\lprod$-monoids and
$A_\infty$ spaces structured by the non-$\Sigma$ linear isometries
operad.  Similarly, there is an isomorphism of categories between
commutative $\lprod$-monoids and $E_\infty$ spaces structured by the
linear isometries operad.
\end{proposition}

\begin{proof}
A straightforward verification shows that $\bC^{\prime} \cong \bC \bU$ as
monads on unbased spaces.  Then Lemma~\ref{lemmonad} shows that
$\bC^{\prime}$-algebras 
in unbased spaces are equivalent to $\bC$-algebras in based spaces.
Next, there is an identification of monads on unbased spaces
$\bC^{\prime} \cong \bT \bL$; Proposition~\ref{propfree} implies an
isomorphism of objects, and the comparison of monad structures is
immediate.  Lemma~\ref{lemmonad} then implies that
$\bC^{\prime}$-algebras in unbased spaces are equivalent to
$\bT$-algebras in $\sL(1)$-spaces, and combined with the initial
observation this implies the desired result.  The commutative case is
analogous. 
\end{proof}

\subsection{The symmetric monoidal category of $*$-modules}

In this section we define a subcategory of $\sL(1)$-spaces which forms
a closed symmetric monoidal category with respect to $\lprod$.  This
is necessary for our application to topological Hochschild homology
--- in order to define the cyclic bar construction as a strict simplicial
object, we need a unital product to define the degeneracies.  In fact,
there are two possible approaches to constructing a symmetric monoidal
category from the weak symmetric monoidal category of $\sL(1)$-spaces:
These parallel the approaches developed in \cite{kriz-may} and
\cite{ekmm}.  If we restrict attention to the category $\aT[\bL]$ of
based $\sL(1)$-spaces (where the $\sL(1)$ action is trivial on the
basepoint), there is a unital product $\star_{\sL}$ formed as the
pushout 
\[
\xymatrix{
(X \lprod *) \vee (* \lprod Y) \ar[r]\ar[d] & X \vee
  Y \ar[d] \\
X \lprod Y \ar[r] & X \star_{\sL} Y. \\
}
\]
In the algebraic setting of \cite{kriz-may}, this kind of construction
is our only option.  However, since there is an isomorphism $*
\lprod * \cong *$ we can also pursue the strategy of considering a
subcategory of $\sL(1)$-spaces analogous to the category of
$S$-modules \cite{ekmm}.  Specifically, observe that the
$\sL(1)$-space $* \lprod X$ is unital in the sense that the
unit map $* \lprod (* \lprod X) \to (*
\lprod X)$ is an isomorphism.  

\begin{definition}
The category $\aM_*$ of $*$-modules is the subcategory of
$\sL(1)$-spaces such that the unit map $\lambda\colon * \lprod X \to
X$ is an isomorphism.  For $*$-modules $X$ and $Y$, define $X
\boxtimes Y$ as $X \lprod Y$ and $F_\boxtimes (X,Y)$ as $* \lprod
\lmap(X,Y)$.
\end{definition}

The work of the previous section implies that $\aM_*$ is a closed
symmetric monoidal category.  For use in establishing a model
structure on $\aM_*$ in Theorem~\ref{theoremstarmod}, we review a more
obscure aspect of this category, following the analogous treatment for
the category of $S$-modules \cite[\S II.2]{ekmm}.  The functor $*
\lprod -$ is not a monad in $\sL(1)$-spaces.  However, the category of
$*$-modules has a ``mirror image'' category to which it is naturally
equivalent, and this equivalence facilitates formal analysis of the
category of $*$-modules.

Let $\aM^*$ be the full subcategory of counital $\sL(1)$-spaces:
$\sL(1)$-spaces $Z$ such that the counit map $Z \to F^\boxtimes(*,Z)$
is an isomorphism.  Following the notation of \cite[\S II.2]{ekmm},
let $f$ denote the functor $F^\boxtimes(*,-)$ and $s$ denote the
functor $* \lprod (-)$.  Let $r$ be the inclusion of the counital
$\sL(1)$-spaces into the category of $\sL(1)$-spaces, and $\ell$ the
inclusion of the unital $\sL(1)$-spaces ($*$-modules) into
$\sL(1)$-spaces.  We have the following easy lemma about these
functors.

\begin{lemma}
The functor $f$ is right adjoint to the functor $s$ and left adjoint
to the inclusion $r$.
\end{lemma}

Now, we obtain a pair of mirrored adjunctions
\[
\xymatrix{
\aU[\bL] \ar@<1ex>[r]^s & \aM_* \ar@<1ex>[l]^{rf\ell}
\ar@<1ex>[r]^\ell & \aU[\bL] \ar@<1ex>[l]^s &&
\aU[\bL] \ar@<1ex>[r]^f & \aM^* \ar@<1ex>[l]^r
\ar@<1ex>[r]^{\ell sr} & \aU[\bL] \ar@<1ex>[l]^f. \\  
}
\]
The composite of the first two left adjoints is $* \lprod (-)$ and the
composite of the second two right adjoints is $F^\boxtimes(*,-)$.
These are themselves adjoints, and now by the uniqueness of adjoints
we have the following consequence.

\begin{lemma}
For an $\sL(1)$-space $X$, the maps 
\[* \lprod X \to * \lprod F^\boxtimes (*, X)\]
and
\[F^\boxtimes(*, * \lprod X) \to F^\boxtimes(*,X)\]
are natural isomorphisms.
\end{lemma}

An immediate consequence of this is that the category $\aM_*$ and
$\aM^*$ are equivalent, and in particular we see that the category of
$*$-modules is equivalent to the category of algebras over the monad
$rf$ determined by the adjunction (see the proof of \cite[\S
  II.2.7]{ekmm} for further details).

\subsection{Monoids and commutative monoids in $\aM_*$}

The monads $\bT$ and $\bP$ on $\sL(1)$-spaces restrict to define
monads on $\aM_*$.  The algebras over these monads are monoids and
commutative monoids for $\boxtimes$, respectively.  Thus, a
$\boxtimes$-monoid in $\aM_*$ is a $\lprod$-monoid in $\sL(1)$ which
is also a $*$-module (and similarly for commutative
$\boxtimes$-monoids).  The functor $* \lprod (-)$ gives us a means to
functorially replace $\lprod$-monoids and commutative $\lprod$-monoids
with $\boxtimes$-monoids and commutative $\boxtimes$-monoids.

\begin{proposition}
Given a $\lprod$-monoid $X$, the object $* \lprod X$ is a
$\boxtimes$-monoid and the map $\lambda \colon * \lprod X \to X$ is a
weak equivalence of $\lprod$-monoids.  The analogous results in the
commutative case hold.
\end{proposition}

Furthermore, note that the standard formal arguments imply that
$\boxtimes$ is the coproduct in the category $\aM_*[\bP]$.

\subsection{Functors to spaces}

In this section, we discuss two functors which allow us to compare the
categories $\aU[\bL]$ and $\aM_*$ to spaces.  There is a continuous
forgetful functor $U \colon \aU[\bL] \to \aU$ such that $U(*) = *$.
This functor restricts to a continuous forgetful functor $U \colon
\aM_* \to U$.  In addition, we have another continuous functor from
$\sL(1)$-spaces to $\aU$.  There is a map of topological monoids
$\theta \colon \sL(1) \to *$.  Associated to any map of monoids is an
adjoint pair $(\theta_*, \theta^*)$.  The right adjoint $\theta^*
\colon \aU \to \aU[\bL]$ is the functor which assigns a trivial action
to a space, and the left adjoint is described in the next definition.

\begin{definition}\label{defq}
The monoid map $\sL(1) \rightarrow *$ induces a functor $Q \colon
\aU[\bL] \to \aU$ which takes an $\sL(1)$-space $X$ to $*
\times_{\sL(1)} X$.  $Q$ is the left adjoint to the pullback functor
which gives a space $Y$ the trivial $\sL(1)$-action.  $Q$ restricts to a functor
$Q \colon \aM_* \to \aU$.
\end{definition}

The interest of this second functor $Q$ is that it is strong symmetric
monoidal: it allows us to relate $\boxtimes$ to the the cartesian
product of spaces.

\begin{lemma}\label{lemqstrong}
The functor $Q \colon \aU[\bL] \to \aU$ is strong symmetric monoidal
with respect to the symmetric monoidal structures induced by $\lprod$
and $\times$ respectively.  Correspondingly, the functor $Q \colon
\aM_* \to \aU$ is strong symmetric monoidal with respect to
$\boxtimes$ and $\times$ respectively.
\end{lemma}

\begin{proof}
Let $X$ and $Y$ be $\sL(1)$-spaces.  We need to compare $*
\times_{\sL(1)} (X \lprod Y)$ and $(* \times_{\sL(1)} X) \times
(* \times_{\sL(1)} Y)$.  Observe that $\sL(2)$ is homeomorphic to
$\sL(1)$ as a left $\sL(1)$-space, by composing with an isomorphism
$g\colon U^2 \rightarrow U$.  Therefore we have isomorphisms
\begin{eqnarray*}
* \times_{\sL(1)} (X \boxtimes_{\sL} Y) & = & * \times_{\sL(1)} \sL(2)
\times_{\sL(1) \times \sL(1)} (X \times Y) \\ & \cong & (*
\times_{\sL(1)} X) \times (* \times_{\sL(1)} Y) 
\end{eqnarray*}
One checks that the required coherence diagrams commute.  The result
now follows, as $* \times_{\sL(1)} * \cong *$.
\end{proof}

The preceding result implies that $Q$ takes $\lprod$-monoids and
$\boxtimes$-monoids to topological monoids.  In
Section~\ref{sechomqu}, we will describe conditions under which the
natural map $UX \to QX$ is a weak equivalence.

\subsection{Model category structures}

In this section, we describe the homotopy theory of the categories
described in the previous sections.  We begin by establishing model
category structures on the various categories and identifying the
cofibrant objects.  We rely on the following standard lifting result
(e.g. \cite[A.3]{Schwede-Shipley}).

\begin{theorem}\label{theoremlift}
Let $\aC$ be a cofibrantly generated model category where all objects
are fibrant, with generating cofibrations $I$ and acyclic generating
cofibrations $J$.  Assume that the domains of $I$ and $J$ are small
relative to the classes of transfinite pushouts of maps in $I$ and $J$
respectively.  Let $\bA$ be a continuous monad on $\aC$ which commutes with
filtered direct limits and such that all $\bA$-algebras have a path
object.  Then the category $\aC[\bA]$ has a cofibrantly generated
model structure in which the weak equivalences and fibrations are
created by the forgetful functor to $\aC$.  The generating
cofibrations and acyclic cofibrations are the sets $\bA I$ and $\bA J$
respectively.
\end{theorem}

Furthermore, when $\aC$ is topological, provided that the monad $\bA$
preserves reflexive coequalizers, the category $\aC[\bA]$ will also be
topological \cite[\S VII.2.10]{ekmm}.  All of the monads that arise in
this paper preserve reflexive coequalizers \cite[\S II.7.2]{ekmm}.
We now record the model structures on the categories we study.  We
assume that $\aU$ and $\aT$ are equipped with the standard model
structure in which the weak equivalences are the weak homotopy
equivalences and the fibrations are the Serre fibrations.

\begin{theorem}\label{theoreml1mod}
The category $\aU[\bL]$ admits a cofibrantly generated
topological model structure in which the fibrations and weak
equivalences are detected by the forgetful functor to spaces.  Limits
and colimits are constructed in the underlying category $\aU$.
\end{theorem}

We use an analogous argument to deduce the existence of a topological
model structure on $*$-modules, employing the technique used in the
proof of \cite[\S VII.4.6]{ekmm}.  The point is that counital
$\sL(1)$-spaces are algebras over a monad, and moreover there is an
equivalence of categories between counital $\sL(1)$-spaces and unital
$\sL(1)$-spaces; recall the discussion of the ``mirror image''
categories above.  The following theorem then follows once again from
Theorem~\ref{theoremlift}.

\begin{theorem}\label{theoremstarmod}
The category $\aM_*$ admits a cofibrantly generated topological model
structure.  A map $f \colon X \to Y$ of $*$-modules is
\begin{itemize}
\item A weak equivalences if the map $Uf \colon UX
  \to UY$ of spaces is a weak equivalence, and
\item A fibration if the induced map $F_\boxtimes(*, X) \to
F_\boxtimes(*,Y)$ is a fibration of spaces.
\end{itemize}
Colimits are created in the category $\aU[\bL]$, and limits
are created by applying $* \boxtimes (-)$ to the limit in the category
$\aU[\bL]$.
\end{theorem}

Notice that although the fibrations have changed (since the functor to
spaces which we're lifting over is $F_\boxtimes(*,-)$ and not the
forgetful functor), nonetheless this category still has the useful
property that all objects are fibrant. 

\begin{lemma}
All objects in the category of $*$-modules are fibrant.
\end{lemma}

\begin{proof}
This is an immediate consequence of the isomorphism $F^\boxtimes(*,*)
\cong *$ and the fact that all spaces are fibrant.
\end{proof}

As a consequence, we obtain the following summary theorem about model
structures on monoids and commutative monoids.

\begin{theorem}\label{theoremmonoids}
The categories $(\aU[\bL])[\bT]$ and $(\aU[\bL])[\bP]$ of
$\lprod$-monoids and commutative $\lprod$-monoids in $\sL(1)$-spaces
admit cofibrantly generated topological model structures in which the
weak equivalences and fibrations are maps which are weak equivalences
and fibrations of $\sL(1)$-spaces.  Similarly, the categories
$\aM_*[\bT]$ and $\aM_*[\bP]$ of $\boxtimes$-monoids and commutative
$\boxtimes$-monoids in $\aM_*$ admit cofibrantly generated topological
model structures in which the weak equivalences and fibrations are the
maps which are weak equivalences and fibrations in $\aM_*$.
Limits are created in the underlying category and colimits are created
as a certain coequalizer \cite[\S II.7.4]{ekmm}.
\end{theorem}

Next, in order to work with left derived functors associated to functors
with domain one of these categories, we describe a convenient
characterization of the cofibrant objects.  Let $\aC$ be a
cofibrantly generated model category with generating cofibrations
$\{A_i \to B_i\}$ and let $\bA$ be a continuous monad.  A cellular
object in the category $\aC[\bA]$ is a sequential colimit $X =
\colim_i X_i$, with $X_0 = *$ and $X_{i+1}$ defined as the pushout
\[
\xymatrix{
\bigvee_\alpha \bA A_i \ar[d] \ar[r] & X_i \ar[d] \\
\bigvee_\alpha \bA B_i \ar[r] & X_{i+1},\\
}
\]
where here $\alpha$ is some indexing subset of the indexing set of the
generating cofibrations.  The following result is then a formal
consequence of the proof of Theorem~\ref{theoremlift}.

\begin{proposition}
In the model structures of Theorem~\ref{theoreml1mod},
Theorem~\ref{theoremstarmod}, and Theorem~\ref{theoremmonoids}, the
cofibrant objects are retracts of cellular objects.
\end{proposition}

These categories also satisfy appropriate versions of the
``Cofibration Hypothesis'' of \cite[\S VII]{ekmm}.  That is, for a
cellular object the maps $X_n \to X_{n+1}$ are (unbased) Hurewicz
cofibrations and the sequential colimit $X = \colim_i X_i$ can be computed
in the underlying category of spaces.

\subsection{Homotopical analysis of $\lprod$}\label{sechomqu}

In this section, we discuss the homotopical behavior of $\lprod$.  We
will show that the left derived functor of $\lprod$ is the cartesian
product.  The analysis begins with an essential proposition based on a
useful property of $\sL(2)$.

\begin{proposition}\label{propfreeprod}
For spaces $X$ and $Y$, there are isomorphisms of $\sL(1)$-spaces 
\[
(\sL(1) \times X) \lprod (\sL(1) \times Y) \cong \sL(2) \times (X
\times Y) \cong \sL(1) \times (X \times Y).
\]
As a consequence, if $M$ and $N$ are cell $\sL(1)$-spaces then so is
$M \lprod N$.  If $M$ and $N$ are cell $*$-modules, then so is $M
\boxtimes N$.
\end{proposition}

\begin{proof}
The first isomorphism is immediate from the definitions.  The second
is a consequence of the fact that $\sL(2)$ is isomorphic to $\sL(1)$
as an $\sL(1)$-space via choice of a linear isometric isomorphism $f
\colon U^2 \to U$; there is then a homeomorphism $\gamma \colon \sL(1)
\times \{f\} \to \sL(2)$.  The last two statements now follow by
induction from the analogous result for the cartesian product of cell
spaces.
\end{proof}

With this in mind, we can prove the following result regarding the
homotopical behavior of $\boxtimes$.

\begin{theorem}
The category $\aM_*$ is a monoidal model category with respect to the
symmetric monoidal product $\boxtimes$ and unit $*$.
\end{theorem}

\begin{proof}
We need to verify that $\aM_*$ satisfies the pushout-product axiom
\cite[2.1]{Schwede-Shipley}.  Thus, for cofibrations $A \to B$ and $X
\to Y$, we must show that the map
\[
(A \boxtimes Y) \coprod_{A \boxtimes X} (B \boxtimes X) \to B
\boxtimes Y
\]
is a cofibration, and a weak equivalence if either $A \to B$ or $X \to
Y$ was.  Since $\aM_*$ is cofibrantly generated, it suffices to check
this on generating (acyclic) cofibrations.  Therefore, we can reduce
to considering a pushout-product of the form
\[
P \to (\sfree{B'}) \boxtimes (\sfree{Y'}),
\]
where $P$ is the pushout
\[
\xymatrix{
\relax \sfree{A'} \boxtimes \sfree{X'} \ar[r] \ar[d] & \relax \sfree{B'}
  \boxtimes \sfree{X'} \ar[d] \\ 
\relax \sfree{A'} \boxtimes \relax \sfree{Y'} \ar[r] & P, \\
}
\]
for $A' \to B'$ and $X' \to Y'$ generating cofibrations in $\aU$.
Using the fact that $(* \boxtimes M) \boxtimes (* \boxtimes N) \cong *
\boxtimes (M \boxtimes N)$ and the fact that $* \boxtimes (-)$ is a
left adjoint, we can bring the $* \boxtimes (-)$ outside.  Similarly,
using Proposition~\ref{propfreeprod} we can bring $\sL(1) \times (-)$
outside and rewrite as
\[
\sfree{\left( A' \times Y' \coprod_{A' \times X'} B' \times X' \to B' \times
  Y' \right)}.
\]
Finally, the pushout-product axiom for $\aU$ implies that it suffices
to show that $\sfree{-}$ preserves cofibrations and weak
equivalences.  By construction of the model structures, it follows
that $\sL(1) \times (-)$ and $* \boxtimes (-)$ preserve cofibrations.
Furthermore, $\sL(1) \times (-)$ evidently preserves weak
equivalences, and the analogous result for $* \boxtimes (-)$ follows
from the fact that the unit map $\lambda$ is a weak equivalence.
\end{proof}

The previous theorem implies that for a cofibrant $*$-module $X$, the
functor $X \boxtimes (-)$ is a Quillen left adjoint.  In particular,
we can compute the derived $\boxtimes$ product by working with
cofibrant objects in $\aM_*$.  Having established the existence of the
derived $\boxtimes$ product, we now will compare it to the cartesian
product of spaces.

The analogues in this setting of \cite[\S I.4.6]{ekmm} and \cite[\S
  II.1.9]{ekmm} yield the following helpful lemma.

\begin{lemma}\label{lemfree}
If $X$ is a cofibrant $\sL(1)$-space, then $X$ is homotopy equivalent
as an $\sL(1)$-space to a free $\sL(1)$-space $\sL(1) \times X'$,
where $X'$ is a cofibrant space.  If $Z$ is a cofibrant $*$-module,
then $Z$ is homotopy equivalent as a $*$-module to a free $*$-module
$* \boxtimes Z'$, where $Z'$ is a cofibrant $\sL(1)$-space.
\end{lemma}

We now use this to analyze the behavior of the forgetful functor on
$\boxtimes$.  Choosing a linear isometric isomorphism $f \colon U^2
\to U$, for $\sL(1)$-spaces $X$ and $Y$ there is a natural map $\alpha
\colon UX \times UY \to \sL(2) \times (X \times Y) \to X \lprod Y$.

\begin{proposition}\label{propueq}
Let $X$ and $Y$ be cofibrant $\sL(1)$-spaces.  Then the natural map
$\alpha \colon UX \times UY \to U(X \lprod Y)$ is a weak equivalence
of spaces.  Let $X$ and $Y$ be cofibrant $*$-modules.  Then the
natural map $\alpha \colon UX \times UY \to U(X \lprod Y)$ is a
weak equivalence of spaces.
\end{proposition}

\begin{proof}
Lemma~\ref{lemfree} allows us to reduce to the case of free
$\sL(1)$-spaces and free $*$-modules, and the result then follows from
Proposition~\ref{propfreeprod}.
\end{proof}

There is also a map $U(X \lprod Y) \to UX \times UY$ induced by the
universal property of the product; it follows from
Proposition~\ref{propueq} that this map is a weak equivalence as well
under the hypotheses of the proposition.

The analogous result for commutative $\boxtimes$-monoids follows from
a result of \cite[6.8]{Basterra-Mandell}.  They prove that for
cofibrant $\sL$-spaces $X$ and $Y$, the natural map $X \vee Y \to X
\times Y$ is a weak equivalence; since the coproduct in $\aM_*[\bP]$
is precisely $\boxtimes$, this implies the natural map $X \boxtimes Y
\to X \times Y$ is a weak equivalence.  To handle the case of
associative $\boxtimes$-monoids, we exploit the following analysis of
the underlying $\sL(1)$-space of a cell $\boxtimes$-monoid, following 
\cite[\S VII.6.2]{ekmm}.

\begin{proposition}\label{propunderlying}
The underlying $*$-module associated to a cell associative monoid is
cell.
\end{proposition}

To prove this, we need to briefly recall some facts about simplicial
objects in $\sL(1)$-spaces.  Recall that for the categories we are
studying there are internal and external notions of geometric
realization.  We need the following compatibility result.

\begin{lemma}
Let $X$ be a simplicial object in any of the categories $\aU[\bL]$, $\aM_*$,
$(\aU[\bL])[\bT])$, $\aM_*[\bT]$, $(\aU[\bL][\bP])$, or $\aM_*[\bP]$.
Then there is an isomorphism between the internal realization
$|X_\cdot|$ and the external realization $|UX_\cdot|$.
\end{lemma}

\begin{proof}
First, assume that $X$ is a simplicial object in $\aU[\bL]$,
$(\aU[\bL])[\bT]$, or $(\aU[\bL])[\bP]$.  In all cases, the argument
is essentially the same; we focus on $\aU[\bL]$.  For any simplicial
space $Z_\cdot$, there is an isomorphism of spaces $\bL|Z| \cong |\bL
Z|$.  Now, since $* \lprod |X_\cdot| \cong |* \lprod X_\cdot|$, the
realization of a $*$-module is a $*$-module.  The remaining parts of
the lemma follow.
\end{proof}

We now begin the proof of Proposition~\ref{propunderlying}.  

\begin{proof}
First, observe that we have the following analogue of \cite[\S
  VII.6.1]{ekmm}, which holds by essentially the same proof: for
$\boxtimes$-monoids $A$ and $B$ which are cell $*$-modules, the
underlying $*$-module of the coproduct in the category of
$\boxtimes$-monoids is a cell $*$-module.

Next, assume that $X$ is a cell $*$-module.  Then $X \boxtimes X
\boxtimes \ldots \boxtimes X$ is a cell $*$-module.  Since $X \to CX$
is cellular, the induced map $\bT X \to \bT CX$ is the inclusion of a
subcomplex.  Let $Y_n$ be a cell $*$-module and consider the pushout of
$\boxtimes$-monoids
\[
\xymatrix{
\bT X \ar[r] \ar[d] & Y_n \ar[d] \\
\bT CX \ar[r] & Y_{n+1}. \\
}
\]
By passage to colimits, it suffices to show that $Y_{n+1}$ is a cell
$*$-module.  We rely on a description of the pushout of
$\boxtimes$-monoids as the realization of a simplicial $*$-module.  By
the argument of \cite[\S XII.2.3]{ekmm}, we know that $\bT$ preserves
Hurewicz cofibrations of $*$-modules.  Moreover, $\bT$ preserves
tensors and pushouts and thus $\bT CX \cong * \coprod_{\bT X} (\bT X
\otimes I)$.  By the argument of \cite[\S VII.3.8]{ekmm}, we can
describe this pushout as the double mapping cylinder $M(\bT CX, \bT X,
*)$, and the argument of \cite[\S VII.3.7]{ekmm} establishes that this
double mapping cylinder is isomorphic to the realization of the
two-sided bar construction with $k$-simplices 
\[[k] \mapsto \bT CX \coprod \underbrace{\bT X \coprod \bT X \coprod
  \ldots \coprod \bT X}_k \coprod *.\]
Since the $k$-simplices of this bar construction are cell $*$-modules
and the face and degeneracy maps are cellular, the realization is
itself a cell $*$-module. 
\end{proof}

Finally, we study the functor $Q$.  Since we wish to use $Q$ to
provide a functorial rectification of associative monoids, we need to
determine conditions under which the natural map $UX \to QX$ is a weak 
equivalence.  Note that we do not expect this map to be a weak
equivalence for commutative monoids, as that would provide a
functorial rectification of $E_\infty$ spaces to commutative
topological monoids.

\begin{proposition}\label{prophomq}
Let $X$ be a cell $\sL(1)$-space.  Then the map $UX \to QX$ is
a weak equivalence of spaces.  Let $X$ be a cell $*$-module.  Then the
map $UX \to QX$ is a weak equivalence of spaces.  Finally, let $X$ be
a $\boxtimes$-monoid in $\aM_*$.  Then the map $UX \to QX$ is a weak
equivalence of spaces.
\end{proposition}

\begin{proof}
First, assume that $X$ is a cell $\sL(1)$-space.  Since colimits in
$\sL(1)$-spaces are created in $\aU$, $U$ commutes with colimits.
Moreover, since $Q$ is a left adjoint, it also commutes with
colimits.  By naturality and the fact that the maps $X_n \to X_{n+1}$
in the cellular filtration are Hurewicz cofibrations, it suffices to
consider the attachment of a single cell.  Applying $Q$ to the diagram 
\[
\xymatrix{
\sL(1) \times A \ar[r] \ar[d] & X_n \ar[d] \\
\sL(1) \times B \ar[r] & X_{n+1}, \\
}
\]
since $Q(\sL(1) \times A) \cong A$ we obtain the pushout
\[
\xymatrix{
A \ar[r] \ar[d] & QX_n \ar[d] \\
B \ar[r] & QX_{n+1}
}
\]
and therefore the result follows by induction, as $Q(X_0) = Q(*) = *$.
Next, assume that $X$ is a cell $*$-module.  In an analogous fashion,
we can reduce to the consideration of the free cells.  Since $Q$ is
strong symmetric monoidal, $Q(* \lprod (\sL(1) \times Z)) \cong *
\times Q(\sL(1) \times Z)$.  Since the unit map $\lambda$ is always a
weak equivalence, once again we induct to deduce the result.  The case
for $\boxtimes$-monoids now follows from Proposition~\ref{propunderlying}.
\end{proof}

Therefore, for cell objects in $\aM_*[\bT]$, $Q$ provides a functorial
rectification to topological monoids.   

\section{Implementing the axioms for $\sL(1)$-spaces}\label{L-spacesection}

In this section, we will show that the Lewis-May Thom spectrum functor
$\T$ restricted to the category of $*$-modules over $(* \boxtimes BG)$
satisfies our axioms.  We begin by reviewing the essential properties
of the Lewis-May functor.

\subsection{Review of the properties of $\T_{\sS}$}

As discussed in Section~\ref{Lewissection}, the Lewis-May construction of the
Thom spectrum yields a functor $\T_{\sS} \co \aU / BG \to \sS$.  If we
work in the based setting $\aT / BG$, we obtain a functor to $\sS
\backslash S$, spectra under $S$, where for $f \co X \to BG$ the unit
$S \rightarrow \T_{\sS}(f)$ is induced by the inclusion $* \rightarrow X$
over $BG$.  In this section, we review various properties of
$\T_{\sS}$ which we will need in verifying the axioms for the version
of the Thom spectrum functor we will construct in the context of
$\sL(1)$-spaces.

\begin{theorem}\label{thmthomomni}
\hspace{5 pt}
\begin{enumerate}
\item{Let $f\colon * \to BG$ be the basepoint inclusion.  Then $\T_{\sS}(f)
  \cong S$.}
\item{The functor $\T_{\sS} \colon \aU / BG \to \sS$ preserves colimits
  \cite[7.4.3]{lewis-may-steinberger}.}
\item{The functor $\T_{\sS} \colon \aT / BG \to \sS \backslash S$ preserves
  colimits.}
\item{Let $f \colon X \to BG$ be a map and $A$ a space.  Let $g$ be
  the composite 
\[\xymatrix{X \times A \ar[r]^{\pi} & X \ar[r]^{f} & BG},\]
where $\pi$ is the projection away from $A$.  Then 
\[\T_{\sS}(g) \cong A_+ \sma \T_{\sS}(f)\]
\cite[7.4.6]{lewis-may-steinberger}.}
\item{If $f : X \rightarrow BG$ and $g : X^{\prime} \rightarrow BG$
  are $T$-good maps such that there is a weak equivalence $h : X \htp X^{\prime}$
over $BG$, then there is a stable equivalence $Mf \htp Mg$ given by
the map of Thom spectra induced by $h$
\cite[7.4.9]{lewis-may-steinberger}.}
\item{If $f : X \rightarrow BG$ and $g : X \rightarrow BG$ are $T$-good
  maps which are homotopic, then there is a stable equivalence $Mf
  \htp Mg$.  However, the stable equivalence depends on the homotopy
  \cite[7.4.10]{lewis-may-steinberger}.}
\end{enumerate}
\end{theorem}

Notice that taking $A = I$, item (4) of the preceding theorem implies
that the functor $\T_{\sS}$ converts fiberwise homotopy equivalences into
homotopy equivalences in $\sS \backslash S$.  Similarly, item (3) and
(4) imply that $\T_{\sS}$ preserves Hurewicz cofibrations.  The requirement
that the maps $X \to BG$ be $T$-good that appears in the homotopy
invariance results suggest that when dealing with spaces over $BF$, a
better functor to consider might be the composite $\T_{\sS}\Gamma$.
Unfortunately, the interaction of $\Gamma$ with some of the
constructions we are interested in (notably extended powers) is
complicated; see Section~\ref{secBF} for further discussion.

Next, we review the multiplicative properties of this version of the
Thom spectrum functor.  Based spaces with actions by the linear
isometries operad $\sL$ can be regarded as algebras with respect to an
associated monad $\bC$ and spectra in $\sS \backslash S$ which are
$E_\infty$-ring spectra structured by the linear isometries operad can
be regarded as algebras with respect to an analogous monad we will
also denote $\bC$.  In order to understand the interaction of the Thom
spectrum functor with these monads, we need to describe the role of
the twisted half-smash product.  Given a map $\chi \colon X \to
\sL(U^j,U)$ and maps $f_i \colon Y_i \to BG$, $1 \leq i \leq j$, we
define a map $\chi \times \Pi_i f_i$ as the composite 
\[
\xymatrix{
\Pi_i X_i \ar[r]^{\Pi_i f_i} & \Pi_i BG \ar[r]^{\chi_*} & BG \\
}
\]
where here $\chi_*$ denotes the map induced on $BG$ by $\chi$.  Given
a subgroup $\pi \subset \Sigma_n$ such that $X$ is a $\pi$-space and
$\chi$ is a $\pi$-map, then $\chi \times \Pi_i f_i$ is a $\pi$-map
(letting $\pi$ act trivially on $BG$) and $\T_{\sS}(\chi \times \Pi_i f_i)$
is a spectrum with $\pi$-action.

The following theorem is \cite[7.6.1]{lewis-may-steinberger}.

\begin{theorem}\label{thmthomextended}
\hspace{5 pt}
\begin{enumerate}
\item{In the situation above, 
\[\T_{\sS}(\chi \times \Pi_i f_i) \cong X \thp
  (\bar{\Sma_i} \T_{\sS}(f_i)).\]}
\item{Passing to orbits, there is an isomorphism
  \[\T_{\sS}(\chi \times_{\Sigma_n} \Pi_i f_i) \cong X \thp_{\Sigma_n}
  (\bar{\Sma_i} \T_{\sS}(f_i)).\]}
\end{enumerate}
Here $\bar{\sma}$ denotes the external smash product.
\end{theorem}

The following corollary is an immediate consequence.

\begin{corollary}\label{corltol}
For $f \colon X \to BG$ a map of spaces, let $\bL f$ denote the
composite 
\[\sL(1) \times X \to \sL(1) \times BG \to BG,\]
where the last map is given by the $\sL$-space structure of $BG$.
Then $\T_{\sS}(\bL f) \cong \bL \T_{\sS}(f)$.
\end{corollary} 

Theorem~\ref{thmthomextended} is the foundation of the essential technical result that
describes the behavior of this Thom spectrum functor in the presence
of operadic multiplications.  Given a map $f \colon X \to BG$, there is an
induced monad $\bC_{BG}$ on $\aT / BG$ specified by defining $\bC_{BG}
f$ as the composite $\bC X \to \bC BG \to BG$.  Lewis proves the
following \cite[7.7.1]{lewis-may-steinberger}.

\begin{theorem}
Given a map $f : X \rightarrow BF$, there is an isomorphism
$\bC \T_{\sS}(f) \cong \T_{\sS}(\bC_{BG} f)$, and this isomorphism is coherently
compatible with the unit and multiplication maps for these monads.
\end{theorem}

This result has the following corollary; the first part of
this is one of the central conclusions of Lewis' thesis, and the
second part is a consequence explored at some length in the first
author's thesis and forthcoming paper
\cite{blumberg-thesis,blumberg-THH-thom-einf}.

\begin{theorem}\label{thmcorcolim}
\hspace{5 pt}
\begin{enumerate}
\item{The functor $\T_{\sS}$ restricts to a functor 
\[\T_{(\sS\backslash
    S)[\bC]}\colon (\aT / BG)[\bC_{BG}] \to (\sS\backslash S)[\bC].\]}
\item{The functor \[\T_{(\sS\backslash S)[\bC]} \colon (\aT /
  BG)[\bC_{BG}] \to \sS[\bC]\] preserves colimits and tensors with
  unbased spaces.}
\end{enumerate}
\end{theorem}

\subsection{Verification of the axioms}

We begin by studying the behavior of $\T_{\sS}$ in
the context of $\sL(1)$-spaces.  The $E_\infty$ space $BG$ constructed
as the colimit over the inclusions of the $\mathcal V$-space $V
\mapsto BG(V)$ is a commutative monoid for $\lprod$.  Therefore, the
category $\aU[\bL] / BG$ has a weakly symmetric monoidal product:
Given $f \colon X \to BG$ and $g \colon Y \to BG$, the product $f
\lprod g$ is defined as 
\[f \lprod g \colon X \lprod Y \to BG \lprod BG
\to BG.\]
The unit is given by the trivial map $* \to BG$.  We define
\[\T_{\sL} \colon \aU[\bL]/BG \to \sS[\bL]\]
as the Lewis-May Thom spectrum functor $\T_{\sS}$ restricted to $\aU[\bL]/
BG$: Corollary~\ref{corltol} implies that $\T_{\sS}$ takes values in
$\bL$-spectra on $\aU[\bL]/BG$.  Moreover, $\T_{\sL}$ is strong
symmetric monoidal (up to unit).

\begin{proposition}\label{propstrong}
Given $f \colon X \to BG$ and $g \colon Y \to BG$, there is a
coherently associative isomorphism $\T_{\sL}(f \boxtimes g) \cong
\T_{\sL}(f) \sma_{\sL} \T_{\sL}(g)$.
\end{proposition}

\begin{proof}
We can describe $f \boxtimes g \colon X \boxtimes Y \to BG$ as the
natural map to $BG$ associated to the coequalizer describing $X
\boxtimes Y$.  Theorem~\ref{thmthomomni} implies that
$\T_{\sL}$ commutes with this coequalizer and
Theorem~\ref{thmthomextended} implies that 
\[\T_{\sL}(\sL(2) \times \sL(1) \times \sL(1) \times (X \times Y))
\cong (\sL(2) \times \sL(1) \times \sL(1)) \thp (X \bar{\sma} Y)\]
and
\[\T_{\sL}(\sL(2) \times (X \times Y)) \cong \sL(2) \thp (X
\bar{\sma} Y).\]
Inspection of the maps then verifies that the resulting coequalizer is
precisely the the coequalizer defining $\T_{\sL}(f) \sma_{\sL}
\T_{\sL}(g)$.
\end{proof}

Now we restrict to the subcategory $\aM_*$.  For a $\boxtimes$-monoid
in $\aM_*$ with multiplication $\mu$, unitality implies that there is
a commutative diagram 
\[
\xymatrix{
X \boxtimes * \ar[r] \ar[dr]_{\lambda} & X \boxtimes X \ar[d]^\mu \\
& X. \\
}
\]  
In conjunction with Proposition~\ref{propstrong}, this diagram implies
that given a map $f \colon X \to BG$, there is an isomorphism 
\[\T_{\sL}(* \lprod X \to * \lprod BG \to BG) \cong S \sma_{\sL}
\T_{\sL}(f),\]
which implies that the following definition is sensible.  We let
$\aM_S$ denote the category of $S$-modules from \cite{ekmm}; this is
our symmetric monoidal category $\aS$ of spectra in this setting.

\begin{definition}\label{defnAthom}
The category $\aA$ is simply $M_*$.  Define $BG_{\aA}$ to be $* \lprod
BG$.  We define a Thom spectrum functor 
\[\T_\aA \colon \aA/BG_{\aA} \to \aM_S\]
given $f \colon X \to BG_{\aA}$ by applying $\T_{\sL}$ to the
composite 
\[
\xymatrix{
X \ar[r]^f & BG_{\aA} \ar[r]^\lambda & BG. \\
}
\]
\end{definition}

It follows immediately from Proposition~\ref{propstrong} and the
observation preceding the definition that $\T_{\aA}$ is a strong
symmetric monoidal functor from $\aM_* / BG_{\aA}$ to $\aM_S$.  Since
Theorem~\ref{thmcorcolim} implies that $\T_{\aA}$ commutes with
colimits and tensors, $\T_{\aA}$ commutes with geometric realization
and therefore the pair $\aA$, $\T_{\aA}$ satisfies axioms {\bf A1}
and {\bf A2}.

Next, we use Theorem~\ref{theoremmonoids} to choose a cofibrant
replacement functor $c$ on the category $\aM_*[\bT]$ such that for any
object $X$ in $\aA$, $cX$ is a cell monoid \cite{Hirschhorn}.  Note
that cell monoids are well-based; one proves this by an inductive
argument using the ``Cofibration Hypothesis'' (see \cite[3.22]{ABGHR}
for details).  Then given an object $f \colon X \to BG_{\aA}$, we
define $Cf$ to be the composite 
\[\xymatrix{cX \ar[r]_{\htp} & X \ar[r] & BG_{\aA}}.\] 

We will work with the notion of flatness encapsulated in the
definition of the class of $S$-modules $\bar{\aF_S}$
\cite[9.6]{basterra} (see also 
\cite[VII.6.4]{ekmm}).  Let $\aF_S$ denote the collection of modules
of the form $S \sma_{\sL} \sL(j) \thp_G K$ where $K$ is a $G$-spectrum
(for $G \subset \Sigma_i$) which has the homotopy type of a $G$-CW
spectrum.  Then $\bar{\aF_S}$ is the closure of $\aF_S$ under finite
$\sma$, wedges, pushouts along cofibrations, colimits of countable
sequences of cofibrations, homotopy equivalences, and
``stabilization'' (in which if $\Sigma M$ is in $\bar{\aF_S}$ then so
is $M$).  The point of this definition is that for $S$-modules $M$ and
$N$ in the class $\bar{\aF_S}$, the point-set smash product $M \sma N$
represents the derived smash product \cite[VII.6.7]{ekmm},
\cite[9.5]{basterra}.

\begin{lemma}\label{lemflat}
Let $f \colon X \to BG_{\aA}$ be an object of $\aM_*[\bT] / BG_{\aA}$
such that $X$ is a cell monoid.  Then the underlying $S$-module of the
$S$-algebra $\T_{\aA}(f)$ is in the class $\bar{\aF_S}$.
\end{lemma}

\begin{proof}
We proceed by induction.  Since $\T_{\aA}$ commutes with
colimits and preserves Hurewicz cofibrations, we can reduce to
consideration of a pushout square of the form 
\[
\xymatrix{
\T_{\aA}(\bT (\sfree{A})) \ar[r] \ar[d] & \T_{\aA}(X_n) \ar[d] \\
\T_{\aA}(\bT (\sfree{B})) \ar[r] & \T_{\aA}(X_{n+1}) \\
}
\]
where $X_n$ can be assumed to be in the class $\bar{\aF_S}$ and $A$
and $B$ are CW-complexes.  Here we are suppressing the maps from $A$,
$B$, $X_n$, and $X_{n+1}$ to $BG_{\aA}$ in our notation.  Next, since
both $\T_{\aA}$ and $\bT$ preserves Hurewicz cofibrations, it suffices
to show that for any CW-complex $Z$, $\T_{\aA}(\bT (\sfree Z))$ is in
the class $\bar{\aF_S}$.  Since $\T(Z)$ has the homotopy type of a
CW-complex (see the proof of \cite[7.5.6]{lewis-may-steinberger}), the
result follows from Proposition~\ref{propfree} and
Theorem~\ref{thmthomextended}.
\end{proof}

This completes the verification of {\bf A3}.  For {\bf A4}, we use the
functor $Q$ of Definition~\ref{defq}; we have already verified that it
has the desired properties in Proposition~\ref{prophomq}.  For {\bf
  A5}, we choose a cofibrant replacement functor on the category of
commutative monoids provided by the model structure of
Theorem~\ref{theoremmonoids} and use this to define $BG_{\aA}'$.
(Recall from the discussion prior to Proposition~\ref{propunderlying}
that $BG_{\aA}'$ has the desired properties.)

Finally, rectification is very straightforward in this context; since
Proposition~\ref{propoperadcomp} tells us that a map $X
\to BG_{\aA}$ over the non-$\Sigma$ linear isometries operad specifies
the data of 
a monoid map in $(\aU[\bL] / BG_{\aA})[\bT]$, by applying $* \boxtimes 
(-)$ we obtain a monoid map in $*$-modules.  To complete the
verification of {\bf A6}, we use Proposition~\ref{prophomq}.  Given a
topological monoid $M$, we can regard this as an $A_\infty$ space over
the non-$\Sigma$ linear isometries operad by pulling back along the
augmentation to the associative operad.  Equivalently, $M$ regarded as
an $\sL(1)$-space with trivial action is a $\lprod$-monoid with
multiplication induced from the monoid multiplication $M \times M \to
M$ and the fact that $M \lprod M \cong M \times M$ by the argument of
Lemma~\ref{lemqstrong}.

Thus, let $X$ be an $A_\infty$ space structured by the non-$\Sigma$
linear isometries operad. 
Since $Q$ is left adjoint to the functor which assigns the trivial
$\sL(1)$-action, the unit of the adjunction induces a map of
$\sL(1)$-spaces $X \lprod X \to Q(X \lprod X) \cong QX \times QX$.  To
show that axiom {\bf A6} holds, it suffices to show that this is a map
of $\lprod$-monoids; this map constructs the homotopy commutative
diagram of the axiom.  But since $QX \times QX \cong QX \lprod QX$
this map is a map of $\lprod$-monoids by the definition of the
multiplication on $QX$.

\section{Modifications when working over $BF$}\label{secBF}

In this section, we discuss the situation when working over $BF$:
there are technical complications which arise from the fact that the
projection
\[\pi_n: B(*, F(n), S^n) \rightarrow B(*, F(n), *)\]
is a universal quasifibration, with section a Hurewicz cofibration.
Quasifibrations are not preserved under pullback, and in general the
pullback of the section will not be a Hurewicz cofibration.  If the
section of $\pi_n$ could be shown to be a fiberwise cofibration,
pullback along any map would provide a section which was a fiberwise
cofibration.  Unfortunately, this seems difficult: the proof that the
section is a Hurewicz cofibration depends on the facts that that the
spaces in question are LEC and retractions between LEC spaces are
cofibrations.

\subsection{A review of the properties of $\Gamma$}

The standard solution to these issues (pioneered by Lewis) is to use
an explicit functor $\Gamma$ which replaces a map by a Hurewicz
fibration.  Since various properties of $\Gamma$ play an essential
role in our work in this section, we will review relevant details
here (see \cite[\S7.1]{lewis-may-steinberger} for a more comprehensive
treatment).

\begin{definition}
Given a map $f : X \rightarrow B$, define 
\[\Pi B = \{(\theta, r) \in B^{[0,\infty]} \times [0,\infty] \, | \,
  \theta(t) = \theta(r), t > r\}.\]
The end-point projection $\nu : \Pi B \rightarrow B$ is defined as $(\theta,
r) \mapsto \theta(r)$.  There is also the evaluation map $e_0 : \Pi B
\rightarrow B$ defined as $(\theta,r) \mapsto \theta(0)$.  Define
$\Gamma X$ to be the pullback
\[
\xymatrix{
\Gamma X \ar[d] \ar[r] & \Pi B \ar[d]^-{e_0} \\
X \ar[r]^-f & B. \\
}
\]
Let $\Gamma f$ denote the the induced map 
\[
\xymatrix{
\Gamma X \ar[r] & \Pi B \ar[r]^-{\nu} & BF.\\
}
\]
There is a map $\delta : X \rightarrow \Gamma X$ specified by taking
$x \in X$ to the pair $(x, \zeta_x)$ where $\zeta_x$ is the path of
length zero at $x$.
\end{definition}

The map $\Gamma f$ is a Hurewicz fibration.

\begin{lemma}
The maps $\delta$ specify a natural transformation $\id \rightarrow
\Gamma$, and $\delta : X \rightarrow \Gamma X$ is a homotopy
equivalence (although not a homotopy equivalence over $B$).
\end{lemma}

Moreover, $\Gamma$ has very useful properties in terms of interaction
with the naive model structures on $\aU / B$ and $\aT / B$
\cite[7.1.11]{lewis-may-steinberger}.

\begin{proposition}
The functor $\Gamma$ on spaces over $B$ takes cofibrations to
fiberwise cofibrations and homotopy equivalences over $B$ to fiberwise
homotopy equivalences.  As a functor on ex-spaces, it takes ex-spaces
with sections which are cofibrations to ex-spaces with sections which
are fiberwise cofibrations.
\end{proposition}

There is a useful related lemma.

\begin{lemma}
If $X \rightarrow X^{\prime}$ is a weak equivalence over $BF$, $\Gamma
X \rightarrow \Gamma X^{\prime}$ is a weak equivalence.
\end{lemma}

There are two possible ways we might use $\Gamma$ to resolve the
problems with $BF$; we could replace $\pi_n$ with $\Gamma \pi_n$,
which will be a Hurewicz fibration and will have section a fiberwise
cofibration, or we could replace a given map $f : X \rightarrow BF$
with a Hurewicz fibration via $\Gamma$.  The latter approach will
yield a homotopically well-behaved Thom spectrum construction, since
the pullback of a quasifibration along a Hurewicz fibration is a
quasifibration and the pullback of a section which is a cofibration
will be a cofibration.  Moreover, Lewis shows that the two approaches
yield stably equivalent Thom spectra.  Since the second approach is
much more felicitous for the study of multiplicative structures, we
will employ it exclusively.

$\Gamma$ behaves well with respect to colimits and unbased tensors
\cite[7.1.9]{lewis-may-steinberger}.

\begin{proposition}
As a functor on $\aU / BF$ and $\aT / BF$, $\Gamma$ commutes with
colimits.
\end{proposition}

There is a related result for tensors with unbased spaces (although
note however that this is false for based tensors).

\begin{proposition}
As a functor on $\aU / BF$ and $\aT / BF$, $\Gamma$ commutes with the
tensor with an unbased space $A$.
\end{proposition}

Finally, we recall some salient facts about the interaction of
$\Gamma$ with operadic multiplications.  May (\cite[1.8]{may-geom})
shows that $\Gamma$ restricts to a functor on $\aT[\bC]$, for an operad
$\aC$ augmented over the linear isometries operad, and Lewis observes
that in fact $\Gamma$ extends to a functor on $\aT[\bC] / BF$.  We
make particular use of these facts in the cases of $\bT$ and $\bP$.
A related observation we will use is that $\Gamma$ restricts to a
functor on $\aU[\bL]$.  An essential aspect of these results is that
all of the various maps associated with $\Gamma$ (notably $\delta$)
are maps of $\bC$-algebras, and so in particular the map $\delta$
yields a weak equivalence in the category $\aT[\bC]$.

\subsection{$\Gamma$ and Cofibrant replacement}

In order to verify axiom {\bf A3} in this setting, we must amend the
cofibrant replacement process.  Given a map $f : X \rightarrow BF$
regarded as a map in $(\aU[\bL])[\bT]$, we consider the map
$\Gamma (cf) : \Gamma(cX) \rightarrow BF$ obtained by
the cofibrant replacement functor in $(\aU[\bL])[\bT]$ followed by
$\Gamma$.  We have the following commutative diagram:
\[
\xymatrix{
c X \ar[r] \ar[d]^-{\delta} & X \ar[r] \ar[d]^-{\delta} & BF \\
\Gamma c X \ar[r]^-{\htp} & \Gamma X \ar[ur] \\
}.
\]

Since the labeled weak equivalence in the preceding diagram connect
objects of $(\aU[\bL])[\bT]/BF$ which are $T$-good, there is a stable
equivalence connecting $\T_{\sL}(\Gamma f)$ to
$\T_{\sL}(\Gamma(c f)))$.  If $f$ itself was $T$-good, then there is a
further stable equivalence from $\T_{\sL} f$.  Therefore 
this process does not change the homotopy type of the Thom spectrum.
Finally, we apply $* \boxtimes (-)$ to ensure that we land in
$\aM_*[\bT] / BG_{\aA}$.  Denote this composite functor by $\gamma$.

Next, we must verify that this process produces something which
allows us to compute derived functors with respect to $\boxtimes$ and
$\sma$, in $\aM_*[\bT]$ and $\aM_S[\bT]$ respectively.  Although $*
\boxtimes \Gamma(c X)$ is not a cofibrant object in
$\aM_*[\bT]$, it has the homotopy type of a cofibrant object and this
suffices to ensure that it can be used to compute the derived
$\boxtimes$ product.  This observation also implies that the functor
$Q$ satisfies axiom {\bf A4} in this context.

Moving on, we now need to show that $\T_{\aA}(* \boxtimes \Gamma(c
f))$ can be used to compute the derived smash product in $\aM_S$.

\begin{lemma}\label{lemthomunder}
Let $f : X \rightarrow BG_{\aA}$ be a map in $\aM_*[\bT]$.  Let $U$ denote
the forgetful functors $\aM_*[\bT] \rightarrow \aM_*$ and $\aM_S[\bT]
\rightarrow \aM_S$ respectively.  Then there is an isomorphism
$\T_{\aA}(Uf) \cong U\T_{\aA}(f)$.
\end{lemma}

By Lemma~\ref{lemthomunder}, it will suffice to show that the
underlying $S$-module of $\T_{\aA}(* \boxtimes \Gamma(c f))$ is
in the class $\bar{\aF_S}$.  Since a slight modification of the
argument of Proposition~\ref{propunderlying} shows that the underlying
$\sL(1)$-space of a cell $\lprod$-monoid is a cell $\sL(1)$-space, in
fact Lemma~\ref{lemthomunder} implies that it will suffice to show
that given \[f: X \rightarrow BF_{\aA}\] such that $X$ is a cell
$\sL(1)$-space, $\T_{\aA}(* \boxtimes \Gamma f)$ is in the class
$\bar{\aF_S}$.  We will make an inductive argument.  Using the fact
that $\Gamma$ commutes with colimits in $\sL(1)$-spaces and preserves
Hurewicz cofibrations, it suffices to show that the Thom spectra of $*
\boxtimes \Gamma \bT(CE_n)$ and $* \boxtimes \Gamma \bT(E_n)$ are
in the class $\bar{\aF_S}$, where $E_n$ is a wedge of cells $D^n$.  We
can further reduce to the case where we are considering a single cell.
Abusing notation by suppressing the maps to $BF$, we will refer to the
relevant Thom spectra as $\T_{\aA}(* \boxtimes \Gamma (\sL(1) \times
S^n))$ and $\T_{\aA}(* \boxtimes \Gamma (\sL(1) \times D^n))$.
Finally, since $S^n$ can be constructed as the pushout $D^n
\cup_{S^{n-1}} D^n$, it suffices to consider $\T_{\aA}(*
\boxtimes \Gamma(\sL(1) \times *))$ and $\T_{\aA}(* \boxtimes \Gamma (\sL(1)
\times D^n))$.  Recall that by Proposition~\ref{propstrong}, these
spectra are isomorphic as $S$-modules to $S \sma_{\sL}
\T_{\sL}(\Gamma (\sL(1) \times *))$ and $S \sma_{\sL}
\T_{\sL}(\Gamma (\sL(1) \times D^n))$ respectively.

\begin{lemma}\label{lemextended}
The $S$-modules $\T_{\aA}(* \boxtimes \Gamma (\sL(1) \times *))$ and
$\T_{\aA}(* \boxtimes \Gamma (\sL(1) \times D^n))$ are in the class
$\bar{\aF_S}$.
\end{lemma}

\begin{proof}
Given a map $D^n \rightarrow BF$, by choosing a point in $BF$ in the
image we obtain a map $* \rightarrow D^n$ over $BF$.  This
induces a map $\sL(1) \times * \rightarrow \sL(1) \times D^n$ which
is a weak equivalence of $\sL(1)$-spaces over $BF$.  Since these are
cofibrant $\sL(1)$-spaces, this is a homotopy equivalence over $BF$.
Applying $\Gamma$ turns this into a fiberwise homotopy equivalence.
Since $\T$ takes fiberwise homotopy equivalences to homotopy
equivalences of spectra, the resulting spectra are homotopy
equivalent.

Therefore, we are reduced to considering the Thom spectra associated
to $\Gamma(\sL(1) \times *)$ associated to the various choices of a
target for point.  But since $BF$ is path-connected, an argument
analogous to the one in the preceding paragraph allows us to show that
all of these spectra are homotopy equivalent.  Thus, it suffices to
consider the trivial map $* \rightarrow BF$.  But then $\Gamma(\sL(1)
\times *))$ is homeomorphic as a space over $BF$ to $\pi_2 : \sL(1) \times
\Gamma(*)$, where $\pi_2$ is the projection away from $\sL(1)$.
Applying $\T$ yields the Thom spectrum $\sL(1)_+ \sma \T(\Gamma(*))$.
Finally, $S \sma_{\sL} (\sL(1)_+ \sma \T(\Gamma(*))$ is in the
class $\bar{\aF_S}$; this follows from the proof of part (ii) of
\cite[7.3.7]{lewis-may-steinberger}, in which the homotopy type of the
Thom spectrum $\T(\Gamma(*))$ is explicitly described.
\end{proof}

In the previous proof, we are implicitly exploiting the ``untwisting''
proposition I.2.1 from EKMM which provides an isomorphism of spectra
$A \thp \Sigma^\infty X \cong \Sigma^\infty (A_+ \sma X)$.

Finally, we need to be able to compare $\Gamma (f \lprod f)$ to
$\Gamma f \lprod \Gamma f$.

\begin{proposition}
Let $f \colon X \rightarrow BF$ be a map of $\sL(1)$-spaces, and
assume that $X$ is a cell $\sL(1)$-space.  Then there is a weak
equivalence between $\T_{\sL} \Gamma (f \lprod f)$ and $\T_{\sL} (\Gamma f
\lprod \Gamma f)$.
\end{proposition}

\begin{proof}
Recall from Proposition~\ref{propfreeprod} that given a choice of a linear
isometric isomorphism $g \colon U^2 \to U$, there is a chain of weak
equivalences 
\[X \times X \rightarrow \sL(2) \times (X \times X) \rightarrow X
\lprod X.\] 
Moreover, these equivalences are given by maps over $BF$.

Lewis shows that there is a map $\Gamma f \times \Gamma f
\rightarrow \Gamma (f \times f)$ given by multiplication of paths
which is a weak equivalence \cite[7.5.5]{lewis-may-steinberger}.  In
addition, he shows 
that $\Gamma f \times \Gamma f$ is a good map.  This is the heart of
our comparison.  Applying $\Gamma$ to the chain of equivalences above,
we have a composite 
\[\Gamma (f \times f) \rightarrow \Gamma (\sL(2) \times (f \times f))
\rightarrow \Gamma (f \lprod f)\] 
which induces weak equivalences of Thom spectra $\T(\Gamma (f \times
f)) \rightarrow \T(\Gamma (f \lprod f))$.  On the other hand,
there is also the composite 
\[\Gamma f \times \Gamma f \rightarrow
\sL(2) \times (\Gamma f \times \Gamma f) \rightarrow \Gamma f
\lprod \Gamma f.\]
Although these are not all good maps, Lemma~\ref{lemextended} implies that
the induced maps of Thom spectra
\[g_*(\T(\Gamma f) \bar{\sma} \T(\Gamma f) \rightarrow \sL(2) \thp
(\T(\Gamma f) \bar{\sma} \T(\Gamma f)) \rightarrow \T(\Gamma f)
\lprod \T(\Gamma f)\] 
are stable equivalences.  Since
$\Gamma f \times \Gamma f \rightarrow \Gamma (f \times f)$ induces a
weak equivalence of Thom spectra, the result follows.
\end{proof}

The previous proposition gives us the following result, which allows
us to compare to a model of the free loop space which is $T$-good in
the proof of the main theorem from the axioms.

\begin{corollary}
Let $f : X \rightarrow BG_{\aA}$ be a map of $\boxtimes$-monoids.  Then there
is a weak equivalence of spectra $\T_{\aA}(N^{\cy}_\boxtimes (\gamma
f))$ and $\T_{\aA}(\gamma N^{\cy}_{\boxtimes} f)$.
\end{corollary}

Finally, the verification of axioms {\bf A5} and {\bf A6} is
unchanged.

\section{Preliminaries on symmetric spectra}\label{symmetricpreliminaries}

Let $\Sp^{\Sigma}$ be the category of topological symmetric spectra as
defined in \cite{mmss}. Thus, a symmetric spectrum $T$ is a spectrum
in which the spaces $T(n)$ come equipped with base point preserving
$\Sigma_n$-actions such that the iterated structure maps $S^m\wedge
T(n)\to T(m+n)$ are $\Sigma_{m+n}$-equivariant. It is proved in \cite{mmss} that $\Sp^{\Sigma}$ has a stable model structure which makes it Quillen equivalent to the category $\Sp$ of spectra. When implementing the axiomatic framework in this setting there are two technical issues that must be addressed. The first is that the forgetful functor from $\Sp^{\Sigma}$ to $\Sp$ does not take stable model equivalences in $\Sp^{\Sigma}$ (that is, the weak equivalences in the stable model structure) to ordinary stable equivalences in $\Sp$. This can be remedied using Shipley's detection functor as we recall below. The second issue is that the symmetric Thom spectrum functor on $\U\I/BF$ does not take cofibrant replacement in the model structure on $\I\U$ to cofibrant replacement in $\Sp^{\Sigma}$. For this reason we shall introduce explicit ``flat replacement'' functors on $\I\U$ and $\Sp^{\Sigma}$ which are strictly compatible with the symmetric Thom spectrum functor. In order for this to be useful, we must of course verify that the topological Hochschild homology of a symmetric ring spectrum is represented by the cyclic bar construction of its flat replacement; this is the content of Propositions \ref{flatTH} and  
\ref{flatringprop}.  

\subsection{The detection functor}\label{detectionfunctor}
A map of symmetric spectra whose underlying map of spectra is an ordinary stable equivalence is also a stable model equivalence in $\Sp^{\Sigma}$. It is a subtle property of the stable model structure on $\Sp^{\Sigma}$ that the converse does not hold; there are stable model equivalences in $\Sp^{\Sigma}$ whose underlying maps of spectra are not stable equivalences. 
In order to characterize the stable model equivalences in terms of ordinary stable equivalences, Shipley \cite{shipleyTHH} has defined an explicit ``detection'' functor $D\co
\Sp^{\Sigma}\to \Sp^{\Sigma}$. This functor takes a symmetric
spectrum $T$ to the symmetric spectrum $DT$ with $n$th space
\[
DT(n)=\hocolim_{\mathbf m\in \I}\Omega^m(T(m)\wedge S^n).
\]
Here we tacitly replace the spaces in the
definition of $DT$ by spaces that are well-based, for example the
realization of their singular simplicial complexes. It then follows as
in \cite[3.1.2]{shipleyTHH}, that a map of symmetric spectra is a
stable model equivalence if and only if applying $D$ gives an ordinary
stable equivalence of the underlying spectra. Furthermore, by
\cite[3.1.6]{shipleyTHH}, the functor $D$ is related to the identity
functor on $\Sp^{\Sigma}$ by a chain of natural stable model
equivalences of symmetric spectra. 

\subsection{The flatness condition for symmetric spectra}\label{flatnesscondition}
It follows from \cite{shipleyTHH}, that if $T$ is a cofibrant symmetric ring spectrum, then the cyclic bar construction $B^{\cy}(T)$ represents the topological Hochschild homology of $T$.  
However, in the study of Thom spectra we find it useful to introduce the notion of a \emph{flat symmetric spectrum}, which is a more general type of symmetric spectrum for which the smash product is homotopically well-behaved.
 We first consider flat symmetric spectra in general and then define what we mean by a flat symmetric ring spectrum. For this we need to recall some convenient notation from \cite{schlichtkrull-thom}. In the following $T$ denotes a symmetric spectrum and $\I$ is the category of finite sets and injective maps defined in the introduction.
Given a morphism $\alpha\co \mathbf m\to\mathbf n$, we write $\mathbf n-\alpha$ for the set 
$\mathbf n-\alpha(\mathbf m)$ and $S^{n-\alpha}$ for the one-point compactification of 
$\mathbb R^{\mathbf n-\alpha}$. Associated to $\alpha$ we have the composite map
$$
S^{n-\alpha}\wedge T(m)\to S^{n-m}\wedge T(m)\to T(n)\xr{\bar \alpha} T(n),
$$
where the first map is the homeomorphism induced by the ordering of 
$\mathbf n-\alpha$ inherited from $\mathbf n$, the second map is the structure map of the symmetric spectrum, and $\bar \alpha$ is the extension of $\alpha$ to a permutation which is order preserving on the complement of $\mathbf m$.   
The advantage of this notation is that it will make some of our constructions self-explanatory.   
Consider for each object 
$\mathbf n$ the $\I/\mathbf n$-diagram of based spaces that to an object $\alpha\co\mathbf m\to\mathbf n$ associates $S^{n-\alpha}\wedge T(m)$. 
If $\beta\co (\mathbf m,\alpha)\to(\mathbf m',\alpha')$ is a morphism in $\I/\mathbf n$, then 
$\alpha=\alpha'\circ\beta$ by definition, and $\beta$ specifies a canonical homeomorphism between 
$S^{n-\alpha}$ and $S^{n-\alpha'}\wedge S^{m'-\beta}$. The induced map is then defined by
$$
S^{n-\alpha}\wedge T(m)\xr{\sim}S^{n-\alpha'}\wedge S^{m'-\beta}\wedge T(m)\to S^{n-\alpha'}\wedge T(m').
$$
Applying this functor to a commutative diagram in $\I$ of the form
\begin{equation}\label{Isquare}
\begin{CD}
\mathbf m@>\alpha_1>> \mathbf n_1\\
@VV\alpha_2V @VV\beta_1V\\
\mathbf n_2 @>\beta_2>> \mathbf n,
\end{CD}
\end{equation}
we get a commutative diagram of based spaces
\begin{equation}\label{spacesquare}
\begin{CD}
S^{n-\gamma}\wedge T(m)@>>> S^{n-\beta_1}\wedge T(n_1)\\
@VVV @VVV\\
S^{n-\beta_2}\wedge T(n_2)@>>> T(n),
\end{CD}
\end{equation}
where $\gamma$ denotes the composite $\beta_1\circ\alpha_1=\beta_2\circ\alpha_2$.
We say that $T$ is \emph{flat} if each of the spaces $T(n)$ is well-based, and if for each diagram 
(\ref{Isquare}), such that the intersection of the images of $\beta_1$ and $\beta_2$ equals the image of $\gamma$, the induced map
$$
S^{n-\beta_1}\wedge T(n_1)\cup_{S^{n-\gamma}\wedge T(m)}S^{n-\beta_2}\wedge T(n_2)\to T(n)
$$
is an $h$-cofibration (that is, it has the homotopy extension property in the usual sense). By Lillig's union theorem for $h$-cofibrations \cite{lillig}, a levelwise well-based symmetric spectrum $T$ is flat if and only if (i) any morphism $\alpha\co\mathbf m\to\mathbf n$ in $\I$ induces 
an $h$-cofibration $S^{n-\alpha}\wedge T(m)\to T(n)$, and (ii) for any diagram of the form 
(\ref{Isquare}), satisfying the above condition, the intersection of the images of 
$S^{n-\beta_1}\wedge T(n_1)$ and 
$S^{n-\beta_2} \wedge T(n_2)$ equals the image of $S^{n-\gamma}\wedge T(m)$. This can be reformulated in a way that is easier to check in practice.
\begin{lemma}
A levelwise well-based symmetric spectrum is flat if and only if the structure maps $S^1\wedge T(n)\to T(n+1)$ are $h$-cofibrations and the diagrams
$$
\begin{CD}
S^l\wedge T(m)\wedge S^n@>>> S^l\wedge T(m+n)\\
@VVV @VVV\\
T(l+m)\wedge S^n @>>> T(l+m+n)
\end{CD}
$$
are pullback diagrams for all $l$, $m$ and $n$.\qed
\end{lemma}
Here the notation is supposed to be self-explanatory. For instance, the symmetric Thom spectra $MO$, $MF$, obtained by applying the Thom space functor levelwise, are  flat.
We shall now prove that smash products of flat symmetric spectra are homotopically well-behaved. Recall from \cite{schlichtkrull-thom} that the smash product of a family of symmetric spectra $T_1,\dots,T_k$ may be identified with the symmetric spectrum whose $n$th space is the colimt 
$$
T_1\wedge\dots\wedge T_k(n)=\colim_{\alpha\co\mathbf n_1\sqcup\dots\sqcup \mathbf n_k\to\mathbf n}S^{n-\alpha}\wedge T(n_1)\wedge\dots\wedge T(n_k).
$$
Here the colimit is over the comma category $\sqcup^k/\mathbf n$, where 
$\sqcup^k\co\I^k\to \I$ denotes the iterated monoidal product. We introduce a  homotopy invariant version by the analogous based homotopy colimit construction,
$$
T_i\wedge^h\dots\wedge^h T_k(n)=\hocolim_{\alpha\co\mathbf n_1\sqcup\dots\sqcup \mathbf n_k\to\mathbf n}S^{n-\alpha}\wedge T(n_1)\wedge\dots\wedge T(n_k).
$$ 
As we shall see in Proposition \ref{hsmashproposition}, the latter construction always represents the ``derived'' homotopy type of the smash product for symmetric spectra that are levelwise well-based.  
\begin{proposition}\label{flatsmashequivalence}
If the symmetric spectra $T_1,\dots,T_k$ are flat, then the canonical projection 
$$
T_1\wedge^h\dots\wedge^h T_k\to
T_1\wedge\dots\wedge T_k
$$ 
is a levelwise equivalence.
\end{proposition}
\begin{proof}
For notational reasons we only carry out the proof for a pair of flat symmetric spectra $T_1$ and $T_2$. The proof in the general case is completely analogous. 
Let $\mathcal A(n)$ be the full subcategory of $\sqcup /\mathbf n$ whose objects $\alpha\co\mathbf n_1\sqcup\mathbf n_2\to\mathbf n$ are such that the restrictions to $\mathbf n_1$ and 
$\mathbf n_2$ are order preserving. Since this is a skeleton subcategory, it suffices to show that the canonical map 
$$
\hocolim_{\mathcal A(n)}S^{n-\alpha}\wedge T_1(n_1)\wedge T_2(n_2)\to \colim_{\mathcal A(n)}S^{n-\alpha}\wedge T_1(n_1)\wedge T_2(n_2)
$$
is a weak homotopy equivalence for each $n$. Notice that $\mathcal A(n)$ may be identified with the partially ordered set of pairs $(U_1,U_2)$ of disjoint subsets of $\mathbf n$, so that we may write the diagram  in the form
$$
Z(U_1,U_2)=S^{\mathbf n-U_1\cup U_2}\wedge T_1(U_1)\wedge T_2(U_2).
$$
Notice also, that since the base points are non-degenerate and the categories 
$\mathcal A(n)$ are contractible, it suffices to consider the unbased
homotopy colimit instead of the based homotopy colimit. We now use
that the categories $\mathcal A(n)$ are very small in the sense of
\cite[\S 10.13] {dwyer-spalinski}. By general model theoretical
arguments using the Str\o m model category structure \cite{strom} on
$\mathcal U$,  we are therefore left with showing that the canonical
map 
\[
\colim_{(U_1,U_2)\subsetneq(U_1^0,U_2^0)}Z(U_1,U_2)\to Z(U_1^0,U_2^0)
\] 
is an $h$-cofibration for each fixed object $(U_1^0,U_2^0)$. Since the structure maps in the 
$\mathcal A(n)$-diagram $Z$ are $h$-cofibrations and since $h$-cofibrations are closed inclusions, we may view each of the spaces $Z(U_1,U_2)$ as a subspace of $Z(U_1^0,U_2^0)$. By the flatness assumptions on $T_1$ and $T_2$ we then have the equality 
$$
Z(U_1,U_2)\cap Z(V_1,V_2)=Z(U_1\cap V_1,U_2\cap V_2)
$$
for each pair of objects $(U_1,U_2)$ and $(V_1,V_2)$. Thus, it follows from the pasting lemma for maps defined on a union of closed subspaces that the colimit in question may be identified with the union of the subspaces $Z(U_1,U_2)$. The conclusion now follows from an inductive argument using Lillig's union theorem for $h$-cofibrations \cite{lillig}. 
\end{proof}

\begin{proposition}\label{hsmashproposition}
If $T_1,\dots,T_k$ are levelwise well-based symmetric spectra, then there is a chain of 
levelwise equivalences
$$
T_1\wedge^h\dots\wedge^h T_k\xl{\sim}T'_1\wedge^h\dots\wedge^h T'_k\xr{\sim}
T'_1\wedge\dots\wedge T'_k,
$$
where $T'_i\to T_i$ are cofibrant replacements in the stable model structure on symmetric spectra.
\end{proposition}
\begin{proof}
It follows from \cite[9.9]{mmss} that we may choose cofibrant symmetric spectra $T_i'$ and levelwise acyclic fibrations  $T_i'\xr{\sim} T_i$. The left hand map is then a levelwise equivalence since homotopy colimits preserve termwise equivalences of well-based diagrams. That the right hand map is an equivalence follows from Proposition \ref{flatsmashequivalence} since cofibrant symmetric spectra are retracts of relative cell-complexes by \cite{mmss}, hence in particular flat.  
\end{proof}

Combining these propositions we get the following corollary which states that smash products of flat symmetric spectra represent the ``derived'' smash products. 
\begin{corollary}\label{flatderived}
If $T_1,\dots,T_k$ are flat symmetric spectra, then there is a levelwise equivalence
$$
T_1'\wedge\dots\wedge T_k'\xr{\sim}T_1\wedge\dots\wedge T_k,
$$
where $T_i'\to T_i$ are cofibrant replacements in the stable model structure on symmetric spectra.\qed
\end{corollary}

In the following definition we use the notion of an $h$-cofibration introduced in Section 
\ref{realizationsection}. 

\begin{definition}\label{flatsymmetricring}
A flat symmetric ring spectrum $T$ is a symmetric ring spectrum whose underlying symmetric spectrum is flat and whose unit $S\to T$ is an $h$-cofibration. 
\end{definition}

\begin{proposition}\label{flatTH}
If $T$ is a flat symmetric ring spectrum, then $B^{\cy}(T)$ represents the topological Hochschild homology of $T$.
\end{proposition}
\begin{proof}
Let $T'\to T$ be a cofibrant replacement of $T$ as a symmetric ring
spectrum. Then it follows from \cite{shipleyTHH} that
$B^{\cy}(T')$ represents the topological Hochschild of $T$. Using
Corollary \ref{flatderived}, we see that the induced map of simplicial
symmetric spectra  
$B^{\cy}_{\bullet}(T')\to B^{\cy}_{\bullet}(T)$ is a levelwise equivalence in each simplicial degree. 
Furthermore, the assumption that the unit be an $h$-cofibration implies by Lemma 
\ref{goodlemma} that these are good simplicial spaces. Therefore the topological realizations are also levelwise equivalent.
\end{proof}

\subsection{Flat replacement of symmetric spectra}\label{symmetricflatreplacement}
We define an endofunctor on the category of symmetric spectra by associating to a symmetric spectrum $T$ the symmetric spectrum $\bar T$ defined by
$$
\bar T(n)=\hocolim_{\alpha\co\mathbf n_1\to\mathbf n}S^{n-\alpha}\wedge T(n_1),
$$
where the (based) homotopy colimit is over the category $\I/\mathbf n$. This is not quite a flat replacement functor since $\bar T$ need not be levelwise well-based. However, we do have the following.
\begin{proposition}\label{symmetricflatprop}
If $T$ is a symmetric spectrum that is levelwise well-based, then
$\bar T$ is flat and the canonical projection $\bar T\to T$ is a
levelwise weak equivalence. 
\end{proposition}

\begin{proof}
We show that $\bar T$ satisfies the flatness conditions (i) and (ii). Thus, given a morphism $\alpha\co\mathbf m\to\mathbf n$, we claim that the structure map 
$S^{n-\alpha}\wedge \bar T(m)\to \bar T(n)$ is an $h$-cofibration. Let
$\alpha_*\co\I/\mathbf m\to\I/\mathbf n$ be the functor induced by
$\alpha$. Using that based homotopy colimits commute with smash
products and that there is a natural isomorphism of
$\I/\mathbf m$-diagrams 
\[
S^{n-\alpha}\wedge S^{m_1-\beta}\wedge T(m_1)\cong S^{n-\alpha\beta}\wedge T(m_1),
\qquad \beta\co \mathbf m_1\to \mathbf m,
\] 
 the map in question may be identified with the map of homotopy colimits
\[
\hocolim_{\beta\co\mathbf m_1\to\mathbf m}S^{n-\alpha\beta}\wedge
T(m_1)\\
 \xr{\alpha_*} \hocolim_{\gamma\co\mathbf n_1\to\mathbf n}S^{n-\gamma}\wedge T(n_1).  
\]
 Notice that $\alpha_*$ induces an isomorphism of $\I/\mathbf m$ onto a full subcategory of 
 $\I/\mathbf n$. The claim therefore follows from the general fact that the map of homotopy colimits obtained by restricting a diagram to a full subcategory is an $h$-cofibration, see  
e.g.\ \cite[X.3.5]{ekmm}. In order to verify (ii) one first checks the condition in each simplicial degree of the simplicial spaces defining the homotopy colimits. The result then follows from the fact that topological realization preserves pullback diagrams. The map $\bar T\to T$ is defined by the canonical projection from the homotopy colimit to the colimit
$$
\bar T(n)=\hocolim_{\I/\mathbf n}T\to \colim_{\I/\mathbf n}T=T(n),
$$
where the identification of the colimit comes from the fact that $\I/\mathbf n$ has a terminal object. The existence of a terminal object implies that this is a weak homotopy equivalence.
\end{proof}

The relationship between the replacement functor and the $\wedge^h$-product is recorded in the following proposition whose proof we leave with the reader.  
\begin{proposition}
There is a natural isomorphism of symmetric spectra
$$
\bar T_1\wedge\dots\wedge\bar T_k\cong T_1\wedge ^h\dots\wedge^h T_k.
\eqno\qed
$$
\end{proposition}

Recall the notion of a monoidal functor from \cite[\S XI.2]{maclane}. This is what is sometimes called lax monoidal. 

\begin{proposition}\label{flatmonoidalprop}
The replacement functor $T\mapsto \bar T$ is (lax) monoidal and the
canonical map $\bar T\to T$ is a monoidal natural transformation. 
\end{proposition}

\begin{proof}
The replacement of the sphere spectrum has 0th space $\bar S(0)=S^0$
and we let $S\to\bar S$ be the unique map of symmetric spectra that is
the identity in degree 0. We must define an associative and unital
natural transformation $\bar T_1\wedge \bar T_2\to\overline{T_1\wedge
  T_2}$, which by the universal property of the smash product amounts
to an associative and unital natural transformation of $\I_S\times
\I_S$-diagrams 
\begin{equation}\label{Hcofibrantmonoidalmap}
\bar T_1(m)\wedge \bar T_2(n)\to \overline{T_1\wedge T_2}(m+n).
\end{equation}
Here $\I_S$ denotes the topological category such that $\Sp^{\Sigma}$
may be identified with the category of based $\I_S$-diagrams, see
\cite{mmss} and \cite{schlichtkrull-thom}.  Consider the natural
transformation  of $\I/\mathbf m\times \I/\mathbf n$-diagrams that to
a pair of objects $\alpha\co \mathbf m_1\to\mathbf m$ and
$\beta\co\mathbf n_1\to\mathbf n$ associates the map  
\[
S^{m-\alpha}\wedge T_1(m_1)\wedge S^{n-\beta}\wedge T_2(n_1)
\to S^{m+n-\alpha\sqcup\beta}\wedge T_1\wedge T_2(m_1+n_1),
\]
where we first permute the factors and then apply the universal map to
the smash product $T_1\wedge T_2$. Using that based homotopy colimits
commute with smash products, the map (\ref{Hcofibrantmonoidalmap}) is
the induced map of homotopy colimits, followed by the map
\[
\hocolim_{\I/\mathbf m\times\I/\mathbf n}S^{m+n-\alpha\sqcup \beta}\wedge 
T_1\wedge T_2(m_1+n_1)\to\overline{T_1\wedge T_2}(m+n)
\]
 induced by the concatenation functor $\I/\mathbf
m\times\I/\mathbf n\to\I/\mathbf m\sqcup\mathbf n$.  With this
definition it is clear that the natural transformation $\bar T\to T$
is monoidal. 
\end{proof}
It follows from this that the replacement functor induces a functor on
the category of symmetric ring spectra. In order to ensure that the unit is an $h$-cofibration we adapt the usual method for replacing a topological monoid with one that is well-based. Thus, given a symmetric ring spectrum $T$, we define $T'$ to be the mapping cylinder
\[
T'=T\cup_{S\wedge\{0\}_+}S\wedge I_+,
\]
where we view the unit interval $I$ as a multiplicative topological monoid.  
This is again a symmetric ring spectrum and arguing as in the case of a topological monoid 
\cite[A.8]{may-geom} one deduces that the unit $S\to T'$ is an $h$-cofibration and that the canonical map of symmetric ring spectra $T'\to T$ is a homotopy equivalence. It is easy to check that if $T$ is flat as a symmetric spectrum, then $T'$ is a flat symmetric
ring spectrum in the sense of Definition \ref{flatsymmetricring}. Combining these remarks with Proposition \ref{symmetricflatprop}, we get the following. 

\begin{proposition}\label{flatringprop}
If $T$ is a symmetric ring spectrum that is levelwise well-based,
then $T^c=(\bar T)'$ is a flat symmetric ring
spectrum.\qed
\end{proposition}  

\section{Implementing the axioms for symmetric spectra}
\label{I-spacesection}
In this section we verify the axioms {\bf A1}--{\bf A6} in the setting of $\I$-spaces and symmetric spectra. The basic reference for this material is the paper \cite{schlichtkrull-thom} in which the theory of Thom spectra is developed in the category of symmetric spectra.
\subsection{Symmetric spectra and $\I$-spaces}
In the axiomatic framework set up in Section \ref{axiomssection} we define $\mathcal S$ to be the category of symmetric spectra $\Sp^{\Sigma}$ and we say that a symmetric ring spectrum is flat if it satisfies the conditions in Definition \ref{flatsymmetricring}.  We define $U$ to be the composite functor
$$
U\co \Sp^{\Sigma}\xr{D}\Sp^{\Sigma}\to \Sp,
$$ 
where $D$ is the detection functor from Section \ref{detectionfunctor} and the second arrow represents the obvious forgetful functor. It follows from the discussion in Section 
\ref{detectionfunctor} that a map of symmetric spectra is a stable model equivalence if and only if applying $U$ gives an ordinary stable equivalence of spectra.   

We define $\mathcal A$ to be the symmetric monoidal category $\I\mathcal U$ of $\I$-spaces as defined in the introduction. Given an $\I$-space $B$, let $B[n]$ be the $\I$-space 
$B(\mathbf n\sqcup-)$ obtained by composing with the ``shift'' functor $\mathbf n\sqcup-$ on $\I$. We write $\Hom_{\I}(A,B)$ for the internal Hom-object in $\I\mathcal U$ defined by
$\mathbf n\mapsto \I\mathcal U(A,B[n])$.  
This makes $\I\mathcal U$ a closed symmetric monoidal topological category in the sense that there is a natural isomorphism
$$
\I\mathcal U(A\boxtimes B,C)\cong\I\mathcal U(A,\Hom_{\I}(B,C)).
$$ 
We define an \emph{$\I$-space monoid} to be a monoid in $\I\mathcal
U$. The tensor of an $\I$-space $A$ with a space $K$ is given by the obvious levelwise cartesian product $A\times K$. Associating to an $\I$-space its homotopy colimit over $\I$ defines the functor $U\co \I\mathcal U\to \mathcal U$, that is, $UA=\hocolim_{\mathcal I}A$. 
Here we adapt the Bousfield-Kan construction \cite{bousfield-kan}, such that by
definition $UA$ is the realization of the simplicial space 
\[
[k]\mapsto\coprod_{\mathbf n_0\leftarrow\dots\leftarrow \mathbf n_k}A(n_k),
\] 
where the coproduct is indexed over the nerve of $\I$. In particular,
$U*$ is the classifying space $B\mathcal I$ which is contractible
since $\mathcal I$ has an initial object.  That $U$ preserves tensors
and colimits follows from the fact that topological realization has
this property. As in \cite{schlichtkrull-thom} and \cite{schlichtkrull-units} we also use the notation $A_{h\I}$ for the homotopy colimit of an $\I$-space $A$. 

In the discussion of the axioms \textbf{A1}--\textbf{A6} we consider
two cases corresponding to the underlying Thom spectrum functor on
$\mathcal U/BF$ and its restriction to $\mathcal U/BO$. In the case of
\textbf{A2} and \textbf{A3} we formulate a slightly weaker version of
the axioms which hold in the $\I\mathcal U/BF$ case and which imply
the original axioms when restricted to objects in $\I\mathcal
U/BO$. In Section \ref{generalsymmetric} we then provide additional
arguments to show why the weaker form of the axioms suffices to prove
the statement in Theorem~\ref{maintheorem}. One can also verify the axioms
for more general families of subgroups of the topological monoids $F(n)$. We omit the details of this since the only reason for singling out the group valued case is
to explain how the arguments simplify in this situation.  
  
\subsection*{Axiom {\bf A1}}
As explained in Section \ref{Lewissection}, the correspondence $\mathbf n\mapsto BF(n)$ defines a commutative monoid in $\I\U$.  In order to be consistent with the notation used in 
\cite{schlichtkrull-thom}, we now redefine $BF$ to be this $\I$-space monoid  and we write 
$BF_{h\I}$ for its homotopy colimit. Let $\mathcal N$ be the subcategory of $\I$ whose only morphisms are the subset inclusions and let $i\co \mathcal N\to\I$ be the inclusion. Thus, $\mathcal N$ may be identified with the ordered set of natural numbers. Let $BF_{h\mathcal N}$ be the homotopy colimit and $BF_{\mathcal N}$ the colimit of the $\mathcal N$-diagram $BF$, such that $BF_{\mathcal N}$
is now what was denoted $BF$ in Section \ref{Lewissection}. We then have a diagram of weak homotopy equivalences
$$
BF_{h\I}\xl{i} BF_{h\mathcal N}\xr{t}BF_{\mathcal N},
$$
where $t$ is the canonical projection from the homotopy colimit to the colimit. Here $i$ is a weak homotopy equivalence by B\"okstedt's approximation lemma \cite[2.3.7]{madsen},
and $t$ is a weak homotopy equivalence since the structure maps are $h$-cofibrations. Using that $BF_{h\I}$ has the homotopy type of a CW-complex, we choose a homotopy inverse $j$ of $i$ and define $\zeta$ to be the composite weak homotopy equivalence
$$
\zeta\co BF_{h\I}\xr{j} BF_{h\mathcal N}\xr{t} BF_{\mathcal N}.
$$
Starting with the commutative $\I$-space monoid $\mathbf n\mapsto BO(n)$, we similarly get a weak homotopy equivalence $\zeta\co BO_{h\I}\xr{\sim} BO_{\mathcal N}$. 

\subsection*{Axiom {\bf A2}}
The symmetric Thom spectrum functor 
$$
T\co \I\mathcal U/BF\to \Sp^{\Sigma}
$$
is defined by applying the Thom space functor from Section \ref{Lewissection} levelwise: given an object $\alpha\co A\to BF$ with level maps $\alpha_n\co A(n)\to BF(n)$, the $n$th space of the symmetric spectrum $T(\alpha)$ is given by $T(\alpha_n)$. Since colimits and tensors in 
$\I\mathcal U/BF$ and 
$\Sp^{\Sigma}$ are formed levelwise, the fact that the Thom space functor preserves these constructions \cite[IX]{lewis-may-steinberger} implies that the symmetric Thom spectrum functor has the same property. Given maps of $\I$-spaces $\alpha\co A\to BF$ and $\beta\co B\to BF$, we may view the canonical maps
$$
A(m)\times B(n)\to A\boxtimes B(m+n) 
$$ 
as maps over $BF(m+n)$ and since the Thom space functor takes cartesian products to smash products, the induced maps of Thom spaces take the form
$$
T(\alpha)(m)\wedge T(\beta)(n)\to T(\alpha\boxtimes\beta)(m+n).
$$
We refer to \cite{schlichtkrull-thom} for a proof of the fact that the induced map of symmetric spectra
$$
T(\alpha)\wedge T(\beta)\to T(\alpha\boxtimes\beta)
$$
is an isomorphism. This implies that $T$ is a strong symmetric monoidal functor. 
We next discuss homotopy invariance. Applying the usual (Hurewicz) fibrant replacement functor levelwise we get an endofunctor $\Gamma$ on $\I\mathcal U/BF$. It is proved in 
\cite{schlichtkrull-thom} that the composite functor 
$$
T\Gamma\co \I\U/BF\xr{\Gamma} \I\U/BF\xr{T} \Sp^{\Sigma}
$$
is a homotopy functor: if $(A,\alpha)\to (B,\beta)$ is a weak equivalence over $BF$ (that is,  
$A_{h\I}\to B_{h\I}$ is a weak homotopy equivalence), then 
$T\Gamma(\alpha)\to T\Gamma(\beta)$ is a stable model equivalence. We say that an object 
$(A,\alpha)$ is \emph{$T$-good} if the canonical map $T(\alpha)\to T\Gamma(\alpha)$ is a stable model equivalence. It follows from the definition that the symmetric Thom spectrum functor preserves weak equivalences on the full subcategory of $T$-good objects in $\I\mathcal U/BF$. Since $\I\U/BO$ maps into the subcategory of $T$-good objects in $\I\U/BF$, this in particular implies that $T$ preserves weak equivalences when restricted to $\I\mathcal U/BO$.

\begin{proposition}\label{restrictedA2}
Restricted to the full subcategory of $T$-good objects, the two compositions in the diagram
\[
\xymatrix{
\I\U/BF \ar[rr]^{T}\ar[d]^{U} & & \Sp^{\Sigma}\ar[d]^U\\
\U/BF_{h\I}\ar[r]^{\zeta_*} & \U/BF_{\mathcal N} \ar[r]^{T\Gamma} & \Sp
}
\]
are related by a chain of natural stable equivalences.
\end{proposition}
Again this implies the statement in \textbf{A2} when restricted to $\I\mathcal U/BO$. We postpone the proof of this result until we have introduced the functor $R$ in \textbf{A6}.

\subsection*{Flat replacement and Axiom {\bf A3}}\label{symmetricA3section}
We first recall the flatness notion for $\I$-spaces introduced in \cite{schlichtkrull-doldthom}. 
Thus, an $\I$-space $A$ is \emph{flat} if for any diagram of the form 
(\ref{Isquare}), such that the intersection of the images of $\beta_1$ and $\beta_2$ equals the image of $\gamma$,  the induced map
$$
A(n_1)\cup_{A(m)}A(n_2)\to A(n)
$$
is an $h$-cofibration. By Lillig's union theorem for $h$-cofibrations \cite{lillig}, this is equivalent to the requirement that (i) any morphism $\alpha\co\mathbf m\to\mathbf n$ in $\I$ induces an $h$-cofibration $A(m)\to A(n)$, and (ii) that the intersection of the images of $A(n_1)$ and $A(n_2)$ in $A(n)$ equals the image of $A(m)$. For instance, the $\I$-spaces $BO$ and $BF$ are flat. 

The \emph{flat replacement} of an $\I$-space $A$ is the $\I$-space $\bar A$ defined by 
$$
\bar A(n)=\hocolim_{\alpha\co\mathbf n_1\to\mathbf n}A(n_1),
$$
where the homotopy colimit is over the category $\I/\mathbf n$. We have the following $\I$-space analogues of Propositions \ref{symmetricflatprop} and \ref{flatmonoidalprop}. 

\begin{proposition}\label{Iflatprop}
The $\I$-space $\bar A$ is flat and the canonical projection $\bar A\to A$ is a levelwise equivalence.\qed
\end{proposition}

\begin{proposition} \label{I-spacemonoidalreplacement}
The flat replacement functor is a monoidal functor on $\mathcal I\mathcal U$ and $\bar A\to A$ is a monoidal natural transformation.\qed
\end{proposition}

It follows from this that the flat replacement functor induces a functor on the category of 
$\I$-space monoids. As in the axiomatic framework from Section \ref{axiomssection} we say that an $\I$-space monoid is \emph{well-based} if the unit $*\to A$ is an $h$-cofibration. 
Let $I$ be the unit interval, thought of as a based topological monoid with base point 0 and unit 1. Given a $\I$-space monoid $A$, let $A'=A\vee I$ be the  $\I$-space monoid defined by the levelwise wedge products $A'(n)=A(n)\vee I$. Notice that if $A$ is commutative, then so is $A'$. This construction is the $\I$-space analogue of the usual procedure for replacing a topological monoid by one that is well-based. Arguing as in the case of a topological monoid  \cite[A.8]{may-geom}, one deduces the following. 

\begin{proposition}\label{well-basedreplacement}
Let $A$ be an $\I$-space monoid. The associated $\I$-space monoid $A'=A\vee I$ is well-based and the canonical map of monoids $A'\to A$ is a  homotopy equivalence.\qed
\end{proposition}

We now define the functor $C$ in \textbf{A3} by
$$
C\co \I\U[\bT]\to \I\U[\bT],\quad A\mapsto A^c=(\bar A)'
$$
and we define the natural transformation $A^c\to A$ to be the composition of the levelwise weak homotopy equivalences $(\bar A)'\to \bar A\to A$. Given a monoid morphism $\alpha\co A\to BF$, we write $\alpha^c$ for the composition
$$
\alpha^c\co A^c\xr{} A\xr{\alpha} BF.
$$
We need a technical assumption to ensure that the associated symmetric spectrum 
$T(\alpha^c)$ is levelwise well-based. In general, we say that a map of $\I$-spaces $\alpha\co A\to BF$ \emph{classifies a well-based $\I$-space} over $A$ if the $h$-cofibrations $BF(n)\to EF(n)$ pull back to $h$-cofibrations via $\alpha$. This condition is automatically satisfied if $\alpha$ factors through $BO$, see \cite[IX]{lewis-may-steinberger}. Thus, restricted to such morphisms the following proposition verifies \textbf{A3} in the $BO$ case. Recall the functor $T\mapsto T^c$ from Proposition \ref{flatringprop}.   
\begin{proposition}\label{flatcompatibility}
There is an isomorphism of symmetric ring spectra 
$$
T(\alpha^c)\cong T(\alpha)^c
$$ 
and if $\alpha$ classifies a well-based $\I$-space over $A$, then $T(\alpha^c)$ is a flat symmetric ring spectrum.
\end{proposition}  
\begin{proof}
Consider in general a monoid morphism $\alpha\co A\to BF$ and the induced morphisms $\alpha'\co A'\to BF$ and $\bar \alpha\co \bar A\to BF$. We claim that there are isomorphisms of symmetric ring spectra $T(\alpha')\cong T(\alpha)'$ and $T(\bar\alpha)\cong \overline{T(\alpha)}$.
This clearly gives the isomorphism in the proposition. The isomorphism for $\alpha'$ follows directly from the fact that $T$ preserves colimits. For the second isomorphism we first observe that since the Thom space functor preserves coproducts and topological realization, it also preserves homotopy colimits. Thus, we have levelwise homeomorphisms
\[
T(\bar \alpha)(n)\cong\hocolim_{\gamma\co\mathbf n_1\to\mathbf n}
T(A(n_1)\xr{\gamma\circ\alpha_{n_1}} BF(n)) 
\cong\hocolim_{\gamma\co\mathbf n_1\to\mathbf n}S^{n-\gamma}\wedge T(\alpha)(n_1)
\]
and since the last term is $\overline{T(\alpha)}(n)$ by definition, the claim follows. For the last statement in the proposition we observe that if $\alpha$ classifies a well-based $\I$-space over $A$, then the symmetric spectrum $T(\alpha)$ is levelwise well-based. The statement therefore follows from Proposition \ref{flatringprop}. 
\end{proof}  

\begin{remark}
It is proved in \cite{sagave-schlichtkrull} that there is a model structure on $\I\U$ whose weak equivalences are the maps that induce weak homotopy equivalences on homotopy colimits. 
Similarly, there are model structures on the categories of monoids and commutative monoids in 
$\I\mathcal U$. However, there are several reasons why these model structures are not directly suited for the analysis of Thom spectra. For example, it is not clear that the $T$-goodness condition on objects in $\I\mathcal U/BF$ is preserved under cofibrant replacement and the Thom spectra associated to cofibrant $\I$-spaces will not in general be cofibrant as symmetric spectra. Using the less restrictive notion of flatness introduced here also makes it clear that many $\I$-spaces, such as for example $BF$, behave well with respect to the $\boxtimes$-product even though they are not cofibrant in these model structures.
\end{remark}

We next formulate some further properties of the flat replacement functor that will be needed later. Consider the homotopy invariant version of the $\boxtimes$-product defined by 
$$
A_1\boxtimes^h\dots\boxtimes ^h A_k(n)= 
\hocolim_{\mathbf n_1\sqcup\dots\sqcup\mathbf n_k\to \mathbf n} A(n_1)\times\dots
\times A(n_1).
$$
As for the analogous constructions on symmetric spectra the flat replacement functor is closely related to the $\boxtimes^h$-product.   
\begin{proposition}\label{replacementboxtimes}
There is a natural isomorphism of $\I$-spaces 
$$
\bar A_1\boxtimes\dots\boxtimes \bar A_k\cong
A_1\boxtimes^h\dots\boxtimes ^h A_k
\eqno\qed 
$$
\end{proposition}

We also have the following $\I$-space analogue of Proposition \ref{flatsmashequivalence}. The proof is similar to but slightly easier than the symmetric spectrum version since we do not have to worry about base points here.   
\begin{proposition}[\cite{schlichtkrull-doldthom}]\label{flatboxtimesequivalence}
If the $\I$-spaces $A_1,\dots, A_k$ are flat, then the canonical projection
$$
A_1\boxtimes^h\dots\boxtimes ^h A_k\to A_1\boxtimes\dots\boxtimes  A_k
$$
is a levelwise equivalence. \qed
\end{proposition}
As in Section \ref{flatnesscondition} we conclude from this that the $\boxtimes^h$-product always represents the ``derived'' homotopy type and that the $\boxtimes$-product has the ``derived'' homotopy type for flat $\I$-spaces.

\subsection*{Axiom {\bf A4}}
We define $Q$ to be the colimit over $\I$, 
$$
Q\co\I\mathcal U\to \mathcal U,\quad QA=\colim_{\I}A.
$$
It is clear that $Q$ preserves colimits and the fact that $\mathcal U$ is closed symmetric monoidal under the categorical product implies that it also preserves tensors. Before verifying  the remaining conditions it is helpful to recall some general facts about  Kan extensions. Thus, consider in general a functor  $\phi\co\mathcal B\to\mathcal C$ between  small categories. Given a $\mathcal B$-diagram $X\co \mathcal B\to\mathcal U$, the (left) Kan extension is the functor 
$\phi_*X\co \mathcal C\to\mathcal U$ defined by
$$
\phi_*X(c)=\colim_{\phi/c}X\circ\pi_c
$$
and the homotopy Kan extension is the functor $\phi_*^hX\co \mathcal C\to\mathcal U$ defined by the analogous homotopy colimits,
$$
\phi_*^hX(c)=\hocolim_{\phi/c}X\circ\pi_c.
$$
Here $\pi_c$ denotes the forgetful functor $\phi/c\to\mathcal B$, see
e.g.\ \cite{schlichtkrull-doldthom}. The effect of evaluating the
colimits of these functors is recorded in the following lemma. 

\begin{lemma}\label{Kanextensionlemma}
There are natural isomorphisms
$$
\colim_{\mathcal C}\phi_*X\cong \colim_{\mathcal B}X,\quad \colim_{\mathcal C}\phi_*^hX\cong \hocolim_{\mathcal B}X,
$$
and the canonical projection from the homotopy colimit to the colimit
$$
\hocolim_{\mathcal C}\phi_*^hX\to \colim_{\mathcal C}\phi_*^hX
$$
is a weak homotopy equivalence.\qed
\end{lemma}

An explicit proof of the last statement can be found in
\cite[1.4]{schlichtkrull-doldthom}.   
Using that the $\boxtimes$-product is defined as a Kan extension, the fact that $Q$ is strong symmetric monoidal now follows from the canonical homeomorphisms
$$
\colim_{\I}A\boxtimes B\cong \colim_{\I\times \I}A\times B\cong \colim_{\I}A\times \colim_{\I}B,
$$ 
where the second homeomorphism again is deduced from the fact that $\mathcal U$ is closed.  We define the natural transformation $U\to Q$ to be the canonical projection from the homotopy colimit to the colimit. 
\begin{lemma}\label{UQlemma}
Given $\I$-spaces $A_1,\dots,A_k$, there is a canonical homeomorphism
$$
Q(\bar A_1\boxtimes\dots\boxtimes \bar A_k)\cong UA_1\times\dots\times UA_k
$$
and the canonical projection gives a weak homotopy equivalence
$$
U(\bar A_1\boxtimes\dots\boxtimes \bar A_k)\xr{\sim}Q(\bar A_1\boxtimes\dots\boxtimes\bar A_k).
$$
\end{lemma}
\begin{proof}
Using Proposition \ref{replacementboxtimes} we may write 
$\bar A_1\boxtimes\dots\boxtimes \bar A_k$ as a homotopy Kan extension, hence the result follows immediately from Lemma  \ref{Kanextensionlemma}. 
\end{proof}
When $A$ is an $\I$-space monoid the canonical map $(\bar A)'\to \bar A$ is a homotopy equivalence of $\I$-space. The above lemma therefore implies the following result which  concludes the verification of the conditions in \textbf{A4}.
\begin{proposition}
Given $\I$-space monoids $A_1,\dots,A_k$, the canonical projection 
$$
U(A_1^c\boxtimes\dots\boxtimes A_k^c)\to Q(A_1^c\boxtimes\dots\boxtimes A_k^c)
$$
is a weak homotopy equivalence.\qed
\end{proposition}

\subsection*{Axiom \textbf{A5}}
In the notation of Axiom \textbf{A5}, we define $BF'$ and $BO'$ to be the well-based  $\I$-space monoids defined from $BF$ and $BO$ as in Proposition \ref{well-basedreplacement}. These are again  commutative flat $\I$-space monoids and the condition in \textbf{A5} therefore follows from the following more general result.
\begin{proposition}
If $A_1,\dots, A_k$ are flat $\I$-spaces, then the canonical map 
$$
U(A_1\boxtimes\dots\boxtimes A_k)\to UA_1\times\dots\times UA_k
$$
is a weak homotopy equivalence. 
\end{proposition}
\begin{proof}
Consider the commutative diagram
$$
\begin{CD}
Q(\bar A_1\boxtimes \dots\boxtimes \bar A_k)@>\simeq  >> 
Q(\bar A_1)\times \dots\times Q(\bar A_k)\\
@AA\sim A @AA\sim A\\
U(\bar A_1\boxtimes \dots\boxtimes \bar A_k)@>>>
U(\bar A_1)\times \dots\times U(\bar A_k)\\
@VV\sim V @VV\sim V\\
U(A_1\boxtimes \dots\boxtimes A_k)@>>>
U(A_1)\times \dots\times U(A_k),
\end{CD}
$$
where the vertical maps are weak homotopy equivalences by Proposition \ref{flatboxtimesequivalence} and Lemma \ref{UQlemma}. The horizontal map on the top is a homeomorphism since $Q$ is strong monoidal and the result follows.
\end{proof}

\subsection*{Axiom \textbf{A6}}
The definition of the functor in \textbf{A6} is based on the $\I$-space lifting functor 
$$
R\co\mathcal U/BF_{h\I}\to \I\mathcal U/BF,\quad(X\xr{f}BF_{h\I})\mapsto (R_fX\xr{R(f)}BF)
$$ 
introduced in \cite{schlichtkrull-thom}. We shall not repeat the details of the definition here, but we remark that a similar construction applies to give an $\I$-space lifting functor with $BO$ instead of $BF$. The following is proved in \cite{schlichtkrull-thom}.
\begin{proposition}[\cite{schlichtkrull-thom}]
The Barratt-Eccles operad $\mathcal E$ acts on $BF_{h\I}$ and if $\mathcal C$ is an operad that   
is augmented over $\mathcal E$, then there is an induced functor
$$
R\co \mathcal U[\mathcal C]/BF_{h\I}\to \I\mathcal U[\mathcal C]/BF
$$
and a natural weak equivalence of $\mathcal C$-spaces $(R_fX)_{h\I}\xr{\sim}X$.
\end{proposition} 
There is an analogous result in the $BO$ case. Now let $\mathcal C$ be the associativity operad such that the categories $\mathcal U[\mathcal C]$ and $\I\mathcal U[\mathcal C]$ are the categories of topological monoids and $\I$-space monoids, respectively. The associativity operad is augmented over the Barratt-Eccles operad and we define the functor $R$ in \textbf{A6} to be the induced functor
$$
R\co\mathcal U[\mathcal C]/BF_{h\I}\to \I\mathcal U[\bT]/BF
$$
and similarly with $BO$ instead of $BF$. The composite functor 
$$
\mathcal U[\mathcal C]/BF_{h\I}\xr{R}\I\mathcal U[\bT]/BF\to \I\mathcal U[\bT]
\xr{C}\I\mathcal U[\bT]\xr{Q}\mathcal U[\bT]\to \mathcal U[\mathcal C]
$$
takes an object $f\co X\to BF_{h\I}$ to $Q(R_fX)^c$ and we have a
chain of weak homotopy equivalences in $\mathcal U[\mathcal C]$ given
by 
\[
Q(R_fX)^c\xr{\sim}Q(\overline{R_fX})\xr{\sim}(R_fX)_{h\I}\xr{\sim}X.
\]
Here the left hand equivalence is induced by the canonical homotopy
equivalence of Proposition~\ref{well-basedreplacement}, the next is
the homeomorphism established in Lemma~\ref{UQlemma}, and the last
equivalence is provided by the preceding Proposition. The 
verification of the axioms is now complete except for the proof of
Proposition~\ref{restrictedA2}.  For this we need the following two
results from \cite{schlichtkrull-thom}.  Recall that by our
conventions a map in $\I\mathcal U$ is a weak equivalence if the
induced map of homotopy colimits is a weak homotopy equivalence.
    
\begin{proposition}[\cite{schlichtkrull-thom}]\label{RUprop}
The composite functor
$$
\I\mathcal U/BF\xr{U}\mathcal U/BF_{h\I}\xr{R} \I\mathcal U/BF
$$
is related to the identity functor on $\I\mathcal U/BF$ by a chain of natural weak equivalences.
\end{proposition}
The next proposition shows that the ordinary Thom spectrum functor can be recovered from the symmetric Thom spectrum functor up to stable equivalence.
\begin{proposition}[\cite{schlichtkrull-thom}]\label{UTRprop}
The two compositions in the diagram
$$
\begin{CD}
\I\U/BF@>\Gamma>> \I\U/BF @> T>>\Sp^{\Sigma} \\
@AA R A @. @VV U V\\
\U/BF_{h\I}@>\zeta_*>> \U/BF_{\mathcal N} @> T\Gamma >> \Sp
\end{CD}
$$
are related by a chain of natural stable equivalences. 
\end{proposition}
Since by \cite{schlichtkrull-thom} the functor $R$ takes values in the subcategory of $T$-good objects in $\I\mathcal U/BF$, the composite functor in the proposition is in fact stably equivalent to $UTR$. However, the above formulation is convenient for the application in the following proof.

\medskip
\noindent\textit{Proof of Proposition \ref{restrictedA2}.} 
It suffices to show that the two compositions in the diagram
$$
\begin{CD}
\I\U/BF@>\Gamma>> \I\U/BF @> T>>\Sp^{\Sigma} \\
@VV U V @. @VV U V\\
\U/BF_{h\I}@>\zeta_*>> \U/BF_{\mathcal N} @> T\Gamma >> \Sp
\end{CD}
$$
are related by a chain of stable equivalences.
Composing the chain of equivalences in Proposition \ref{RUprop} with the levelwise fibrant replacement functor $\Gamma$ gives a chain of natural equivalences relating the functors 
$\Gamma RU$  and $\Gamma$ on $\I\U/BF$. Therefore, applying the symmetric Thom spectrum functor to this chain, we get a chain of stable model equivalences
$$
T\Gamma R U\simeq T\Gamma\co \I\mathcal U/BF\to\Sp^{\Sigma}.
$$
Combining this with Proposition \ref{UTRprop} ,we finally get the required chain of stable equivalences
$$
UT\Gamma\simeq UT\Gamma R U\simeq T\Gamma\zeta_* U.
\eqno\qed
$$

\subsection{The proof of Theorem \ref{maintheorem} in the general case}
\label{generalsymmetric}
In this section we show that the weaker forms of the axioms \textbf{A2} and \textbf{A3} suffice for the proof of Theorem \ref{maintheorem}. The main technical difficulty is that the symmetric Thom spectrum functor only preserves weak equivalences on the full subcategory of $T$-good objects in 
$\I\mathcal U/BF$. In order to maintain homotopical control we must
therefore be careful only to apply the Thom spectrum functor to
$T$-good objects. We say that a map $\alpha_n\co A(n)\to
BF(n)$ \emph{classifies a well-based quasifibration} if 
\begin{enumerate}
\item the pullback $\alpha^*EF(n)\to A(n)$ is a quasifibration, and 
\item the induced section $A(n)\to \alpha^*EF(n)$ is an $h$-cofibration.
\end{enumerate}
An object $\alpha$ in $\I\mathcal U/BF$ is said to classify well-based
quasifibrations if the level maps $\alpha_n$ do. 
This condition implies that $\alpha$ is $T$-good and is sometimes technically convenient as in the following lemma from \cite{schlichtkrull-thom}.  
\begin{lemma}[\cite{schlichtkrull-thom}]\label{hocolimlemma}
Let $\Lambda$ be a small category and let $f_\lambda\co X_\lambda\to BF(n)$ be a $\Lambda$-diagram in $\mathcal U/BF(n)$ such that each $f_\lambda$ classifies a well-based quasifibration. Then the induced map
$$
f\co\hocolim_{\Lambda}X_\lambda\to BF(n)
$$ 
also classifies a well-based quasifibration.
\end{lemma}
Now let $f\co X\to BF_{h\I}$ be a map of topological monoids and let
 $\alpha\co A\to BF$ be the object in $\I\mathcal U[\bT]/BF$ obtained by applying the functor 
 $R$. The first step is to ensure that the symmetric Thom spectra $T(\alpha)$ and $T(\alpha^c)$ have the correct homotopy types. 
 
\begin{lemma}
The objects $\alpha\co A\to BF$ and $\alpha^c\co A^c\to BF$ classify well-based quasifibrations. 
\end{lemma} 
\begin{proof}
It follows from the definition of the functor $R$ in \cite{schlichtkrull-thom} that the level maps 
$\alpha_n$ classify well-based quasifibrations and that the same holds for the composite maps
\[
A(m)\xr{\alpha_m}BF(m)\xr{\gamma}BF(n)
\]
for each morphism $\gamma\co \mathbf m\to\mathbf n$ in 
$\I$. Thus, $\alpha$ classifies well-based quasifibrations and by Lemma \ref{hocolimlemma} the same holds for $\bar \alpha$. Since $A^c$ is the homotopy colimit of the diagram $*\to \bar A$, the result follow.
\end{proof} 
It now follows from Lemma \ref{flatcompatibility} that $T(\alpha^c)$
is a flat symmetric ring spectrum and therefore that the cyclic bar
construction $B^{\cy}(T(\alpha^c))$ represents the topological
Hochschild homology of $T(\alpha)$.  It remains to analyze the map
$B^{\cy}(\alpha^c)$.  
 
\begin{lemma} \label{Bcyalpha}
The object $B^{\cy}(\alpha^c)\co B^{\cy}(A^c)\to BF$ 
in $\I\mathcal U/BF$ is $T$-good. 
\end{lemma} 

In preparation for the proof, consider in general a simplicial object
$f_{\bullet}\co X_{\bullet}\to BF(n)$ in $\mathcal U/BF(n)$ with topological realization $f\co X\to BF(n)$. Evaluating the Thom spaces degree-wise we get a simplicial based space $T(f_{\bullet})$ whose realization is isomorphic to $T(f)$ as follows from Lemma \ref{realizationlemma}. 

\begin{lemma}
Suppose that $f_{\bullet}\co X_{\bullet}\to BF(n)$ is degree-wise $T$-good and that the simplicial spaces $X_{\bullet}$ and $T(f_{\bullet})$ are good. Then the realization $f\co X\to BF(n)$ is also $T$-good.  
\end{lemma}
\begin{proof}
We must show that the canonical map $X\to \Gamma_f(X)$ induces a weak homotopy equivalence of Thom spaces $T(f)\to T\Gamma(f)$. Using that $\Gamma$ preserves tensors and colimits, we identify the latter map with the realization of the simplicial map $T(f_{\bullet})\to T\Gamma(f_{\bullet})$. The assumption that $f$ be degree-wise $T$-good implies that this is a degree-wise weak homotopy equivalence.  Since the simplicial space $T(f_{\bullet})$ is good by assumption, it remains to show that the goodness assumption on $X_{\bullet}$ implies that $T\Gamma(f_{\bullet})$ is also good. For this we use \cite[IX.1.11]{lewis-may-steinberger}, which implies that the degeneracy maps $\Gamma(X_k)\to \Gamma(X_{k+1})$ are fibrewise 
$h$-cofibrations over $BF(n)$. The induced maps of Thom spaces are therefore also 
$h$-cofibrations. The conclusion now follows from the fact that a levelwise weak equivalence of good simplicial spaces induces a weak equivalence after topological realization.   
\end{proof}

\bigskip
\noindent\textit{Proof of Lemma \ref{Bcyalpha}.}
Notice first that $B^{\cy}_{\bullet}(A)$ and $T(B^{\cy}_{\bullet}(\alpha^c))$ are good simplicial objects since $A^c$ and $T(\alpha^c)$ are well-based . 
 By the lemma just proved it therefore suffices to show that the simplicial map 
 $B^{\cy}_{\bullet}(A^c)(n)\to BF(n)$
 is degree-wise $T$-good for each $n$. In simplicial degree $k$ this is the composition
$$
A^c\boxtimes\dots\boxtimes A^c(n)\to \bar A\boxtimes\dots\boxtimes\bar A(n)\to BF(n) 
$$   
where the first map is a fibrewise homotopy equivalence over $BF(n)$. It therefore suffices to show that the second map is $T$-good and using Proposition \ref{replacementboxtimes} we write the latter as a map of homotopy colimits
$$
\hocolim_{\mathbf n_0\sqcup\dots\sqcup \mathbf n_k\to \mathbf n}
A(n_0)\times\dots\times A(n_k)\to BF(n).
$$  
The maps in the underlying diagram classify well-based quasifibrations by construction and the result now follows from Lemma \ref{hocolimlemma}. \qed

\medskip
\noindent\textit{Proof of Theorem \ref{maintheorem}.}
Let $\mathcal C$ be the associativity operad. As explained in Appendix \ref{loopappendix} we may assume that our loop map is a map of $\mathcal C$-spaces $f\co X\to BF_{h\I}$ and we again write 
$\alpha\co A\to BF$ for the associated monoid morphism. It follows from the above discussion that the topological Hochschild homology of $T(f)$ is represented by $B^{\cy}(T(\alpha^c))$ which in turn is isomorphic to the symmetric Thom spectrum $T(B^{\cy}(\alpha^c))$. Since $B^{\cy}(\alpha^c)$ is $T$-good, the weaker version of \textbf{A2} suffices to give a stable equivalence
$$
UT(B^{\cy}(\alpha^c))\simeq T\Gamma\zeta_*(UB^{\cy}(\alpha^c)).
$$
From here the argument proceeds as in Section \ref{mainproofsection} and we get a stable equivalence
$$
UT(B^{\cy}(\alpha^c))\simeq T\Gamma(L^{\eta}(Bf))
$$
which is the content of Theorem \ref{maintheorem}.\qed

\medskip
Proceeding as in Section \ref{mainproofsection} one finally deduces
the general case of Theorem \ref{2foldtheorem} and Theorem
\ref{3foldtheorem} from this result.

\appendix
\section{Loop maps and $A_{\infty}$ maps}\label{loopappendix}
In order to prove our main results we need to pass from loop map data to the more rigid kind of data specified by a map of $A_{\infty}$ spaces. This can be done using the machinery from 
\cite{may-geom} as we now explain. Let $\mathcal C_1$ be the little 1-cubes operad. Given an 
$A_{\infty}$ operad $\mathcal C$ we let $\mathcal D$ be the fibred product 
$\mathcal C\triangledown \mathcal C_1$ such that there is a diagram of $A_{\infty}$ operads 
$\mathcal C\leftarrow \mathcal D\to\mathcal C_1$. Let $C$, $C_1$ and $D$ be the associated monads on the category $\mathcal T$ of based spaces. With the notation from 
\cite[13.1]{may-geom},  we have for each 
$\mathcal C$-space $X$ a diagram of $\mathcal D$-spaces 
$$
X\xl{\epsilon} B(D,D,X)\xr{\gamma}\Omega B(\Sigma, D,X)
$$    
where $\epsilon$ is a homotopy equivalence. We write $B_1X$ for the space $B(\Sigma,D,X)$. If $X$ is grouplike,  then this is a delooping in the sense that the map $\gamma$ also is a weak equivalence, see \cite{may-geom}, \cite{segal-conf}, \cite{Thomason}. 
It is known by \cite{Thomason} that the functor 
$X\mapsto B_1X$ is equivalent to the functor obtained by first replacing $X$ by a topological monoid and then applying the usual classifying space construction (this is the functor denoted $B'X$ in Section \ref{mainproofsection}). Here and in the following we tacitly assume that all base points are non-degenerate. This is not a serious restriction since by \cite{may-geom} any $A_{\infty}$ space may be functorially replaced by one with a non-degenerate base point. 
The following result shows that we can always rectify loop maps over a grouplike $A_{\infty}$ space to $A_{\infty}$ maps.

\begin{proposition}\label{rectifyAinfty}
Let $\mathcal C$ be an $A_{\infty}$ operad and let $Z$ be a grouplike $\mathcal C$-space. Then there exists a ``loop functor''
$$
\bar\Omega\co \mathcal T/B_1Z\to \mathcal U[\mathcal C]/Z,\quad (Y\xr{g}B_1Z)\mapsto 
(\bar\Omega_g(Y)\xr{\bar\Omega(g)}Z),
$$ 
and a chain of weak equivalences of $\mathcal D$-spaces 
$\bar\Omega_g(Y)\simeq\Omega(Y)$ such that the diagram
$$
\begin{CD}
\bar\Omega_g(Y)@>\bar \Omega(g)>> Z\\
@VV\sim V @VV \sim V\\
\Omega(Y)@>\Omega(g)>> \Omega(B_1Z)
\end{CD}
$$
is commutative in the homotopy category.
\end{proposition}   
\begin{proof}
Given a based map $g\co Y\to B_1Z$, let 
$\Omega'_g(Y)$ be the homotopy pullback of the diagram of $\mathcal D$-spaces
$$
\Omega(Y)\xr{\Omega(g)}\Omega(B_1Z)\xl{\gamma} B(D,D,Z).
$$ 
Then $\Omega_g'(Y)$ is a $\mathcal D$-space and we have an induced map of 
$\mathcal D$-spaces
$$
\Omega'(g)\co \Omega'_g(Y)\to B(D,D,Z)\xr{\epsilon} Z.
$$
Consider now the functor $B(C,D,-)$ from $\mathcal D$-spaces to $\mathcal C$-spaces and notice that there are natural equivalences of $\mathcal D$-spaces
$$
X\xl{\epsilon} B(D,D,X)\to B(C,D,X)
$$  for any $\mathcal D$-space $X$. We define $\bar\Omega(g)$ to be the map of $\mathcal C$-spaces induced by $\Omega'(g)$,
$$
\bar\Omega(g)\co \bar\Omega_g(Y)=B(C,D,\Omega'_g(Y))\to B(C,D,Z)\to Z.
$$
The last map is defined since $Z$ is already a $\mathcal C$-space. It is clear from the construction that $\bar\Omega(g)$ is related to $\Omega(g)$ by a homotopy commutative diagram as stated. 
\end{proof}

\begin{proposition}
Let $\bar \Omega$ be the functor from Proposition \ref{rectifyAinfty} and suppose that $Z$ is connected. Restricted to $g\co Y\to B_1Z$ with $Y$ connected, there is a chain of weak equivalences $Y\simeq B_1\bar\Omega_g(Y)$ such that the diagram
$$
\xymatrix{
Y\ar[rr]^-{\sim}\ar[dr]_{g} & & B_1\bar\Omega_g(Y)\ar[dl]^{B_1\bar\Omega(g)}\\
& B_1Z & }
$$
is commutative in the homotopy category. 
\end{proposition}
\noindent\textit{Proof.}
Given a based space $Y$ there is a canonical map $\rho\co B_1\Omega(Y)\to Y$ which by 
\cite[13.1]{may-geom} and \cite{Thomason} is a weak equivalence if $Y$ is connected. Applying the functor $B_1$ to the diagram of $\mathcal D$-spaces in Proposition 
\ref{rectifyAinfty} we therefore get a homotopy commutative diagram
$$
\begin{CD}
B_1\bar\Omega_g(Y)@>B_1\bar\Omega(g)>> B_1Z\\
@VV\simeq V @VV\simeq V \\
Y@>g>>B_1Z
\end{CD}
$$ 
where the vertical arrow on the right represents the chain of maps
$$
B_1Z\xl{B(\Sigma,D,\epsilon)} B(\Sigma,D,B(D,D,Z))\xr{B(\Sigma,D,\gamma)} 
B(\Sigma,D,\Omega(B_1Z))\xr{\rho}B_1Z.
$$
It remain to show that this represents the identity on $B_1Z$. For this we observe that 
$B(\Sigma,D,B(D,D,Z))$ is homeomorphic to the realization of the bisimplicial space with $(p,q)$-simplices $\Sigma D^pDD^qZ$ and that the two maps correspond respectively to multiplication in the $p$ and $q$ direction. We may also view this space as the realization of the diagonal simplicial space $\Sigma D^pDD^pZ$ and the result now follows from the explicit homotopy
$$
H\co B(\Sigma,D,B(D,D,Z))\times I\to B_1Z,\quad H([b,u],t)=[b,tu,(1-t)u].
$$
Here $I$ denotes the unit interval, $b$ is an element in $\Sigma D^pDD^pZ$, and $u$ is an element in the standard $p$-simplex
$$
\Delta^p=\{(u_0,\dots,u_p)\in I^{p+1}: u_0+\dots+ u_p=1\}.
\eqno\qed
$$

For completeness we finally show that one may rectify $n$-fold loop maps over a grouplike 
$E_{\infty}$ space to maps of $E_n$-spaces for all $n$. Let $\mathcal C_n$ be the little $n$-cubes operad. Given an $E_{\infty}$ operad $\mathcal C$ we let $\mathcal D_n$ be the product operad $\mathcal C\times \mathcal C_n$. This is an $E_n$ operad in the sense that that the projection $\mathcal D_n\to\mathcal C_n$ is an equivalence. We again write $C$, $C_n$ and $D_n$ for the associated monads. With the notation from \cite[13.1]{may-geom}, we have a diagram of $D_n$-maps
$$
X\xl{\epsilon}B(D_n,D_n,X)\xr{\gamma}\Omega^n B(\Sigma^n,D_n,X)
$$
and we write $B_nX$ for the space $B(\Sigma^n,D_n,X)$. When $X$ is grouplike this is an $n$-fold delooping in the sense that $\gamma$ is a weak equivalences. An argument similar to that used in the proof of Proposition \ref{rectifyAinfty} then gives the following.

\begin{proposition}
Let $\mathcal C$ be an $E_{\infty}$ operad and let $Z$ be a grouplike $\mathcal C$-space. Then there exists an ``$n$-fold loop functor''
$$
\bar\Omega^n\co \mathcal T/B_nZ\to\mathcal U[\mathcal D_n]/Z,\quad (Y\xr{g}B_nZ)\mapsto
(\bar\Omega^n_g(Y)\xr{\bar\Omega^n(g)} Z), 
$$
and a chain of weak equivalences of $\mathcal D_n$-spaces $\bar\Omega_g^n(Y)\simeq \Omega^n(Y)$ such that the diagram
$$
\begin{CD}
\bar\Omega_g^n(Y)@>\bar\Omega^n(g)>> Z\\
@VV\sim V @VV\sim V\\
\Omega^n(Y)@>\Omega^n(g)>> \Omega^n(B_nZ)
\end{CD}
$$
is commutative in the homotopy category.\qed
\end{proposition}



\end{document}